\documentclass[12pt,a4paper]{article}
\usepackage{amsfonts}
\usepackage{mathrsfs}
\usepackage[usenames,dvipsnames]{xcolor}
\usepackage{stmaryrd}
\usepackage{amsmath}
\usepackage{amssymb}
\usepackage{fouridx}
\usepackage{soul}
\usepackage{paralist}
\usepackage{dsfont}
\usepackage{mathtools}

\usepackage{amssymb,amsmath}
\usepackage{amsbsy}
\usepackage{amsopn}
\usepackage{amsthm}
\usepackage{amscd}
\usepackage{appendix}
\def\divv{{\rm div}\,}

\newcommand{\duality}[1]{\fourIdx{}{\rV^\prime}{}{\rV}{\Big\langle #1 \Big\rangle }}
\newcommand{\dual}[1]{\fourIdx{}{\rV^\prime}{}{\rV}{#1}}
\newcommand{\UT}{\Upsilon_T(\rV)}
\newcommand{\MT}{M_T}
\newcommand{\MTb}{\mathbb{M}_T} 
\newcommand{\Q}{\mathbb{Q}}
\newcommand{\G}{\mathcal{G}}
\newcommand{\Gb}{\mathbb{G}}
\newcommand{\Pc}{\mathcal{P}}
\newcommand{\Nbar}{\mathrm{N}}
\newcommand{\mbar}{\mathrm{m}}
\newcommand{\psij}{\psi} 
\newcommand{\zsmall}{J}
\newcommand{\zbig}{\mathbf{J}}

\newcommand{\usmall}{u}
\newcommand{\ysmall}{y}
\newcommand{\xsmall}{x}
\newcommand{\Xbig}{\xi}
\newcommand{\Zbig}{v}
\newcommand{\phieps}{\varphi_{\eps}}
\newcommand{\phiepsold}{\varphi_{\eps}}

\allowdisplaybreaks

\usepackage{bbm}
\usepackage{mathrsfs}
\usepackage{graphicx}

\usepackage{enumerate}

\makeatletter
\def\tank#1{\protected@xdef\@thanks{\@thanks
    \protect\footnotetext[0]{#1}}}
\def\bigfoot{

  \@footnotetext}
\makeatother

\newcommand{\ea}{\end{array}}

\numberwithin{equation}{section}

\newtheorem{theorem}{Theorem}[section]
\newtheorem{thm}{Theorem}[section]
\newtheorem{proposition}{Proposition}[section]
\newtheorem{corollary}{Corollary}[section]
\newtheorem{lemma}{Lemma}[section]
\newtheorem{definition}{Definition}[section]
\newtheorem{remark}{Remark}[section]
\newtheorem{con}{Assumption}[section]

\renewcommand{\theequation}{\arabic{section}.\arabic{equation}}

\newcommand{\eps}{\ensuremath{\varepsilon}}

\newcommand{\rZ}{\mathrm{Z}}
\newcommand{\rV}{\mathrm{V}}
\newcommand{\rH}{\mathrm{H}}
\newcommand{\rA}{\mathrm{A}}
\newcommand{\rB}{\mathrm{B}}

\newcommand{\rY}{\mathrm{Y}}

\usepackage[normalem]{ulem}

\newcommand{\Leb}{\mathrm{Leb}}

\newcommand{\embed}{\hookrightarrow}
\newcommand\adda[1]{{\color{blue} #1}}

\newcommand\old[1]{}
\newcommand\addjz[1]{\textcolor[rgb]{0.50,0.00,1.00} {#1}}
\newcommand\addjzok[1]{{#1}}
\newcommand\addaok[1]{{#1}}
\newcommand\newok[1]{{#1}}

\usepackage{todonotes}

\newcommand{\todozbdel}[1]{}

\newcommand\dela[1]{}
\newcommand\deljz[1]{}
\newcommand\deln[1]{}

\usepackage{mathrsfs}
\let\mathcal \undefined
\def\mathcal{\mathscr}

\newcommand\toup{\nearrow}
\newcommand\tof{\rightarrow}
\newcommand{\Cphi}{C_1}
\newcommand{\xit}{L^2([0,t];\rV)}
\newcommand{\xis}{L^2([0,s];\rV)}
\newcommand{\DS}{\displaystyle}

\begin{document}
\title{\Large\bf Well-posedness and large deviations for 2D stochastic Navier-Stokes equations with jumps
\thanks{Zdzis{\l}aw Brze{\'z}niak was partially supported by the National Science Foundation of China (No. 11501509) and by a Royal Society grant ``Stochastic Landau-Lifshitz-Gilbert equation with L\'evy noise and ferromagnetism".
Xuhui Peng is supported by National Natural Science Foundation of China(NSFC)(No.12071123), the Scientific Research Fund of Hunan Provincial Education Department (No. 20A329) and  the Construct Program of the Key Discipline in Hunan Province.
  Jianliang Zhai is supported by the National Natural Science Foundation of China (Nos. 11971456, 11721101, 11671372, 11431014).
  }}

\author{
   {Zdzis{\l}aw Brze{\'z}niak}$^a$\footnote{E-mail: zdzislaw.brzezniak@york.ac.uk}\;
\ \ \ {Xuhui Peng}$^b$\footnote{E-mail: xhpeng@hunnu.edu.cn}\;
\ \ \
{Jianliang Zhai}$^c$\footnote{E-mail: zhaijl@ustc.edu.cn
}
\\
 \small a. Department of Mathematics, The University of York, Heslington, York, YO10 5DD, UK  \\
 \small b. MOE-LCSM, School of Mathematics and Statistics,\\
  \small Hunan Normal University, Changsha, Hunan 410081, P. R. China\\
\small c. Key Laboratory of Wu Wen-Tsun Mathematics CAS, School of Mathematical Science,\\
 \small University of Science and Technology of China, Hefei, Anhui 230026, P.R. China
 }\,
\date{29 September, 2021}
\maketitle
\tableofcontents

\begin{center}
\begin{minipage}{160mm}
{\bf Abstract.}
\newok{The aim of this paper is threefold.
Firstly, we prove the existence and the uniqueness of a global strong (in both the probabilistic and the PDE senses) $\mathrm{H}^{1}_2$-valued solution to the 2D stochastic Navier-Stokes equations (SNSEs) driven by a multiplicative L\'evy noise under the natural Lipschitz on balls and 
linear growth assumptions on the jump coefficient.
Secondly, we prove a Girsanov-type  theorem for Poisson random measures and apply this result to a study of the well-posedness of the corresponding stochastic controlled problem for these SNSEs.
Thirdly, we apply these results to establish a  Freidlin-Wentzell-type large deviation principle for the solutions of these  SNSEs by employing the weak convergence method introduced in papers \cite{Budhiraja-Chen-Dupuis, Budhiraja-Dupuis-Maroulas}.}

\vspace{3mm} {\bf Keywords:} 2D stochastic Navier-Stokes equations, L\'evy processes, the Girsanov Theorem, strong solutions in the probabilistic
and PDE senses, Freidlin-Wentzell-type large deviation principle.
\vspace{3mm}


\end{minipage}
\end{center}

\section{Introduction}\label{sec-Introduction}

\newok{In this paper, we concentrate on  stochastic
Navier-Stokes equations (SNSEs), but we believe that our results can be generalized to other types of stochastic
 partial differential equations (SPDEs). Since the seminal work \cite{BT} by Bensoussan and Temam, a great number of papers have been written on the subject of SNSEs driven by Gaussian noise. The questions of the existence and uniqueness of solutions to such equations have been investigated in many papers; see, for example, \cite{Brzezniak+Motyl_2013_NSE}, \cite{DZ}, \cite{Flandoli+Gatarek_1995}, \cite{GHZ2009}, \cite{Liu R}, and \cite{MR}. The Freidlin-Wentzell-type large deviation principle for 2D SNSEs has been proved in \cite{CM} and \cite{SS}. In a recent paper \cite{WZZ}, the authors established the moderate deviation principle for these equations. The ergodic properties of invariant measures of the Markov semigroups generated by SNSEs (in the Gaussian case) and related questions have been studied in papers such as \cite{Brz+Cerrai_2015}, \cite{Flandoli},  and \cite{HM}.}

\vskip 0.1cm

\newok{However, some real-world models of financial, economic, physical, chemical, and biological phenomena cannot
be described well by Gaussian noise. For example, in some circumstances, some large moves and unpredictable events
can be captured by jump type noises. In recent years, SPDEs driven by jump type L\'evy noise  have become extremely popular in modelling these phenomena.}
\vskip 0.1cm

\newok{Many researchers put much  effort into understanding various properties of SPDEs driven by general L\'evy noise. Compared with the Gaussian case, SPDEs driven by pure jump L\'evy process  behave in a drastically different manner because of the appearance of jumps. Examples of such are provided by
\begin{inparaenum}\item[(i)] the Burkholder-Davis-Gundy inequality -- see, for instance, \cite{Ichikawa 1986} and \cite{ZBL};
\item[(ii)] the Girsanov Theorem -- see, for example,
\cite[Theorem III.3.24]{Jacod-Shiryaev} for semimartingales,
\cite[Theorem 3.10.21]{BICHTELER} for Poisson point processes,
\cite[Theorem 3.9.19]{BICHTELER} for finite-dimension Wiener processes,
and
\cite[Appendix A.1.]{DFPR 2013} for infinite dimensions with respect to a cylindrical
Wiener process;
\item[(iii)] the time regularity of solutions -- see, for instance, \cite{IMMTZ} for OU processes driven by a cylindrical Wiener process, \cite{BR} on the uniform convergence of random series in Skorohod space and representations of c{\`a}dl{\`a}g infinitely divisible processes, and
\cite{BGIPPZ,Liu Zhai 2012,Liu Zhai 2016,PZ,Priola Zabczyk} for OU processes driven by cylindrical pure jump processes;
\item[(iv)] ergodicity -- see, for example, \cite{BHR 2016,FHR, Priola Zabczyk} for the pure jump case;
\item[(v)] irreducibility -- see, for example,  \cite{DWX 2020,FHR,HP 2015, Priola Zabczyk,WXX 2017} for the pure jump case;
 and
\item[(vi)]
 other long-time properties of the solutions to SPDEs driven by jump processes -- see \cite{CK 2019, CK 2020}.
 \end{inparaenum}

In general, the methods and techniques available for SPDEs driven by Gaussian noise are not suitable for investigating  SPDEs driven by jump type noise, and therefore, new and sophisticated tools are needed. We refer to the  above-mentioned references and references therein for more details.
}

\newok{As an example let us  consider SNSEs. Under the classical Lipschitz on balls and  linear growth assumptions on the noise coefficients, one can prove the existence and uniqueness of a strong solution in both the probabilistic and PDE senses for 2D SNSEs driven by Gaussian noise; see, for example, \cite{GHZ2009,Liu R,Mikulevicius_2009}. However, in order to prove a similar result for  the pure jump case,  {in the existing literature} one is required to
introduce additional conditions on the jump coefficient $G$; see Remark \ref{Rem1} in Section \ref{subsection 2.1}. The reason for this difference is that the proof of the existence of solutions relies on the use of the Burkholder-Davis-Gundy inequality for the compensated Poisson random measures for the exponent $p\neq2$ (see \cite{BHZ}, \cite{BLZ}, and \cite{ZBL}). Similar problems arise in the study of  martingale solutions; see, for instance, \cite{Dong-Zhai,Flandoli+Gatarek_1995,EM_2014} and many more recent papers.}

\newok{
A natural approach to proving the well-posedness of SPDEs driven by jump type L\'evy process  is to approximate the Poisson random measure $N$ by a sequence of Poisson random measures $N_n$ with finite intensity measures. Dong and Xie  used this way in \cite{Dong+Xie_2009} to establish the well-posedness of the strong solutions in the probabilistic sense for 2D SNSEs driven by L\'evy noise.
However, to make  this method work, one needs to impose additional assumptions to control the ``small jumps''; see Remark \ref{Rem1} in Section \ref{subsection 2.1}.
Another basic idea used to prove the well-posedness of SPDEs is based on introducing an appropriate cut-off and then applying the Banach fixed point theorem for the approximated problems. This method has been exploited in a recent paper \cite{BHR15} to establish
the existence of strong solutions in the PDE sense for 2D SNSEs driven by L\'evy processes of jump type. However, because they relied on  the Burkholder-Davis-Gundy inequality with exponent $p\neq2$ for the compensated Poisson random measure, in addition to the natural Lipschitz and linear growth assumptions, the authors of \cite{BHR15} had to impose additional and unnatural assumptions on the noise coefficient $G$; see Remark \ref{Rem2} in Section \ref{subsection 2.2}.}
\vskip 0.1cm

\newok{The first aim of this paper is to remove these unnecessary assumptions imposed in   \cite{BHR15} and other papers. For this purpose we employ different ideas and techniques. We use the cut-off approximation method and the Banach fixed point theorem, used recently by the first author and Millet in \cite{Brzezniak Millet} and later in \cite{BHR15} to prove the existence of strong solutions in the PDE sense to a class of stochastic hydrodynamical systems driven by L\'evy process. Earlier, a similar idea had been used by De Bouard and Debussche \cite{Bouard Debussche01,Bouard Debussche02}.
However, our auxiliary equations are different from the equations introduced in the papers cited above.

Using these auxiliary equations, we are able to remove  the  atypical assumptions described above. Our method strongly depends on the cut-off function $\theta_m$ introduced in  \eqref{eqn-theta_m}. To achieve our goals,
it is crucial to establish new \emph{a priori} estimates.
We believe that this method can also be used for other systems driven by L\'evy noise to weaken the assumptions and, in particular, eliminate  those that are not necessary.
}
\vskip 0.2cm

\newok{Our second, and in fact, the main aim of this paper is to establish a Freidlin-Wentzell-type large deviation principle (LDP) for the strong solutions in the PDE sense (obtained in the first part) of 2D SNSEs driven by L\'evy processes of jump type.
}

{Large deviation principle for finite dimensional stochastic differential equations (SDEs) with a Poisson
noise term has been studied by several authors; see \cite{A1,Bao-Yuan,Budhiraja-Dupuis-Ganguly}. There is not much study on the topic of the LDP for infinite dimensional SDEs driven by L\'evy processes of jump type.
The first paper was \cite{Rockner-Zhang} by $\rm R\ddot{o}ckner$ and Zhang, where stochastic evolution equations with additive noise was considered. The case of multiplicative noise was studied in \cite{Budhiraja-Chen-Dupuis,Budhiraja-Dupuis-Ganguly,Swiech-Zabczyk, Yang-Zhai-Zhang}. The study of the LDP for SPDEs with highly nonlinear terms
 has been carried out as well, see, e.g., \cite{Dong-Xiong-Zhai-Zhang, Xiong Zhai, Xu-Zhang,Zhai-Zhang,Zheng Zhai Zhang}.
 Concerning the  2D SNSEs, it is important to mention
  that Xu and Zhang \cite{Xu-Zhang} studied the LDP for these equations driven by additive L\'evy noise, while the recent papers by the third author  and collaborators deal with such SNSEs driven by multiplicative L\'evy noise; see  \cite{Dong-Xiong-Zhai-Zhang}, \cite{Xiong Zhai}, and \cite{Zhai-Zhang}. In all these results, the authors consider strong solutions in the probabilistic sense.
}

To prove our results, we use the weak convergence approach introduced by Budhiraja, Dupuis, and others in \cite{Budhiraja-Chen-Dupuis,Budhiraja-Dupuis-Maroulas} for the case of Poisson random measures, which has been proved to be very effective to study the LDP for finite/infinite dimensional SDEs driven by L\'evy processes; see \cite{Bao-Yuan,Budhiraja-Chen-Dupuis,Budhiraja-Dupuis-Ganguly,Dong-Xiong-Zhai-Zhang,Xiong Zhai, Yang-Zhai-Zhang,Zhai-Zhang,Zheng Zhai Zhang}.  Compared with the existing results, our main object  are  the strong solutions in the PDE sense, and hence we need to find new \emph{a priori} estimates to establish the tightness of the solutions of the perturbed equations; see Lemmata \ref{lem3 LDP deter 1}, \ref{lem LDP stochastic 01}--\ref{lem LDP stochastic 05}. We believe that this is nontrivial.

{Finally, let us mention the third aim of our article: the well-posedness of the controlled SPDEs (\ref{eq3 prop2 00}). Such a result is
a basic step in applying the weak convergence approach. During our study, it became apparent that although such a result has been used in previous literature, e.g., \cite{Budhiraja-Chen-Dupuis,Budhiraja-Dupuis-Ganguly, Dong-Xiong-Zhai-Zhang, Xiong Zhai, Yang-Zhai-Zhang,Zhai-Zhang}, it has never been rigorously formulated nor proven.  Thus, we fill this gap by formulating Lemma \ref{lem LDP stochastic 00} and providing a rigorous proof of it. The proof  of Lemma \ref{lem LDP stochastic 00} heavily depends on a Girsanov-type theorem for Poisson random measures, i.e., Theorem \ref{thm-Girsanov}. Although, this is apparently a ``standard result", see  for instance
\cite[Theorem III.3.24]{Jacod-Shiryaev} for the semimartingales and
\cite[Theorem 3.10.21]{BICHTELER} for the Poisson point processes,  it seems hard to find an accessible reference in the literature which would work under our conditions. Therefore in Section \ref{sec-Girsanov} we include a complete proof of the version of the Girsanov-type theorem  we need.
The Girsanov Theorem for the Wiener process states that the shifted and the original Wiener measures are equivalent  if and only if the shift function belongs to the  corresponding Cameron-Martin space. However, in contrast to the Wiener space case, the Girsanov Theorem for Poisson random measures is related to invertible and predictable nonlinear
transformations, see \cite[Theorem III.3.24]{Jacod-Shiryaev} and  \cite[Theorem 3.10.21]{BICHTELER}. These differences
lead to many difficulties in proving the variational representation for the Poisson functionals, and therefore  in applying
the weak convergence method for the case of Poisson random measures and the Freidlin-Wentzell-type LDP for SPDEs driven by L\'evy processes of jump type; see, for example, \cite{Budhiraja-Chen-Dupuis,Budhiraja-Dupuis-Ganguly, Budhiraja-Dupuis-Maroulas,DE,Zhang XC}. This is also one of the main difficulties
this paper had to deal with. Let us mention that another application of the Girsanov Theorem is in its use, in combination with the Yamada-Watanabe Theorem,  in proving the well-posedness of  SPDEs. For instance, see \cite{Gat+Goldys_1992}  for the case of SPDEs driven by Wiener process. However, for applications of the Girsanov Theorem in framework of SPDEs defined in terms of  Poisson random measures, the literature contains only few results; see for instance, the paper \cite{HP 2015} in which, however, no proofs are provided.
}

\vskip 0.2cm

\deln{\newok{The approach of this paper is currently being applied by the first and third authors, joined by Utpal Manna, to study the LDP of stochastic ferromagnetic wires, i.e., solutions of 1D stochastic Landau-Lifshitz equations; see \cite{Brz+Manna+Zhai_2018}. In other direction, the first and third authors, together with Martin Ondrej\'at, are preparing a publication about the LDP for stochastic PDEs driven by infinite-dimensional L\'evy processes; see \cite{Brz+Ondrejat+Zhai_2018}.}
\vskip 0.2cm}

\newok{The organization of this paper is as follows. Section \ref{sec-SNSEs 2} is to introduce 2D SNSEs.
In Sections \ref{subsection 2.1} and \ref{subsection 2.2}, we apply a cut-off and the Banach fixed point theorem to establish the existence and uniqueness of strong (in the probabilistic sense and PDE sense, respectively) solutions for 2D SNSEs with L\'evy noise, under the Lipschitz on balls and linear growth assumptions. We do this for initial data from the space $\rH$
(see Theorems \ref{thm-initial data in H} and \ref{thm-initial data in H local})
and for initial data from the space $\rV$
(see Theorems \ref{thm-initial data in V} and \ref{thm-initial data in V local}).
Section \ref{sec-LDP} is devoted to the formulation of the LDP, see
Theorem \ref{thm-LDP}. This section  contains also a proof Theorem \ref{thm-LDP}  provided some auxilairy results hold true. Moreover, we prove the first one of
auxilairy results, the so called  first continuity lemma, i.e., Proposition \ref{prop-1st continuity lemma}. The remaining  auxiliary results are proven in the following sections. Thus,
Section \ref{sec-Girsanov} contains  a formulation and a proof of a
Girsanov-type theorem for Poisson random measures; see Theorem \ref{thm-Girsanov}. The last section \ref{sec-LDP-verification} is devoted to  a proof of the
 second continuity lemma, i.e., Proposition \ref{prop3 02}. The paper also contains also two appendices. Appendix \ref{sec-A-PRM} contains necessary definitions related to Poission random measures.
Appendix \ref{sec-B} is devoted to the last auxiliary result, i.e.,  Lemma \ref{lem3 LDP deter 1}.
}

\textbf{Acknowledgement.}
The first named author would like to thank Utpal Manna for discussions. The third named author acknowledges funding by
K.C. Wong Foundation for a year-long fellowship at King's College London, during which this work was started. He also
would like to thank Markus Riedle for his kind help in London. All authors would like to thank Alexei Daletski, Carl Chalk, John Herman and Tomasz Zastawniak, for careful reading of the manuscript and many useful comments.

\section{The stochastic Navier-Stokes equations (SNSEs)}\label{sec-SNSEs 2}

 We assume that $D$ \addjzok{is }a bounded open domain in $\mathbb{R}^2$, with smooth boundary $\partial D$. Let us define the following fundamental functional space:
\begin{eqnarray*}
\rV=\{u\in W^{1,2}_0(D,\mathbb{R}^2):\,\divv\, u=0 \mbox{ {weakly} in}\ D\},\ \ \ \ \|u\|^2_{\rV}:=\int_D|\nabla u(x)|^2\,dx.
\end{eqnarray*}
Let  $\rH$ be the closure of $\rV$ in the space $\mathbb{L}^2(D):=L^{2}(D,\mathbb{R}^2)$. The space $\rH$ is a separable Hilbert space endowed with the following norm
$$
|u|^2_{\rH}:=\int_D|u(x)|^2\,dx.
$$

Let $\Pi:\, \mathbb{L}^2(D)\to \rH$ is the orthogonal projection, which is called the Leray-Helmholtz projection. Let us define the Stokes operator $\rA$ in $\rH$ by
\begin{eqnarray*}
{\rA}f=-\Pi\Delta f,\;\; f\in \mathcal{D}({\rA}),\;\; \mathcal{D}({\rA}):=W^{2,2}(D,\mathbb{R}^2)\cap \rV.
\end{eqnarray*}
It is well known (e.g., Cattabriga \cite{Cattabriga_1961}) that ${\rA}$ is positive self-adjoint with compact resolvent. Hence, there is an orthonormal basis
$\{e_i,\ i\in\mathbb{N}\}$\footnote{ We use $\mathbb{N}$ to denote the set of natural numbers starting from $1$ .}
of $\rH$, consisting of eigenvectors of ${\rA}$, with corresponding eigenvalues
$\{\lambda_i:\ i\in\mathbb{N}\}$, i.e., satisfying
$
{\rA}e_i=\lambda_i e_i$, $i \in \mathbb{N}$,
such that $0<\lambda_i $ for all $i$ and $\lambda_i\toup\infty$. In this paper, the space $\mathcal{D}({\rA})$ is  endowed with the following norm
$$
\|u\|_{\mathcal{D}({\rA})}:=\vert \rA u \vert_{\rH},\ \ u\in \mathcal{D}({\rA}).
$$

It is also well known that
\[
\rV=\mathcal{D}(\rA^{\frac{1}{2}}) \mbox{ and }
\Vert u \Vert_{\rV}^2= \vert \rA^{\frac{1}{2}} u \vert_{\rH}^2,\;\; u \in \rV.
\]

Let $\rB :\mathcal{D}(\rB ) \to \rH$, where $\mathcal{D}(\rB ) \subset \rH\times \rV$ is the bilinear operator defined as
$$
\rB (u,v)=\Pi[(u\cdot\nabla)v].
$$
{Without danger of ambiguity, by $\rB $ we also denote the quadratic function corresponding to the bilinear map $\rB $:
\[
\ \ \ \ \rB (u):=\rB (u,u).
\]
}

It is well known (\cite{Temam_2001}) that the Navier-Stokes equations can be formulated in the following abstract form:
\begin{eqnarray}\label{eq NS}
du(t)+{\rA}u(t)\,dt+\rB (u(t))\,dt=f(t)\,dt,\ \ u(0)=u_0,
\end{eqnarray}
where $u_0\in \rH$ and $f\in L^2_{loc}([0,\infty),\rV^\prime)$ denote respectively the initial data and the  external force.

Let $\rH^\prime$ and $\rV^\prime$ denote the dual spaces of $\rH$ and $\rV$ respectively. Considering the framework of the following Gelfand triple
$$
\rV\subset \rH \cong \rH^\prime\subset \rV^\prime,
$$
one can show that there exist unique extensions $\mathscr{A}$ and $\mathscr{B}$, of, respectively, ${\rA}$ and $\rB $, such that
$$
\mathscr{A}:\rV\to \rV^\prime,\ \ \ \mathscr{B}:\rV\times \rV\to \rV^\prime
$$
are bounded, respectively, linear and bilinear, maps. In what follows, in agreement with the practice of almost all papers on NSEs, these extensions are denoted by the original symbols $\rA$
and $\rB $ respectively.

\noindent

In the following lemma, we list some useful and well-known inequalities about the bilinear map $\rB $. Some of these are only true because we assume that  $D \subset \mathbb{R}^2$. In this list, $C$ denotes a generic constant.

\begin{lemma}\label{lem B baisc prop}
If $u,v,z\in \rV$, then
\begin{align*}
 \dual{\langle \rB (u,v),z\rangle}&=-\;\dual{\langle \rB (u,z),v\rangle},
 \;\;\; \dual{\langle \rB (u,v),v\rangle}=0,\\
|\dual{\langle \rB (u,v),z\rangle}|
  &\leq
  2 \|u\|_{\rV}^{\frac{1}{2}}|u|^{\frac{1}{2}}_{\rH}\|v\|^{\frac{1}{2}}_{\rV}|v|^{\frac{1}{2}}_{\rH}\|z\|_{\rV},\\
  |\dual{\langle \rB (u)-\rB (v),u-v\rangle}|
  &=
  |\dual{\langle \rB (u-v),v\rangle}|
  \leq
  \frac{1}{2} \|u-v\|^2_{\rV}+\|v\|^4_{L^4(D,\mathbb{R}^2)}|u-v|^2_{\rH},
  \\
|\rB (u,v)|^2_{\rH}
&\leq
C|u|_{\rH}\|u\|_{\rV}\|v\|_{\mathcal{D}({\rA})}\|v\|_{\rV},
\\
\|v\|^4_{L^4(D,\mathbb{R}^2)}&\leq 2|v|^2_{\rH}\|v\|^2_{\rV}.
\end{align*}

\end{lemma}
The last inequality, see \cite{Temam_2001}, is often called the Ladyzhenskya inequality.

In this paper, we consider SNSEs driven by multiplicative L\'evy noise in the following abstract form:
\begin{eqnarray}\label{eq SNS 01}
&& du(t)+{\rA}u(t)\,dt+\rB (u(t))\,dt=f(t)\,dt+\int_{\rZ}G(u(t-),z)\widetilde{\eta}(dz,dt),\nonumber\\
&& u_0\in \rH.
\end{eqnarray}
Here we make the following assumptions.
\begin{con}\label{ass-Z}
We assume that ${\rZ}$ is a locally compact Polish space, and $\nu$ is a $\sigma$-finite measure on $(\rZ\addaok{,\mathcal{B}(\rZ))}$, \addaok{where $\mathcal{B}(\rZ)$ denotes the Borel $\sigma$-field on $\rZ$}.\\
We assume that $(\Omega, \mathcal{F}, \mathbb{F}, \mathbb{P})$, where $\mathbb{F}=\{\mathcal{F}_t\}_{t\geq0}$, is a {filtered} probability space  satisfying the usual conditions{, i.e., the family $\mathbb{F}$ is right continuous, and every set $A$ belonging to the $\mathbb{P}$-completion of the $\sigma$-field  $\mathcal{F}_\infty$ with $\mathbb{P}(A)=0$ belongs to every $\mathcal{F}_t$, $t\geq 0$. }\\
We also assume that $\eta$ is a \addaok{time-homogenous} Poisson random measure on $[0,\infty)\times{\rZ}$ with the intensity measure $\Leb\otimes\nu$ on $(\Omega, \mathcal{F}, \mathbb{F}, \mathbb{P})$, where  $\Leb$ is the Lebesgue measure on $[0,\infty)$.\\
We define the compensated Poisson random measure $\widetilde{\eta}$ by
\begin{align}\label{eqn-eta tilde}\widetilde{\eta}([0,t]\times O)=\eta([0,t]\times O)-t\nu(O),\;\;\; t\geq 0,
\end{align}
whenever $O\in\mathcal{B}({\rZ})$ is such that $ \nu(O)<\infty$.
\end{con}

\addaok{Let us point out that the measure $\Leb\otimes\nu$ is a $\sigma$-finite measure on $([0,\infty)\times \rZ, \mathcal{B}([0,\infty))\otimes \mathcal{B}(\rZ))$.}

\vskip 0.2cm
In the following, if $X$ is a metric space and $I \subset \mathbb{R}$ is a time interval, we denote by $D(I,X)$ the space of all c\`adl\`ag paths from $I$ to $X$.

\section{
Solutions to SNSEs with initial data in the space $\rH$}
\label{subsection 2.1}

Our treatment of SNSEs \eqref{eq SNS 01} consists of two steps. In the first, we assume that the coefficient $G$ is globally Lipschitz. In the second, we
assume that $G$ is Lipschitz on balls and has  linear growth.

 Below, we present our standing assumptions on the coefficient  $G$ in the first step.
\begin{con}\label{con G} We assume that
$G:\rH\times {\rZ}\to \rH$ is a measurable map such that  there exist positive constants $C_1$ and $C_2$ such that
\begin{itemize}
\item[(G-H1)](Global Lipschitz)
\begin{equation}\label{eqn-Lipschitz-H}
       \int_{\rZ}|G(v_1,z)-G(v_2,z)|_{\rH}^2\nu(dz)\leq C_1|v_1-v_2|_{\rH}^2,\;\;\; v_1,v_2\in \rH,
\end{equation}

\item[(G-H2)](Linear growth)
\begin{equation}\label{eqn-linear growth-H}
       \int_{{\rZ}}|G(v,z)|^2_{\rH}\nu(dz)\leq C_2(1+|v|^2_{\rH}),\;\;\; v\in \rH.
      \end{equation}

\end{itemize}
\end{con}
\begin{remark}\label{rem-growth}
We note that the linear growth condition \eqref{eqn-linear growth-H} follows from the global Lipschitz condition \eqref{eqn-Lipschitz-H} and the following one, with
$C_2=2 \max\{C_1, \int_{{\rZ}}|G(0,z)|^2_{\rH}\nu(dz)\}$.

\begin{equation}\label{eqn-H3}
       \int_{{\rZ}}|G(0,z)|^2_{\rH}\nu(dz)<\infty.
      \end{equation}

\end{remark}

First, we prove  the following existence result in the \newok{natural}  setting.
\begin{thm}\label{thm-initial data in H}
Assume that \textbf{Assumption \ref{con G}} holds. Then, for all $u_0\in \rH$ and $f\in L^2_{loc}([0,\infty),\rV^\prime)$, there exists a unique $\mathbb{F}$-{progressively measurable} process $u$ such that
\begin{itemize}
\item[(1)] $u\in D([0,\infty),\rH)\cap L^2_{loc}([0,\infty),\rV)$, $\mathbb{P}$-a.s.,

\item[(2)] the following equality holds, for all $t\in[0,\infty)$, $\mathbb{P}$-a.s., in $\rV^\prime$,
      \begin{eqnarray}\label{eq star page 3}
      \hspace{-1truecm}u(t)=u_0-\int_0^t{\rA}u(s)\,ds - \int_0^t\rB(u(s))\,ds+\int_0^t f(s)\,ds+\int_0^t\int_{{\rZ}}G(u({s-}),z)\widetilde{\eta}(dz,ds).
      \end{eqnarray}
\end{itemize}
Moreover, the solution $u$ satisfies the following estimate:
For any $T>0$, 
\begin{eqnarray*}\label{eqn-apriori estimates-H}
\mathbb{E}\Big(\sup_{t\in[0,T]}|u(t)|^2_{\rH}\Big)
    +
   \mathbb{E}\Big(\int_0^{T}\|u(s)\|^2_{\rV}\,ds\Big)
   \leq
   C_T(1+|u_0|^2_{\rH}+ \int_0^T\|f(s)\|^2_{\rV^\prime}\,ds).
\end{eqnarray*}
\end{thm}

\begin{remark}\label{Rem1}
Assumption \ref{con G} is a fairly standard assumption when one considers the existence and uniqueness of solutions to SPDEs driven by multiplicative Gaussian noise.
However, for the case of L\'evy noise, the literature results always require additional assumptions on $G$ besides Assumption \ref{con G}. For example, in \cite{Dong+Xie_2009}, the authors assume that there exists a sequence $(\rZ_m)_{ m\in\mathbb{N}}$ of measurable subsets of $\rZ$ with $\rZ_m\toup \rZ$ and $\nu(\rZ_m)<\infty$ such that, for any $k>0$,
\begin{equation}\label{eqn_Dong+Xie_2009}
\sup_{|v|_{\rH}\leq k}\int_{\rZ^c_m}|G(v,z)|^2_{\rH}\nu(dz)\to 0\text{ as }m\to\infty,
\end{equation}
while in \cite{BHZ} and \cite{BLZ}, it is assumed that
$K>0$ exists such that
\begin{equation}\label{eqn_BHZ}
\int_{{\rZ}}|G(v,z)|^4_{\rH}\nu(dz)\leq K(1+|v|_{\rH}^4),\;\; v\in \rH.
\end{equation}
Similarly, Motyl, in a paper (\cite{EM_2014}), assumed that for each $p\in\{1,2,2+\gamma,4,4+2\gamma\}$, where $\gamma$ is some positive constant, there exists a constant
$c_p>0$ such that
$$
\int_{{\rZ}}|G(v,z)|_{\rH}^p\nu(dz)\leq c_p(1+|v|_{\rH}^p),\;\; v\in \rH.
$$
Hence, our Theorem \ref{thm-initial data in H} improves the existing results in the
literature.

\dela{
\textcolor[rgb]{1.00,0.00,0.00}{Let us mention that in a forthcoming paper \cite{BPZ_2021-local Lipschitz}, we are generalizing the results from the present paper to the case when the coefficient $G$ is Lipschitz on balls, i.e., it satisfies the following condition.
For every $\hbar>0$ there exists $C_\hbar>0$ such that
\begin{equation}\label{eqn-Lipschitz-H-local}
       \int_{\rZ}|G(v_1,z)-G(v_2,z)|_{\rH}^2\nu(dz)\leq C_\hbar|v_1-v_2|_{\rH}^2,\;\;\; v_1,v_2\in B_\hbar(\rH),
\end{equation}
where $B_\hbar(\rH)$ denotes the ball of radius $\hbar$ in $\rH$.
The case when $G$ is only locally Lipschitz is open.}\textcolor[rgb]{0.44,0.00,0.94}{Should be delete? from Zhai.}
}
\end{remark}

In the second step, we relax the global Lipschitz condition in \textbf{Assumption \ref{eqn-Lipschitz-H}} and consider the following assumptions.

\begin{con}\label{con G Local}
 We assume that $G:\rH\times {\rZ}\rightarrow \rH$ is a measurable map such that
\begin{itemize}
\item[(G-H1-local)](Lipschitz on balls) for every $\hbar>0$, there exists a constant $C_\hbar>0$ such that, for all $v_1,v_2\in\rH$ with $|v_1|_\rH\vee|v_2|_\rH\leq \hbar$,
\begin{equation}\label{eqn-Lipschitz-H-local}
       \int_{\rZ}|G(v_1,z)-G(v_2,z)|^2_{\rH}\nu(dz)\leq C_\hbar|v_1-v_2|^2_{\rH},
\end{equation}
and it satisfies the assumption {(G-H2)(Linear growth)}, i.e., that \eqref{eqn-linear growth-H} holds.
\end{itemize}
\end{con}

Let us now formulate our main theorem in this relaxed framework.
\begin{thm}\label{thm-initial data in H local}
Assume that \textbf{Assumption \ref{con G Local}} holds.
Then,
for every $u_0\in \rH$ and $f\in L^2_{loc}([0,\infty),\rV^\prime)$,
there exists  a unique $\mathbb{F}$-{progressively measurable} process $u$ such that
\begin{itemize}
\item[(1)] $u\in D([0,\infty),\rH)\cap L^2_{loc}([0,\infty),\rV)$, $\mathbb{P}$-a.s.,

\item[(2)] the following equality holds, for all $t\in[0,\infty)$, $\mathbb{P}$-a.s., in $\rV^\prime$,
      \begin{eqnarray*}\label{eq star page 3}
      u(t)&=&u_0-\int_0^t{\rA}u(s)\,ds - \int_0^t\rB(u(s))\,ds+\int_0^t f(s)\,ds\nonumber\\
      &&+
      \int_0^t\int_{{\rZ}}G(u({s-}),z)\widetilde{\eta}(dz,ds).
      \end{eqnarray*}
\end{itemize}
\end{thm}
\begin{proof}[Proof of Theorem \ref{thm-initial data in H local}]
The proof of Theorem \ref{thm-initial data in H local} is based on the proof of Theorem \ref{thm-initial data in H} and the standard truncation procedure, and it is essentially the same as the proof of Theorem
3.1 in \cite{Bao Truman Yuan 2011}, keeping in mind that for any $u,v\in \rV$,
$
 \dual{\langle \rB(u,v),v\rangle}=0.
$
The proof proceeds as follows.
\dela{\textcolor[rgb]{0.44,0.00,0.94}{From Zhai: the details:}}

For any natural number $k\in \mathbb{N}$, we define a map $G_k$ by
$$
G_k: \rH\times {\rZ} \ni (y,z) \mapsto G\Big(\frac{|y|_\rH\wedge k}{|y|_\rH}y,z\Big) \in \rH,
$$
where we put $\frac{|y|_\rH\wedge k}{|y|_\rH}=1$ when $y=0$. Since $G$ satisfies \textbf{Assumption \ref{con G Local}}, we observe that
the map $G_k$ satisfies \textbf{Assumption \ref{con G}}.
Hence, for every $k>|u_0|_\rH$, there exists by Theorem \ref{thm-initial data in H}
a unique $\mathbb{F}$-{progressively measurable} process $X^k$ such that
\begin{itemize}
 \item $X^k\in D([0,\infty),\rH)\cap L^2_{loc}([0,\infty),\rV)$, $\mathbb{P}$-a.s.,
 \item the following equality holds, for all $t\in[0,\infty)$, $\mathbb{P}$-a.s., in $\rV^\prime$,
\end{itemize}
      \begin{equation*}\label{Eq th local H 01}
      X^k(t)=u_0-\int_0^t{\rA}X^k(s)\,ds - \int_0^t\rB(X^k(s))\,ds+\int_0^t f(s)\,ds+\int_0^t\int_{{\rZ}}G_k(X^k({s-}),z)\widetilde{\eta}(dz,ds).
      \end{equation*}
Define {a random time $\sigma_k$ by }
{\begin{equation}\label{eqn-sigma_k}
\sigma_k:=\inf\{t\geq 0:|X^k(t)|_\rH>k\},
\end{equation}}
where{, for the whole paper, we adopt a convention that }  $\inf \emptyset= \infty$. {By Theorem 2.1.6 from \cite{Ethier+Kurtz_1986} it follows that $\sigma_k$ is a stopping time.} It is not difficult to see that $\sigma_k$ is increasing in $k$, and $X^{k+1}(t)=X^k(t),\, t\in[0,\sigma_k)$.
This enables us to define a stopping time $\sigma:=\lim_{k\rightarrow\infty}\sigma_k$ and a process
$u=\{u(t), \;\;t\in[0,\sigma)\}$ as follows
$$
u(t):=X^k(t),\, t\in[0,\sigma_k).
$$
It is easy to see that $u(t),\, t\in[0,\sigma)$ is a local solution of Problem (\ref{eq SNS 01}). To complete the proof, we need only show that
$\mathbb{P}(\sigma=\infty)=1$.

By {the It\^o} formula, see, e.g., \cite{GK_1982} and \cite{BHZ}, we have
\begin{eqnarray*}
&&\hspace{-0.3truecm}|u(t\wedge\sigma_k)|^2_{\rH}+2\int_0^{t\wedge\sigma_k}\|u(s)\|^2_{\rV}\,ds\\
&&=
  |u_0|^2_{\rH}
 +
  2\int_0^{t\wedge\sigma_k}\dual{\langle f(s),u(s)\rangle}\,ds
  +
  2\int_0^{t\wedge\sigma_k}\int_{{\rZ}}\langle G(u(s-),z),u(s-)\rangle_{\rH}\widetilde{\eta}(dz,ds)\\
  &&\ \ +
  \int_0^{t\wedge\sigma_k}\int_{{\rZ}}|G(u(s-),z)|^2_{\rH} \eta(dz,ds),\;\;\mathbb{P}\mbox{-a.s., for }t\geq 0.\nonumber
\end{eqnarray*}
Noting that the process $\int_0^{t\wedge\sigma_k}\int_{{\rZ}}\langle G(u(s-),z),u(s-)\rangle_{\rH}\widetilde{\eta}(dz,ds)$, $t\geq 0$, is a martingale, we infer that
$$
\mathbb{E}\int_0^{t\wedge\sigma_k}\int_{{\rZ}}\langle G(u(s-),z),u(s-)\rangle_{\rH}\widetilde{\eta}(dz,ds)=0.
$$
Thus, it follows by the linear growth condition \eqref{eqn-linear growth-H}, which is a part of \textbf{Assumption \ref{con G Local}}, that
there exists $C>0$ such that for all $k$,
\begin{eqnarray*}
&&\hspace{-1truecm}\mathbb{E}|u(t\wedge\sigma_k)|^2_{\rH}+\mathbb{E}\int_0^{t\wedge\sigma_k}\|u(s)\|^2_{\rV}\,ds\\
&\leq&
  |u_0|^2_{\rH}
  +
  \int_0^{t}\|f(s)\|_{\rV^\prime}^2ds
  +
  C\mathbb{E}\int_0^{t}(1+|u(s\wedge\sigma_k)|^2_{\rH})ds,\;\;t\geq 0.\nonumber
\end{eqnarray*}
Therefore, by applying Gronwall's lemma, we deduce that
\begin{eqnarray*}
\mathbb{E}|u(t\wedge\sigma_k)|^2_{\rH}
\leq
  \Big(|u_0|^2_{\rH}
  +
  \int_0^{t}\|f(s)\|_{\rV^\prime}^2ds
  +
  Ct\Big)e^{Ct},\ \ t\geq 0,
\end{eqnarray*}
which further gives
$$
\mathbb{P}(\sigma_k\leq t)\leq \frac{\mathbb{E}(|u(t\wedge\sigma_k)|^2_{\rH}\mathds{1}_{\sigma_k\leq t})}{k^2}
\leq
\frac{\Big(|u_0|^2_{\rH}
  +
  \int_0^{t}\|f(s)\|_{\rV^\prime}^2ds
  +
  Ct\Big)e^{Ct}}{k^2},\;\;t\geq 0.
$$
Letting $k\rightarrow\infty$, we obtain
$$
\mathbb{P}(\sigma\leq t)=0,\;\;t\geq 0.
$$
Since $t\geq 0$ is arbitrary, we must have
$$
\mathbb{P}(\sigma=\infty)=1.
$$

The proof of Theorem \ref{thm-initial data in H local} is complete.
\end{proof}

To prove Theorem \ref{thm-initial data in H}, we first introduce the following notation (used throughout the paper)  and state three preliminary and auxiliary results: Lemmata \ref{lem 1} and \ref{lem 2}, and Corollary \ref{cor 1}.

The following notation is useful.
For $T\geq 0$,
\begin{equation}
\label{eqn-Upsilon_T-H}
\Upsilon_T(\rH)= D([0,T],\rH)\cap L^2([0,T],\rV).
\end{equation}
It is standard that the space $\Upsilon_T(\rH)$ is endowed  with the norm $\|\cdot \|_{\Upsilon_T(\rH)}$ defined by
\begin{equation}
\label{eqn-Upsilon_T-H-norm}
\|y\|_{\Upsilon_T(\rH)}=\sup_{s\in[0,T]}|y(s)|_{\rH} + \Big(\int_0^T\|y(s)\|^2_{\rV} \, ds\Big)^{\frac{1}{2}}
\end{equation}
is a Banach space.
\todozbdel{
It would more natural to use the norm below.
\begin{equation*}
\|y\|_{\Upsilon_T(\rH)}^2=\sup_{s\in[0,T]}|y(s)|_{\rH}^2 + \int_0^T\|y(s)\|^2_{\rV} \, ds.
\end{equation*}
Is it safe to make such a change?
}

Let $\Lambda_T(\rH)$ be the space of all
$\rH$-valued {c\`adl\`ag} $\mathbb{F}$-{progressively measurable} processes {$y:[0,T]\times \Omega \tof \rV$ such that $\mathbb{P}$-a.s. its trajectories belong to the space $\Upsilon_T(\rH)$ and}
\begin{equation}
\label{eqn-Lambda_T-H}
\|y\|_{\Lambda_T(\rH)}^2:=\mathbb{E}\Big(\sup_{s\in[0,T]}|y(s)|_{\rH}^2 + \int_0^T\|y(s)\|^2_{\rV}\,ds\Big)<\infty.
\end{equation}

For every $m\in \mathbb{N}$, let us fix a function $\theta_m:[0,\infty)\to [0,1]$ satisfying

\begin{equation}
\label{eqn-theta_m}
\left\{ \begin{split}
\theta_m\in C^2[0,\infty), & \sup_{t\in[0,\infty)}|\theta_m'(t)|\leq {C_1}<\infty, \\
\mathds{1}_{[0,m]} &\leq \theta_m \leq \mathds{1}_{[0,m+1]},
\end{split}
\right.\end{equation}
for some constant $C_1>0$ which is independent of $m$. We also  set \[\phi =\theta_1.
\]
Let us also define, for every $\delta>0$, a function $\phi_\delta:[0,\infty)\to[0,1]$ by
 \[\phi_\delta(r)=\phi(\delta r), \;\; r \in [0,\infty).\]
It can be easily seen that every function $\phi_\delta$ satisfies the following conditions:

\begin{equation}
\label{eqn-phi_delta}
\left\{ \begin{split}
\phi_\delta\in C^2[0,\infty), & \sup_{t\in[0,\infty)}|\phi_\delta^\prime(t)|\leq {\Cphi}\delta, \\
\mathds{1}_{[0,\frac{1}{\delta}]} &\leq \phi_\delta \leq \mathds{1}_{[0,\frac{2}{\delta}]}.
\end{split}
\right.\end{equation}

We are now ready to state the first of the three promised auxiliary results.

\begin{lemma}\label{lem 1}
Assume that  $T>0$, $m\in \mathbb{N}$, $M\in\Upsilon_T(\rH)$, $u_0\in \rH$ and $f\in L^2 ([0,T],\rV^\prime)$. Then, there exists a function $Y\in C([0,T],\rH)\cap L^2([0,T],\rV)$ satisfying
\begin{eqnarray}\label{eq lemma 1 01}
&& dY(t)+{\rA}Y(t)\,dt+\theta_m(\|Y+M\|_{\Upsilon_t(\rH)})\rB(Y(t)+M(t))\,dt=f(t)\,dt,\nonumber\\
&& Y(0)=u_0.
\end{eqnarray}
\todozbdel{This lemma does not involve Assumption \eqref{eqn-Lipschitz-H} at all and hence it is sufficient to assume that \eqref{eqn-Lipschitz-H-local}. }

\end{lemma}

\begin{proof}[Proof of Lemma \ref{lem 1}]
 The proof is divided into three steps.

\begin{proof}[{\bf Step 1.}]
Let us fix $T>0$, $m\in \mathbb{N}$, $M\in\Upsilon_T(\rH)$, $u_0\in \rH$, and $f\in L^2 ([0,T],\rV^\prime)$.\\
We will use the Picard iteration method to prove that there exists a number $\delta_0>0$ depending only on $m$, and there exists $X\in C([0,T],\rH)\cap L^2([0,T],\rV)$ which solves the following auxiliary deterministic evolution equation with $\delta=\delta_0$.
\todozbdel{Here we do not use \eqref{eqn-Lipschitz-H} at all and hence it is sufficient to assume that \eqref{eqn-Lipschitz-H-local}. }
\begin{eqnarray}\label{App auxi spde 01}
&& X^\prime(t)+{\rA}X(t)+\theta_m(\|X+M\|_{\Upsilon_t(\rH)})\phi_\delta(\|X+M\|_{\xit})\rB(X(t)+M(t))=f(t),\nonumber\\
&& X(0)=u_0.
\end{eqnarray}

\vskip 0.2cm
Let us choose and fix\footnote{For instance $y(t)=e^{-t\rA}u_0$, $t \in [0,T]$.}
$y_0\in C([0,T],\rH)\cap L^2([0,T],\rV)$ with $y_0(0)=u_0$. Suppose that for $n\in \mathbb{N}$ a function $y_{n}\in C([0,T],\rH)\cap L^2([0,T],\rV)$ such that $y_{n}(0)=u_0$ is given. Let us observe that
it is not difficult to prove that there exists a unique $y_{n+1}\in C([0,T],\rH)\cap L^2([0,T],\rV)$ solving the following linear
evolution equation
\todozbdel{Here we do not use \eqref{eqn-Lipschitz-H} at all and hence it is sufficient to assume that \eqref{eqn-Lipschitz-H-local}. }
\begin{eqnarray}\label{App auxi spde 02}
\left\{
 \begin{array}{rcl} \DS
 y^\prime_{n+1}(t)&\!\!+\!\!&{\rA}y_{n+1}(t)+\theta_m(\|y_n+M\|_{\Upsilon_t(\rH)})\phi_\delta(\|y_n+M\|_{\xit})\\
 &&\\
 \DS && \rB(y_n(t)+M(t),y_{n+1}(t)+M(t))=f(t), \, \, t\in [0,T];\\
 &&\\
\DS y_{n+1}(0)&\!\!=\!\!&u_0.
 \end{array}
\right.
\end{eqnarray}

Our aim is to show that the sequence $\{y_{n}:n\in\mathbb{N}\}$ is a Cauchy sequence in $C([0,T],\rH)\cap L^2([0,T],\rV)$.

We now estimate the norm of the difference $y_{n+1}-y_n$ {for $n\geq 1$.} {We cannot do this for $n=0$.}

Given four functions $x_i\in C([0,T],\rH)\cap L^2([0,T],\rV)$, $i=1,\cdots,4$, we set , for $ t\in [0,T]$,
$$
\Pi(x_1,x_2,x_3,x_4)(t)=\theta_m(\|x_1+M\|_{\Upsilon_t(\rH)})\phi_\delta(\|x_2+M\|_{\xit})\rB(x_3(t)+M(t),x_4(t)+M(t))
$$
and
$$
\Xi(x_1,x_2)(t)=\theta_m(\|x_1+M\|_{\Upsilon_t(\rH)})\phi_\delta(\|x_2+M\|_{\xit}).
$$
By \cite[Lemma III.1.2]{Temam_2001} we have 
\begin{eqnarray}\label{App auxi esti 01}
&& \hspace{-1truecm}|y_{n+1}(t)-y_n(t)|^2_{\rH}+2\int_0^t\|y_{n+1}(s)-y_n(s)\|^2_{\rV}\,ds\\
&=&
-2\int_0^t \duality{ \Pi(y_n,y_n,y_n,y_{n+1})(s)-\Pi(y_{n-1},y_{n-1},y_{n-1},y_{n})(s),y_{n+1}(s)-y_n(s)}\,ds\nonumber\\
&=&
-2\int_0^t I(s)\,ds,\;\;t \in [0,T],\nonumber
\end{eqnarray}
where, with the processes $I_1$ and $I_2$ defined, for $s \in [0,T]$, by
\begin{eqnarray*}
I_1(s)&=&\Xi(y_n,y_n)(s)\dual{\Big\langle \rB(y_n(s)-y_{n-1}(s),y_n(s)+M(s)),y_{n+1}(s)-y_n(s)\Big\rangle},
\\
I_2(s)&=&\Big(\Xi(y_n,y_n)(s)-\Xi(y_{n-1},y_{n-1})(s)\Big)\\
&&\;\;\;\;\; \vspace{2truecm}\lefteqn{\times \dual{\Big\langle \rB(y_{n-1}(s)+M(s),y_n(s)+M(s)),y_{n+1}(s)-y_n(s)\Big\rangle},}
\end{eqnarray*}
we have
\begin{eqnarray}\label{app I 01}
I(s)
&=&
\dual{\Big\langle \Pi(y_n,y_n,y_n,y_{n+1})(s)-\Pi(y_{n},y_{n},y_{n},y_{n})(s),y_{n+1}(s)-y_n(s)\Big\rangle}\nonumber\\
&&+
\dual{\Big\langle \Pi(y_n,y_n,y_n,y_{n})(s)-\Pi(y_{n},y_{n},y_{n-1},y_{n})(s),y_{n+1}(s)-y_n(s)\Big\rangle}\nonumber\\
&&+
\dual{\Big\langle \Pi(y_n,y_n,y_{n-1},y_{n})(s)-\Pi(y_{n-1},y_{n-1},y_{n-1},y_{n})(s),y_{n+1}(s)-y_n(s)\Big\rangle}\nonumber\\
&=&
0+I_1(s)+I_2(s), \;\;\;s \in [0,T].
\end{eqnarray}

To estimate $I(s)$, for a fixed $s \in [0,T]$, we will consider three cases, with \textbf{Case 1} being divided into
three subcases.  Each case will contain a calculation of a certain
``partial" integral $\int_0^t|I(s)|ds$.

\begin{proof}[{\bf Case 1.}] Assume that $\|y_n+M\|_{\xis}\leq \frac{3}{\delta}$ and $\|y_{n-1}+M\|_{\xis}\leq \frac{3}{\delta}$.
\vskip 0.2cm
\noindent{\bf Subcase 1.1}  Assume further that $\|y_n+M\|_{\Upsilon_s(\rH)}\leq m+2$ and $\|y_{n-1}+M\|_{\Upsilon_s(\rH)}> m+2$.


The definition of $\theta_m$ implies that in this subcase
$$
I(s)=\dual{\langle \Pi(y_n,y_n,y_n,y_{n+1})(s),y_{n+1}(s)-y_n(s)\Big\rangle}
$$
and
\begin{align}\label{eq subcase 1.1}
&\hspace{-0.6truecm}\int_0^t|I(s)|\mathds{1}_{\{\|y_n+M\|_{\xis}\vee\|y_{n-1}+M\|_{\xis}\leq \frac{3}{\delta}\}}
    \mathds{1}_{\{\|y_n+M\|_{\Upsilon_s(\rH)}\leq m+2\}}
    \mathds{1}_{\{\|y_{n-1}+M\|_{\Upsilon_s(\rH)}> m+2\}}\,ds\\
   &\hspace{-0.6truecm}=
\int_0^t|I(s)|\mathds{1}_{\{\|y_n+M\|_{\xis}\vee\|y_{n-1}+M\|_{\xis}\leq \frac{3}{\delta}\}}
  \mathds{1}_{\{\|y_n+M\|_{\Upsilon_s(\rH)}\leq m+1\}}
     \mathds{1}_{\{\|y_{n-1}+M\|_{\Upsilon_s(\rH)}> m+2\}}\,ds.
    \nonumber
\end{align}
For any $s\in[0,t]$ such that
$$
\mathds{1}_{\{\|y_n+M\|_{\xis}\vee\|y_{n-1}+M\|_{\xis}\leq \frac{3}{\delta}\}}
    \mathds{1}_{\{\|y_n+M\|_{\Upsilon_s(\rH)}\leq m+1\}}\mathds{1}_{\{\|y_{n-1}+M\|_{\Upsilon_s(\rH)}> m+2\}}=1,
$$
we have
\begin{eqnarray}\label{eq subcase 1.1 01}
\|y_n-y_{n-1}\|_{\Upsilon_s(\rH)}=\|(y_n+M)-(y_{n-1}+M)\|_{\Upsilon_s(\rH)}\geq 1,
\end{eqnarray}
and for any ${\eps}>0$,
\begin{eqnarray}\label{app case 1.2 01}
 \Big|I(s)\Big|
&\leq&
   \Big| \dual{\Big\langle \rB(y_n(s)+M(s),y_{n+1}(s)+M(s)),-y_n(s)-M(s)\Big\rangle}
   \Big|\nonumber\\
&=&
   \Big| \dual{\Big\langle \rB(y_n(s)+M(s),y_{n+1}(s)-y_n(s)),-y_n(s)-M(s)\Big\rangle}
   \Big|\nonumber\\
&\leq&
   2|y_n(s)+M(s)|^{\frac{1}{2}}_{\rH}\|y_n(s)+M(s)\|^{\frac{1}{2}}_{\rV}|y_{n+1}(s)-y_n(s)|^{\frac{1}{2}}_{\rH}\nonumber\\
   &&\cdot\|y_{n+1}(s)-y_n(s)\|^{\frac{1}{2}}_{\rV}\|y_n(s)+M(s)\|_{\rV}\cdot \|y_n-y_{n-1}\|_{\Upsilon_s(\rH)}\nonumber\\
&\leq&
   \frac{3}{2}{\eps}^{4/3}\|y_{n+1}(s)-y_n(s)\|^{2/3}_{\rV}\|y_n(s)+M(s)\|^{4/3}_{\rV}\|y_n-y_{n-1}\|^{4/3}_{\Upsilon_s(\rH)}\nonumber\\
   &&+
    \frac{(m+2)^2}{2{\eps}^4}\|y_n(s)+M(s)\|^{2}_{\rV}|y_{n+1}(s)-y_n(s)|^{2}_{\rH}\nonumber\\
&\leq&
   \frac{1}{2}{\eps}^{4/3}\|y_{n+1}(s)-y_n(s)\|^{2}_{\rV}+{\eps}^{4/3}\|y_n(s)+M(s)\|^{2}_{\rV}\|y_n-y_{n-1}\|^{2}_{\Upsilon_s(\rH)}\nonumber\\
   &&+
    \frac{(m+2)^2}{2{\eps}^4}\|y_n(s)+M(s)\|^{2}_{\rV}|y_{n+1}(s)-y_n(s)|^{2}_{\rH}.
\end{eqnarray}
In the second ``$\leq$" of (\ref{app case 1.2 01}), we have used (\ref{eq subcase 1.1 01}).
\vskip 0.2cm

By inequalities (\ref{eq subcase 1.1}) and (\ref{app case 1.2 01}), we get
\begin{eqnarray}\label{Eq Claim 1.1}
&&\hspace{-0.7truecm}
  \int_0^t|I(s)|\mathds{1}_{\{\|y_n+M\|_{\xis}\vee\|y_{n-1}+M\|_{\xis}\leq \frac{3}{\delta}\}}
    \mathds{1}_{\{\|y_n+M\|_{\Upsilon_s(\rH)}\leq m+2\}}
     \mathds{1}_{\{\|y_{n-1}+M\|_{\Upsilon_s(\rH)}> m+2\}}\,ds\nonumber\\
&\leq&
  \frac{1}{2}{\eps}^{4/3}\int_0^t\|y_{n+1}(s)-y_n(s)\|^{2}_{\rV}\,ds
  \nonumber\\
  &&+
  {\eps}^{4/3}\|y_n-y_{n-1}\|^{2}_{\Upsilon_t(\rH)}\int_0^t\|y_n(s)+M(s)\|^{2}_{\rV}\mathds{1}_{\{\|y_n+M\|_{\xis}\leq \frac{3}{\delta}\}}\,ds\nonumber\\
   &&+
    \frac{(m+2)^2}{2{\eps}^4}\int_0^t\|y_n(s)+M(s)\|^{2}_{\rV}\mathds{1}_{\{\|y_n+M\|_{\xis}\leq \frac{3}{\delta}\}}(s)\,ds\sup_{s\in[0,t]}|y_{n+1}(s)-y_n(s)|^{2}_{\rH}\nonumber\\
&\leq&
  \frac{1}{2}{\eps}^{4/3}\int_0^t\|y_{n+1}(s)-y_n(s)\|^{2}_{\rV}\,ds
  +
  {\eps}^{4/3}\|y_n-y_{n-1}\|^{2}_{\Upsilon_t(\rH)}\cdot \frac{9}{\delta^2}\nonumber\\
   &&+
    \frac{(m+2)^2}{2{\eps}^4}\cdot \frac{9}{\delta^2}\sup_{s\in[0,t]}|y_{n+1}(s)-y_n(s)|^{2}_{\rH}\nonumber\\
&\leq&
  \Big(\frac{1}{2}{\eps}^{4/3}+\frac{(m+2)^2}{2{\eps}^4}\cdot \frac{9}{\delta^2}\Big)\|y_{n+1}-y_{n}\|^{2}_{\Upsilon_t(\rH)}
  +
  \frac{9}{\delta^2}{\eps}^{4/3}\|y_n-y_{n-1}\|^{2}_{\Upsilon_t(\rH)}.
\end{eqnarray}

\noindent{\bf Subcase 1.2} Assume further that $\|y_n+M\|_{\Upsilon_s(\rH)}> m+2$ and $\|y_{n-1}+M\|_{\Upsilon_s(\rH)}\leq m+2$.

Similar to \textbf{Subcase 1.1}, in this subcase, we have
$$
I(s)=-\dual{\Big\langle \Pi(y_{n-1},y_{n-1},y_{n-1},y_{n})(s),y_{n+1}(s)-y_n(s)\Big\rangle},
$$
\begin{eqnarray}\label{Eq Subcase 1.2}
&&\hspace{-1.1truecm}\lefteqn{\int_0^t\!|I(s)|\mathds{1}_{\{\|y_n+M\|_{\xis}\vee\|y_{n-1}+M\|_{\xis}\leq \frac{3}{\delta}\}}
    \!\mathds{1}_{\{\|y_n+M\|_{\Upsilon_s(\rH)}> m+2\}}\!\mathds{1}_{\{\|y_{n-1}+M\|_{\Upsilon_s(\rH)}\leq m+2\}}\!\,ds}\\
&&\hspace{-1.1truecm}=
 {\int_0^t|I(s)|\mathds{1}_{\{\|y_n+M\|_{\xis}\vee\|y_{n-1}+M\|_{\xis}\leq \frac{3}{\delta}\}}
    \mathds{1}_{\{\|y_n+M\|_{\Upsilon_s(\rH)}> m+2\}}\mathds{1}_{\{\|y_{n-1}+M\|_{\Upsilon_s(\rH)}\leq m+1\}}\,ds},\nonumber
\end{eqnarray}
and for any $s\in[0,t]$ such that
$$
\mathds{1}_{\{\|y_n+M\|_{\xis}\vee\|y_{n-1}+M\|_{\xis}\leq \frac{3}{\delta}\}}
    \mathds{1}_{\{\|y_n+M\|_{\Upsilon_s(\rH)}> m+2\}}\mathds{1}_{\{\|y_{n-1}+M\|_{\Upsilon_s(\rH)}\leq m+1\}}=1,
$$
we have
\begin{eqnarray}\label{eq subcase 1.2 01}
\|y_n-y_{n-1}\|_{\Upsilon_s(\rH)}=\|(y_n+M)-(y_{n-1}+M)\|_{\Upsilon_s(\rH)}\geq 1,
\end{eqnarray}
and, for any ${\eps}>0$,
\begin{eqnarray}\label{app case 1.3 01}
 \Big|I(s)\Big|
&\leq&
   \Big|\dual{ \Big\langle \rB(y_{n-1}(s)+M(s),y_{n}(s)+M(s)),y_{n+1}(s)-y_n(s)\Big\rangle}
   \Big|\nonumber\\
&\leq&
   2|y_{n-1}(s)+M(s)|^{\frac{1}{2}}_{\rH}\|y_{n-1}(s)+M(s)\|^{\frac{1}{2}}_{\rV}|y_{n+1}(s)-y_n(s)|^{\frac{1}{2}}_{\rH}\nonumber\\
   &&\cdot\|y_{n+1}(s)-y_n(s)\|^{\frac{1}{2}}_{\rV}\|y_n(s)+M(s)\|_{\rV}\nonumber\\
&\leq&
   \frac{1}{2{\eps}^4}|y_{n+1}(s)-y_n(s)|^{2}_{\rH}|y_{n-1}(s)+M(s)|^{2}_{\rH}\|y_{n-1}(s)+M(s)\|^{2}_{\rV}\nonumber\\
   &&+
   \frac{3}{2}{\eps}^{4/3}\|y_{n+1}(s)-y_n(s)\|^{2/3}_{\rV}\|y_{n}(s)+M(s)\|^{4/3}_{\rV}\nonumber\\
&\leq&
   \frac{1}{2{\eps}^4}(m+2)^2|y_{n+1}(s)-y_n(s)|^{2}_{\rH}\|y_{n-1}(s)+M(s)\|^{2}_{\rV}\nonumber\\
   &&+
   \frac{1}{2}{\eps}^{4/3}\|y_{n+1}(s)-y_n(s)\|^{2}_{\rV}+{\eps}^{4/3}\|y_{n}(s)+M(s)\|^{2}_{\rV}\nonumber\\
&\leq&
   \frac{1}{2{\eps}^4}(m+2)^2|y_{n+1}(s)-y_n(s)|^{2}_{\rH}\|y_{n-1}(s)+M(s)\|^{2}_{\rV}\nonumber\\
   &&+
   \frac{1}{2}{\eps}^{4/3}\|y_{n+1}(s)-y_n(s)\|^{2}_{\rV}+2{\eps}^{4/3}\|y_{n-1}(s)+M(s)\|^{2}_{\rV}+2{\eps}^{4/3}\|y_n(s)-y_{n-1}(s)\|^{2}_{\rV}\nonumber\\
&\leq&
   \frac{1}{2{\eps}^4}(m+2)^2|y_{n+1}(s)-y_n(s)|^{2}_{\rH}\|y_{n-1}(s)+M(s)\|^{2}_{\rV}
   +
   \frac{1}{2}{\eps}^{4/3}\|y_{n+1}(s)-y_n(s)\|^{2}_{\rV}\nonumber\\
   &&
   +
   2{\eps}^{4/3}\|y_{n-1}(s)+M(s)\|^{2}_{\rV}\cdot \|y_n-y_{n-1}\|^2_{\Upsilon_s(\rH)}+2{\eps}^{4/3}\|y_n(s)-y_{n-1}(s)\|^{2}_{\rV}.
\end{eqnarray}
In the last inequality ``$\leq$" of (\ref{app case 1.3 01}), we used inequality (\ref{eq subcase 1.2 01}).

\vskip 0.2cm
By inequalities (\ref{Eq Subcase 1.2}) and (\ref{app case 1.3 01}), we get, similar to inequality (\ref{Eq Claim 1.1}), the following:
\begin{align}\label{Eq Claim 1.2}
&\hspace{-1.3truecm}\lefteqn{\int_0^t|I(s)|\mathds{1}_{\{\|y_n+M\|_{\xis}\vee\|y_{n-1}+M\|_{\xis}\leq \frac{3}{\delta}\}}
    \mathds{1}_{\{\|y_n+M\|_{\Upsilon_s(\rH)}> m+2\}}\mathds{1}_{\{\|y_{n-1}+M\|_{\Upsilon_s(\rH)}\leq m+2\}}\,ds}\nonumber\\
&\hspace{-1.3truecm}\leq
  \Big(\frac{{\eps}^{4/3}}{2}+\frac{9(m+2)^2}{2{\eps}^4\delta^2} \Big)\|y_{n+1}-y_{n}\|^{2}_{\Upsilon_t(\rH)}
  +
  2{\eps}^{4/3}\Big(1+\frac{9}{\delta^2}\Big)\|y_n-y_{n-1}\|^{2}_{\Upsilon_t(\rH)}.
\end{align}

\noindent{\bf Subcase 1.3} Assume further that $\|y_n+M\|_{\Upsilon_s(\rH)}\leq m+2$ and $\|y_{n-1}+M\|_{\Upsilon_s(\rH)}\leq m+2$.

Under the assumptions of {\bf Subcase 1.3}, we infer that  for any ${\eps}>0$,
\begin{eqnarray}\label{app I I1 1.1}
   \Big|I_1(s)\Big|
&\leq&
   \Big| \dual{\Big\langle \rB(y_n(s)-y_{n-1}(s),y_n(s)+M(s)),y_{n+1}(s)-y_n(s)\Big\rangle}
   \Big|\nonumber\\
&\leq&
   2\|y_n(s)-y_{n-1}(s)\|^{\frac{1}{2}}_{\rV}|y_n(s)-y_{n-1}(s)|^{\frac{1}{2}}_{\rH}\|y_{n+1}(s)-y_{n}(s)\|^{\frac{1}{2}}_{\rV}\nonumber\\
   &&\cdot|y_{n+1}(s)-y_{n}(s)|^{\frac{1}{2}}_{\rH}\|y_n(s)+M(s)\|_{\rV}\nonumber\\
&\leq&
   {\eps}\|y_n(s)-y_{n-1}(s)\|_{\rV}\|y_{n+1}(s)-y_{n}(s)\|_{\rV}\nonumber\\
   &&+
   \frac{2}{{\eps}}|y_n(s)-y_{n-1}(s)|_{\rH}|y_{n+1}(s)-y_{n}(s)|_{\rH}\|y_n(s)+M(s)\|^2_{\rV}\nonumber\\
&\leq&
   {\eps}\|y_n(s)-y_{n-1}(s)\|^2_{\rV}+{\eps}\|y_{n+1}(s)-y_{n}(s)\|^2_{\rV}\nonumber\\
   &&+
   {\eps} \|y_n(s)+M(s)\|^2_{\rV}|y_n(s)-y_{n-1}(s)|^2_{\rH}\nonumber\\
   &&+
   \frac{4}{{\eps}^3}|y_{n+1}(s)-y_{n}(s)|^2_{\rH}\|y_n(s)+M(s)\|^2_{\rV}.
\end{eqnarray}

For $I_2(s)=I_{2,1}(s)+I_{2,1}(s)$, where
\begin{eqnarray*}
I_{2,1}(s)&=&\Big(\Xi(y_n,y_n)(s)-\Xi(y_{n-1},y_{n})(s)\Big)
\\
&& \dual{\Big\langle \rB(y_{n-1}(s)+M(s),y_n(s)+M(s)),y_{n+1}(s)-y_n(s)\Big\rangle}
\end{eqnarray*}
and
\begin{eqnarray*}
I_{2,2}(s)&=&\Big(\Xi(y_{n-1},y_n)(s)-\Xi(y_{n-1},y_{n-1})(s)\Big)
\\&&
\dual{\Big\langle \rB(y_{n-1}(s)+M(s),y_n(s)+M(s)),y_{n+1}(s)-y_n(s)\Big\rangle}
,
\end{eqnarray*}
we have, for any ${\eps}>0$,
\begin{eqnarray}\label{App I I2 1 01}
 \hspace{-1.2truecm}\Big|I_{2,1}(s)\Big|
&\leq&
   C_1\|y_n-y_{n-1}\|_{\Upsilon_s(\rH)}|y_{n-1}(s)+M(s)|^{\frac{1}{2}}_{\rH}\|y_{n-1}(s)+M(s)\|^{\frac{1}{2}}_{\rV}\nonumber\\
                  &&\cdot |y_{n+1}(s)-y_n(s)|^{\frac{1}{2}}_{\rH}\|y_{n+1}(s)-y_n(s)\|^{\frac{1}{2}}_{\rV}
                   \|y_{n}(s)+M(s)\|_{\rV}\nonumber\\
&\leq&
   3/4{\eps}^{4/3}\|y_n-y_{n-1}\|^{4/3}_{\Upsilon_s(\rH)}\|y_{n}(s)+M(s)\|^{4/3}_{\rV}\|y_{n+1}(s)-y_n(s)\|^{2/3}_{\rV}\nonumber\\
   &&+
   \frac{C_1^4}{4{\eps}^4}|y_{n-1}(s)+M(s)|^{2}_{\rH}\|y_{n-1}(s)+M(s)\|^{2}_{\rV}|y_{n+1}(s)-y_n(s)|^{2}_{\rH}\nonumber\\
&\leq&
   \frac{1}{4}{\eps}^{4/3}\Big[2\|y_n-y_{n-1}\|^{2}_{\Upsilon_s(\rH)}\|y_{n}(s)+M(s)\|^{2}_{\rV}+\|y_{n+1}(s)-y_n(s)\|^{2}_{\rV}\Big]\nonumber\\
   &&+
   \frac{C_1^4}{4{\eps}^4}(m+2)^2\|y_{n-1}(s)+M(s)\|^{2}_{\rV}|y_{n+1}(s)-y_n(s)|^{2}_{\rH},
\end{eqnarray}
and
\begin{eqnarray}\label{App I I2 2 01}
  \Big|I_{2,2}(s)\Big|
&\leq&
   {\Cphi}\delta\|y_n-y_{n-1}\|_{\xis}|y_{n-1}(s)+M(s)|^{\frac{1}{2}}_{\rH}\|y_{n-1}(s)+M(s)\|^{\frac{1}{2}}_{\rV}\nonumber\\
                  &&\cdot |y_{n+1}(s)-y_n(s)|^{\frac{1}{2}}_{\rH}\|y_{n+1}(s)-y_n(s)\|^{\frac{1}{2}}_{\rV}
                   \|y_{n}(s)+M(s)\|_{\rV}\nonumber\\
&\leq& \frac{3\delta}{4}{\eps}^{4/3}\|y_n-y_{n-1}\|^{4/3}_{\xis}\|y_{n}(s)+M(s)\|^{4/3}_{\rV}\|y_{n+1}(s)-y_n(s)\|^{2/3}_{\rV}\nonumber\\
    &&+
    \frac{{\Cphi}^4\delta}{4{\eps}^4}|y_{n-1}(s)+M(s)|^{2}_{\rH}\|y_{n-1}(s)+M(s)\|^{2}_{\rV}|y_{n+1}(s)-y_n(s)|^{2}_{\rH}\nonumber\\
&\leq&
   \frac{1}{2}\delta^{\frac{3}{2}}{\eps}^{4/3}\|y_n-y_{n-1}\|^{2}_{\xis}\|y_{n}(s)+M(s)\|^{2}_{\rV}+\frac{1}{4}{\eps}^{4/3}\|y_{n+1}(s)-y_n(s)\|^{2}_{\rV}\nonumber\\
    &&+
    \frac{{\Cphi}^4\delta}{4{\eps}^4}(m+2)^2\|y_{n-1}(s)+M(s)\|^{2}_{\rV}|y_{n+1}(s)-y_n(s)|^{2}_{\rH}.
\end{eqnarray}
\vskip 0.2cm

Similar to inequality (\ref{Eq Claim 1.1}), by inequalities (\ref{app I I1 1.1})-(\ref{App I I2 2 01}), we infer that
\begin{align}\label{Eq Claim 1.3}
&\hspace{-0.2truecm}\int_0^t|I(s)|\mathds{1}_{\{\|y_n+M\|_{\xis}\vee\|y_{n-1}+M\|_{\xis}\leq \frac{3}{\delta}\}}
    \mathds{1}_{\{\|y_n+M\|_{\Upsilon_s(\rH)}\leq m+2\}}\mathds{1}_{\{\|y_{n-1}+M\|_{\Upsilon_s(\rH)}\leq m+2\}}\,ds\nonumber\\
&\hspace{-0.1truecm}\leq
  \Big({\eps}+\frac{36}{{\eps}^3\delta^2}+\frac{1}{2}{\eps}^{4/3}+\frac{C_1^4}{4{\eps}^4}(m+2)^2\frac{9}{\delta^2}
     +\frac{9}{4}\frac{(m+2)^2}{{\eps}^4\delta}
     \Big)\|y_{n+1}-y_{n}\|^{2}_{\Upsilon_t(\rH)}\nonumber\\
  &\ \ +
  \Big({\eps}+\frac{9{\eps}}{\delta^{2}}+\frac{9 {\eps}^{4/3}}{2\delta^2}+\frac{9{\eps}^{4/3}}{2\delta^{\frac{1}{2}}}\Big)\|y_n-y_{n-1}\|^{2}_{\Upsilon_t(\rH)}.
\end{align}

The proof of \textbf{Case 1} is complete.
\end{proof}

\begin{proof}[{\bf Case 2.}] Assume that $\|y_n+M\|_{\xis}\leq \frac{3}{\delta}$ and $\|y_{n-1}+M\|_{\xis}> \frac{3}{\delta}$.

In this case, by the definitions of functions $\theta_m$ and $\phi_\delta$, we have
$$
I(s)=\dual{\big\langle \Pi(y_n,y_n,y_n,y_{n+1})(s),y_{n+1}(s)-y_n(s)\big\rangle}
,
$$
and
\begin{eqnarray}\label{Eq Case 2}
&&\hspace{-1truecm}\lefteqn{\int_0^t|I(s)|\mathds{1}_{\{\|y_n+M\|_{\xis}\leq \frac{3}{\delta}\}}(s)
    \mathds{1}_{\{\|y_{n-1}+M\|_{\xis}> \frac{3}{\delta}\}}(s)\,ds}\\
&=&
 {\int_0^t|I(s)|\mathds{1}_{\{\|y_n+M\|_{\xis}\leq \frac{2}{\delta}\}}(s)
    \mathds{1}_{\{\|y_{n-1}+M\|_{\xis}> \frac{3}{\delta}\}}(s)\mathds{1}_{\{\|y_n+M\|_{\Upsilon_s(\rH)}\leq m+1\}}\,ds}\nonumber.
\end{eqnarray}

For any $s\in[0,t]$ such that
$$
\mathds{1}_{\{\|y_n+M\|_{\xis}\leq \frac{2}{\delta}\}}(s)
    \mathds{1}_{\{\|y_{n-1}+M\|_{\xis}> \frac{3}{\delta}\}}(s)\mathds{1}_{\{\|y_n+M\|_{\Upsilon_s(\rH)}\leq m+1\}}=1,
$$
 we have
 \begin{eqnarray}\label{Eq Case 2 00}
 \delta\|y_n-y_{n-1}\|_{\Upsilon_s(\rH)}\geq \delta\|(y_n+M)-(y_{n-1}+M)\|_{\xis}\geq 1,
 \end{eqnarray}
 and for any $p\in(0,\frac{1}{2})$ and $\epsilon>0$,
\begin{eqnarray}\label{app case 2}
  \Big|I(s)\Big|
&\leq&
   \Big|
   \dual{\Big\langle \rB(y_n(s)+M(s),y_{n+1}(s)+M(s)),-y_n(s)-M(s)\Big\rangle}
   \Big|\nonumber\\
&=&
   \Big|\dual{\Big\langle \rB(y_n(s)+M(s),y_{n+1}(s)-y_n(s)),y_n(s)+M(s)\Big\rangle}
   \Big|\nonumber\\
&\leq&
   2|y_n(s)+M(s)|^{\frac{1}{2}}_{\rH}\|y_n(s)+M(s)\|^{\frac{1}{2}}_{\rV}|y_{n+1}(s)-y_n(s)|^{\frac{1}{2}}_{\rH}\nonumber\\
   &&\cdot\|y_{n+1}(s)-y_n(s)\|^{\frac{1}{2}}_{\rV}\|y_n(s)+M(s)\|_{\rV}\cdot \delta\|y_n-y_{n-1}\|_{\Upsilon_s(\rH)}\nonumber\\
&\leq&
   \frac{3}{2}{\eps}^{4/3}\delta^{\frac{4(1-p)}{3}}\|y_{n+1}(s)-y_n(s)\|^{2/3}_{\rV}\|y_n(s)+M(s)\|^{4/3}_{\rV}\|y_n-y_{n-1}\|^{4/3}_{\Upsilon_s(\rH)}\nonumber\\
   &&+
   \frac{1}{2{\eps}^4}\delta^{4p}\|y_n(s)+M(s)\|^{2}_{\rV}|y_n(s)+M(s)|^{2}_{\rH}|y_{n+1}(s)-y_n(s)|^{2}_{\rH}\nonumber\\
&\leq&
   \frac{1}{2}{\eps}^{4/3}\|y_{n+1}(s)-y_n(s)\|^{2}_{\rV}+{\eps}^{4/3}\delta^{2(1-p)}\|y_n(s)+M(s)\|^{2}_{\rV}\|y_n-y_{n-1}\|^{2}_{\Upsilon_s(\rH)}\nonumber\\
   &&+
   \frac{1}{2{\eps}^4}\delta^{4p}(m+1)^2\|y_n(s)+M(s)\|^{2}_{\rV}|y_{n+1}(s)-y_n(s)|^{2}_{\rH}.
\end{eqnarray}
In the second ``$\leq$" of (\ref{app case 2}), we have used (\ref{Eq Case 2 00}).


\vskip 0.2cm
 Similar to (\ref{Eq Claim 1.1}), by (\ref{Eq Case 2}) and (\ref{app case 2}),
\begin{eqnarray}\label{Eq Claim 2}
&&\hspace{-1truecm}\lefteqn{\int_0^t|I(s)|\mathds{1}_{\{\|y_n+M\|_{\xis}\leq \frac{3}{\delta}\}}(s)
    \mathds{1}_{\{\|y_{n-1}+M\|_{\xis}> \frac{3}{\delta}\}}(s)\,ds}\nonumber\\
&\leq&
  \Big(\frac{1}{2}{\eps}^{4/3}+\frac{9}{2{\eps}^4}(m+1)^2\delta^{4p-2}
     \Big)\|y_{n+1}-y_{n}\|^{2}_{\Upsilon_t(\rH)}
  +
  9\frac{{\eps}^{4/3}}{\delta^{2p}}\|y_n-y_{n-1}\|^{2}_{\Upsilon_t(\rH)}.
\end{eqnarray}
The proof of \textbf{Case 2} is complete.
\end{proof}

\begin{proof}[{\bf Case 3.}] Assume that $\|y_n+M\|_{\xis}> \frac{3}{\delta}$ and $\|y_{n-1}+M\|_{\xis}\leq \frac{3}{\delta}$.

In this case, similar to the \textbf{Case 2}, by the definitions of functions $\theta_m$ and $\phi_\delta$, we have
$$
I(s)=-\dual{ \big\langle \Pi(y_{n-1},y_{n-1},y_{n-1},y_{n})(s),y_{n+1}(s)-y_n(s)\big\rangle}
,
$$
and
\begin{eqnarray}\label{Eq Claim 3 00}
&&\hspace{-0.8truecm}\lefteqn{\int_0^t|I(s)|\mathds{1}_{\{\|y_n+M\|_{\xis}> \frac{3}{\delta}\}}(s)
    \mathds{1}_{\{\|y_{n-1}+M\|_{\xis}\leq \frac{3}{\delta}\}}(s)\,ds}\\
&=&
  {\int_0^t|I(s)|\mathds{1}_{\{\|y_n+M\|_{\xis}> \frac{3}{\delta}\}}(s)
    \mathds{1}_{\{\|y_{n-1}+M\|_{\xis}\leq \frac{2}{\delta}\}}(s)\mathds{1}_{\{\|y_{n-1}+M\|_{\Upsilon_s(\rH)}\leq m+1\}}\,ds}.\nonumber
\end{eqnarray}

For any $s\in[0,t]$ such that
$$
\mathds{1}_{\{\|y_n+M\|_{\xis}> \frac{3}{\delta}\}}(s)
    \mathds{1}_{\{\|y_{n-1}+M\|_{\xis}\leq \frac{2}{\delta}\}}(s)\mathds{1}_{\{\|y_{n-1}+M\|_{\Upsilon_s(\rH)}\leq m+1\}}=1,
$$
we have
\begin{eqnarray}\label{Eq Claim3 001}
\delta^2\|y_n-y_{n-1}\|^2_{\Upsilon_s(\rH)}\geq \delta^2\|(y_n+M)-(y_{n-1}+M)\|^2_{\xis}\geq 1
\end{eqnarray}
and, for any $p\in(0,\frac{1}{2})$ and ${\eps}>0$,
\begin{eqnarray}\label{app case 3}
  \Big|I(s)\Big|
&\leq&
   \Big|
   \dual{\Big\langle \rB(y_{n-1}(s)+M(s),y_{n}(s)+M(s)),y_{n+1}(s)-y_n(s)\Big\rangle}
   \Big|\nonumber\\
&\leq&
   2|y_{n-1}(s)+M(s)|^{\frac{1}{2}}_{\rH}\|y_{n-1}(s)+M(s)\|^{\frac{1}{2}}_{\rV}|y_{n+1}(s)-y_n(s)|^{\frac{1}{2}}_{\rH}\nonumber\\
   &&\cdot\|y_{n+1}(s)-y_n(s)\|^{\frac{1}{2}}_{\rV}\|y_n(s)+M(s)\|_{\rV}\nonumber\\
&\leq&
   \frac{3}{2}\delta^{\frac{-4p}{3}}\|y_{n+1}(s)-y_n(s)\|^{2/3}_{\rV}\|y_n(s)+M(s)\|^{4/3}_{\rV}\nonumber\\
   &&+
   \frac{1}{2}\delta^{4p}\|y_{n-1}(s)+M(s)\|^{2}_{\rV}|y_{n-1}(s)+M(s)|^{2}_{\rH}|y_{n+1}(s)-y_n(s)|^{2}_{\rH}\nonumber\\
&\leq&
   \frac{1}{2}\delta^{4p}(m+1)^2\|y_{n-1}(s)+M(s)\|^{2}_{\rV}|y_{n+1}(s)-y_n(s)|^{2}_{\rH}\nonumber\\
   &&+
   \frac{1}{2}{\eps}^3\|y_{n+1}(s)-y_n(s)\|^{2}_{\rV}
   +
   \frac{1}{{\eps}^{\frac{3}{2}}\delta^{2p}}\|y_n(s)+M(s)\|^{2}_{\rV}\nonumber\\
&\leq&
   \frac{1}{2}\delta^{4p}(m+1)^2\|y_{n-1}(s)+M(s)\|^{2}_{\rV}|y_{n+1}(s)-y_n(s)|^{2}_{\rH}\\
   &&+
   \frac{1}{2}{\eps}^3\|y_{n+1}(s)-y_n(s)\|^{2}_{\rV}
   +
   2\frac{1}{{\eps}^{\frac{3}{2}}\delta^{2p}}\Big[\|y_{n}(s)-y_{n-1}(s)\|^{2}_{\rV}+\|y_{n-1}(s)+M(s)\|^{2}_{\rV}\Big]\nonumber\\
&\leq&
   \frac{1}{2}\delta^{4p}(m+1)^2\|y_{n-1}(s)+M(s)\|^{2}_{\rV}|y_{n+1}(s)-y_n(s)|^{2}_{\rH}
   +
   \frac{1}{2}{\eps}^3\|y_{n+1}(s)-y_n(s)\|^{2}_{\rV}\nonumber\\
   &&
   +
   2\frac{1}{{\eps}^{\frac{3}{2}}\delta^{2p}}\|y_{n}(s)-y_{n-1}(s)\|^{2}_{\rV}+2\frac{1}{{\eps}^{\frac{3}{2}}\delta^{2p}}\|y_{n-1}(s)+M(s)\|^{2}_{\rV}\cdot \|y_{n}-y_{n-1}\|^{2}_{\Upsilon_s(\rH)}\delta^2.\nonumber
\end{eqnarray}
In the last ``$\leq$" in (\ref{app case 3}), we have used (\ref{Eq Claim3 001}).

\vskip 0.2cm
 Similarly to (\ref{Eq Claim 1.1}), by (\ref{Eq Claim 3 00}) and (\ref{app case 3}),
\begin{eqnarray}\label{Eq Claim 3}
&&\lefteqn{\int_0^t|I(s)|\mathds{1}_{\{\|y_n+M\|_{\xis}> \frac{3}{\delta}\}}
    \mathds{1}_{\{\|y_{n-1}+M\|_{\xis}\leq \frac{3}{\delta}\}}\,ds}\\
&\leq&
  \Big(\frac{1}{2}{\eps}^{3}+{2}(m+1)^2\delta^{4p-2}
     \Big)\|y_{n+1}-y_{n}\|^{2}_{\Upsilon_t(\rH)}
  +
  \frac{2}{\eps^{\frac{3}{2}}\delta^{2p}}\|y_n-y_{n-1}\|^{2}_{\Upsilon_t(\rH)}.\nonumber
\end{eqnarray}

{The proof of \textbf{Case 3 }is complete.}
\end{proof}
\vskip 0.2cm

Combining (\ref{Eq Claim 1.1}), (\ref{Eq Claim 1.2}), (\ref{Eq Claim 1.3}), (\ref{Eq Claim 2}), and (\ref{Eq Claim 3}), there exist constants $C>0$ and $l_m>0$, for any ${\eps}>0$ and $p\in(0,\frac{1}{2})$, such that
\begin{eqnarray}\label{Eq esta IS P}
&&\hspace{-1truecm}\int_0^t|I(s)|\,ds\\
&\leq&
  l_m\Big({\eps}+\frac{1}{{\eps}^{3}\delta^{2}}+{\eps}^{4/3}+\frac{1}{{\eps}^{4}\delta^{2}}+\frac{1}{{\eps}^{4}\delta}+
  {\eps}^{3}+\delta^{4p-2}+{\eps}^{-4}\delta^{4p-2}
     \Big)\|y_{n+1}-y_{n}\|^{2}_{\Upsilon_t(\rH)}\nonumber\\
  && +\nonumber
  C\Big({\eps}+\frac{{\eps}}{\delta^{2}}+\frac{{\eps}^{4/3}}{\delta^{2}}+\frac{{\eps}^{4/3}}{\delta^{\frac{1}{2}}}
  +{\eps}^{4/3}+\frac{{\eps}^{4/3}}{\delta^{2p}}+\frac{1}{{\eps}^{\frac{3}{2}}\delta^{2p}}\Big)\|y_n-y_{n-1}\|^{2}_{\Upsilon_t(\rH)}.
\end{eqnarray}
Choosing $p=\frac{1}{4}$. Let ${\eps}$ small enough first, and then $\delta$ large enough, there exist ${\eps}_0>0$ and $\delta_0>0$ such that
\begin{eqnarray}\label{Eq esta IS P 0}
  && l_m\Big({{\eps}_0}+\frac{1}{{{\eps}_0}^{3}{\delta_0}^{2}}+{{\eps}_0}^{4/3}+\frac{1}{{{\eps}_0}^{4}{\delta_0}^{2}}
  +\frac{1}{{{\eps}_0}^{4}{\delta_0}}+{{\eps}_0}^{3}+{\delta_0}^{4p-2}+\frac{{\delta_0}^{4p-2}}{{{\eps}_0}^{4}}
     \Big)\\\nonumber
  && +
  C\Big({{\eps}_0}+\frac{{{\eps}_0}}{{\delta_0}^{2}}+\frac{{{\eps}_0}^{4/3}}{{\delta_0}^{2}}
  +\frac{{{\eps}_0}^{4/3}}{{\delta_0}^{\frac{1}{2}}}+{{\eps}_0}^{4/3}+\frac{{{\eps}_0}^{4/3}}{{\delta_0}^{2p}}
  \frac{1}{+{{\eps}_0}^{\frac{3}{2}}{\delta_0}^{2p}}\Big)
  \leq \frac{1}{16}.
\end{eqnarray}

Set $\delta=\delta_0$ in (\ref{App auxi spde 01}).
By (\ref{App auxi esti 01}), (\ref{Eq esta IS P}), and (\ref{Eq esta IS P 0}), we arrive at
\begin{eqnarray}\label{App auxi esti 02}
\sup_{s\in[0,t]}|y_{n+1}(s)-y_n(s)|^2_{\rH}\!\!\!\!&+\!\!\!\!&2\int_0^t\|y_{n+1}(s)-y_n(s)\|^2_{\rV}\,ds
\\ \!\!\!\!&\ \leq&\!
  \frac{1}{8}\Big[\|y_{n+1}-y_n\|^2_{\Upsilon_t(\rH)}+\|y_{n}-y_{n-1}\|^2_{\Upsilon_t(\rH)}\Big].
  \nonumber
\end{eqnarray}
Since
$$
\frac{1}{2}\|y_{n+1}-y_n\|^2_{\Upsilon_t(\rH)}\leq \sup_{s\in[0,t]}|y_{n+1}(s)-y_n(s)|^2_{\rH}+\int_0^t\|y_{n+1}(s)-y_n(s)\|^2_{\rV}\,ds,
$$
by (\ref{App auxi esti 02}) {we infer that}
\begin{eqnarray}\label{App auxi esti 03}
\|y_{n+1}-y_n\|^2_{\Upsilon_T(\rH)}
\leq
  \frac{1}{3}\|y_{n}-y_{n-1}\|^2_{\Upsilon_T(\rH)},
\end{eqnarray}
which implies that $\{y_{n},n\in\mathbb{N}\}$ is a Cauchy sequence in $C([0,T],\rH)\cap L^2([0,T],\rV)$; we denote its limit by $Y^1$.
Using classical arguments, it is not difficult to prove that $Y^1$ is a solution of Problem $(\ref{App auxi spde 01})$ with $\delta=\delta_0$.


\end{proof}

\begin{proof}[{\bf Step 2.}]
Let $\delta_0$ be as in \textbf{Step 1}, and set  \[ t_1:=\inf \{ t\in [0,T] : \|Y^1+M\|_{\xit}\geq\frac{1}{\delta_0}\}.\]
Since by \eqref{eqn-phi_delta}, $\phi_\delta(r)=1$ if $r\in [0,\frac1\delta]$, it is easy to show  that $Y^1$ is a solution of Problem (\ref{eq lemma 1 01}) on time interval $[0,t_1]$. {If $t_1=T$ then the proof of Lemma \ref{lem 1} is finished. Otherwise,} let us consider the following {deterministic time-inhomogeneous} evolution equation
\begin{eqnarray}\label{App auxi spde 03}
&& X^\prime(t)+{\rA}X(t)=f(t)
\\
&&\ \ \ \ \ \ \ \ \ \ \ \ \ \ \ \ \ -\theta_m(\|X+M\|_{\Upsilon_t(\rH)})\phi_{\delta_0}\Big(\|X+M\|_{L^2([t_1,t],\rV)}\Big)\rB(X(t)+M(t)),\ t>t_1,\nonumber\\
&& X(t)=Y^1(t), \ \ t\in[0,t_1].
\nonumber
\end{eqnarray}
Using a similar argument to that in \textbf{Step 1}, we can find a solution $Y^2$ to Problem (\ref{App auxi spde 03}). As in the beginning of this step, we set
$$
t_2:=\inf\{t\in [t_1,T]: \|Y^2+M\|_{L^2([t_1,t],\rV)}\geq\frac{1}{\delta_0}\}.
$$
and see that $Y^2$ is a solution of problem (\ref{eq lemma 1 01}) on the time interval $[0,t_2]$. {If $t_2=T$ then the proof of Lemma \ref{lem 1} is finished. Otherwise,} by induction, we construct two sequences $\{ t_n\}_{n\in \mathbb{N}}$ and $\{Y^n\}_{n\in \mathbb{N}}$ satisfying
\begin{itemize}
 \item $0<t_1<\cdots<t_{n}<t_{n+1}<\cdots$,
 \item $Y^n\in C([0,T],\rH)\cap L^2([0,T],\rV)$ and $Y^{n+1}(t)=Y^n(t)$ on $t\in[0,t_n]$,
 \item $Y^n$ is a solution of Problem (\ref{eq lemma 1 01}) for $t\in[0,t_n]$,
 \item $t_{n+1}:=\inf\{ t \in [t_n,T]: \{\|Y^{n+1}+M\|_{L^2([t_n,t],\rV)}\geq\frac{1}{\delta_0}\}$.
\end{itemize}
{The proof of Lemma \ref{lem 1} is concluded once we prove that for some $n\in \mathbb{N}$, $t_n=T$. This is done in the next step. }
\end{proof}

\begin{proof}[{\bf Step 3.}] Assume that $X\in C([0,\tau],\rH)\cap L^2([0,\tau],\rV)$, for some $\tau>0$, is a solution of the deterministic Problem (\ref{eq lemma 1 01}). By the Lions-Magenes lemma (\cite{Lions+Magenes_1972}; \cite[Lemma III.1.2]{Temam_2001}), we have this:
For every $t\in [0,\tau]$,
\begin{eqnarray*}
&&\hspace{-1truecm}|X(t)|_{\rH}^2+2\int_0^t\|X(s)\|^2_{\rV}\,ds\\
&&\hspace{-1truecm}=
  |u_0|_{\rH}^2-2\int_0^t\theta_m(\|X+M\|_{\Upsilon_s(\rH)})\langle \rB(X(s)+M(s)), X(s)\rangle_{\rV^\prime,\rV}\,ds+2\int_0^t\langle f(s),X(s)\rangle_{\rV^\prime,\rV}\,ds\\
&&\hspace{-1truecm}\leq
  |u_0|_{\rH}^2+\!\int_0^t\!\|X(s)\|^2_{\rV}\,ds\!+\!8\int_0^t\theta^2_m(\|X+M\|_{\Upsilon_s(\rH)})\|\rB(X(s)\!+\!M(s))\|^2_{\rV^\prime}\,ds\!+\!8\int_0^t\!\|f(s)\|^2_{\rV^\prime}\,ds\\
&&\hspace{-1truecm}\leq
  |u_0|_{\rH}^2+\int_0^t\|X(s)\|^2_{\rV}\,ds+8\int_0^t\theta^2_m(\|X+M\|_{\Upsilon_s(\rH)})|X(s)+M(s)|^2_{\rH}\|X(s)+M(s)\|^2_{V}\,ds\\
  &&\hspace{-0.5truecm}+8\int_0^t\|f(s)\|^2_{\rV^\prime}\,ds\\
&&\hspace{-1truecm}\leq
  |u_0|_{\rH}^2+\int_0^t\|X(s)\|^2_{\rV}\,ds+8(m+1)^4+8\int_0^\tau\|f(s)\|^2_{\rV^\prime}\,ds.
\end{eqnarray*}
Hence,
\begin{eqnarray*}
\sup_{t\in[0,\tau]}|X(t)|_{\rH}^2+\int_0^\tau\|X(s)\|^2_{\rV}\,ds
\leq
  |u_0|_{\rH}^2+8(m+1)^4+8\int_0^\tau\|f(s)\|^2_{\rV^\prime}\,ds.
\end{eqnarray*}
{This implies that there exists $n\in \mathbb{N}$: $t_n=T$. Therefore, the function $Y^n$ is  the solution sought in Lemma \ref{lem 1}. }
\end{proof}

\vskip 0.2cm
The proof of Lemma \ref{lem 1} is complete.
\end{proof}

\vskip 0.2cm

The following lemma implies that the solution of Problem (\ref{eq lemma 1 01}) is unique (see Corollary \ref{cor 1}); this lemma will be used later.
Recall that the space $\Lambda_T(\rH)$ (and its norm) was defined around equality \eqref{eqn-Lambda_T-H}.

\begin{lemma}\label{lem 2}
Assume that  $m\in\mathbb{N}$. Assume that for all $u_0\in \rH$ and $f \in L^2([0,T];\rV^\prime)$ and  $y\in \Lambda_T(\rH)$, there exists {an} element $u=\Phi^y\in \Lambda_T(\rH)$ satisfying
\begin{eqnarray}\label{eq SNS jieduan 2}
&& du(t)+{\rA}u(t)\,dt+\theta_m(\|u\|_{\Upsilon_t(\rH)})\rB(u(t))\,dt=f(t)\,dt + \int_{{\rZ}}G(y(t-),z)\widetilde{\eta}(dz,dt),\nonumber\\
&& u(0)=u_0.
\end{eqnarray}
Then there exists a constant $C_m>0$ such that
\begin{eqnarray}\label{eq BPFT 01}
\|\Phi^{y_1}-\Phi^{y_2}\|_{\Lambda_T(\rH)}^2\leq C_mT\|y_1-y_2\|^2_{\Lambda_T(\rH)},\ \ \ \forall y_1, y_2\in \Lambda_T(\rH).
\end{eqnarray}
\todozbdel{Lemma \ref{lem 2} involves $G$ in an essential way and I think assumption  \eqref{eqn-Lipschitz-H} is needed and it is not sufficient to assume that \eqref{eqn-Lipschitz-H-local}.
I think that in order to prove inequality similar to \eqref{eq BPFT 01} we would need to modify the ``noisy" coefficient by multiplying it by a something like we used to multiple the nonlinearity $\rB(u)$, i.e., by the $\theta_m(\|y\|_{\Upsilon_t(\rH)})$, which depends on $y$.
 }
\end{lemma}
{\textbf{Remark.} The above result is not true without the smoothing function $\theta_m$.}

\begin{proof}[Proof of Lemma \ref{lem 2}]
\vskip 0.3cm
For simplicity, define $u_1=\Phi^{y_1}$ and $u_2=\Phi^{y_2}$.
Set $U=u_1-u_2$. By {the It\^o} formula, we have
\begin{eqnarray}\label{eq Z 0}
&&\hspace{-1truecm}|U(t)|^2_{\rH}+2\int_0^t\|U(s)\|^2_{\rV}\,ds\nonumber\\
&=&
 -2\int_0^t\dual{\Big{\langle \theta_m(\|u_1\|_{\Upsilon_s(\rH)}) \rB(u_1(s))-\theta_m(\|u_2\|_{\Upsilon^H_2}) \rB(u_2(s)),U(s)\Big\rangle}}\,ds\nonumber\\
 &&+
 2\int_0^t\int_{{\rZ}}\Big\langle G(y_1(s-),z)-G(y_2(s-),z),U(s-)\Big\rangle_{\rH} \widetilde{\eta}(dz,ds)\nonumber\\
 &&+
 \int_0^t\int_{{\rZ}}|G(y_1(s-),z)-G(y_2(s-),z)|^2_{\rH} \eta(dz,ds)\nonumber\\
&=:&
 J_1(t)+J_2(t)+ J_3(t), {\;\; t \in [0,T]}.
\end{eqnarray}

Concerning $J_1$, we have
\begin{eqnarray}\label{eq J1}
 |J_1(t)| &\leq&
 \frac{1}{2} \int_0^t\|U(s)\|^2_{\rV}\,ds\\
 &&+
 2 \int_0^t\Big\|\theta_m(\|u_1\|_{\Upsilon_s(\rH)}) \rB(u_1(s))-\theta_m(\|u_2\|_{\Upsilon_s(\rH)})\rB(u_2(s))\Big\|^2_{\rV^\prime}\,ds.\nonumber
\end{eqnarray}

Set
$$
K(s):=\Big\|\theta_m(\|u_1\|_{\Upsilon_s(\rH)}) \rB(u_1(s))-\theta_m(\|u_2\|_{\Upsilon_s(\rH)})\rB(u_2(s))\Big\|^2_{\rV^\prime}, \;\; s\in [0,T].
$$
We distinguish four cases to find appropriate bounds for $K$. By the property of $\theta_m$ and the Minkowski inequality, we have the following estimates. {Let us fix $s\in [0,T]$.}
\begin{itemize}
\item[(1)] {Assume that} $\|u_1\|_{\Upsilon_s(\rH)}\vee\|u_2\|_{\Upsilon_s(\rH)}\leq m+1$. {In this case, we have}
  \begin{eqnarray*}
      K(s)
   &\leq&
      C \Big[
        \|\rB(u_1(s))-\rB(u_2(s))\|^2_{\rV^\prime}
      +
       \Big|\theta_m(\|u_1\|_{\Upsilon_s(\rH)})-\theta_m(\|u_2\|_{\Upsilon_s(\rH)})\Big|^2\|\rB(u_2(s))\|^2_{\rV^\prime}
       \Big]\\
   &\leq&
     C |U(s)|_{\rH}\|U(s)\|_{\rV}\Big[
       |u_1(s)|_{\rH}\|u_1(s)\|_{\rV}+|u_2(s)|_{\rH}\|u_2(s)\|_{\rV}
       \Big]\\
       &&+
     C|u_2(s)|^2_{\rH}\|u_2(s)\|^2_{\rV}\|U\|^2_{\Upsilon_s(\rH)}\\
   &\leq&
     \frac{1}{4} \|U(s)\|^2_{\rV}
     +
     C\|U\|^2_{\Upsilon_s(\rH)}
       \Big[
       |u_1(s)|^2_{\rH}\|u_1(s)\|^2_{\rV}+|u_2(s)|^2_{\rH}\|u_2(s)\|^2_{\rV}
       \Big].
  \end{eqnarray*}

\item[(2)] {Assume that} $\|u_1\|_{\Upsilon_s(\rH)}\leq m+1$ and $\|u_2\|_{\Upsilon_s(\rH)}\geq m+1$. {In this case, we have}
  \begin{eqnarray*}
      K(s)&=&
          \Big\|\theta_m(\|u_1\|_{\Upsilon_s(\rH)}) \rB(u_1(s))\Big\|^2_{\rV^\prime}\\
        &=&
          \Big|\theta_m(\|u_1\|_{\Upsilon_s(\rH)})-\theta_m(\|u_2\|_{\Upsilon_s(\rH)})\Big|^2\| \rB(u_1(s))\|^2_{\rV^\prime}\\
        &\leq&
          C|u_1(s)|^2_{\rH}\|u_1(s)\|^2_{\rV}\|U\|^2_{\Upsilon_s(\rH)}.
  \end{eqnarray*}

\item[(3)] {Assume that} $\|u_1\|_{\Upsilon_s(\rH)}\geq m+1$ and $\|u_2\|_{\Upsilon_s(\rH)}\leq m+1$ . In this case, as in the similar Case (2), we get
  \begin{eqnarray*}
      K(s)
        \leq
          C|u_2(s)|^2_{\rH}\|u_2(s)\|^2_{\rV}\|U\|^2_{\Upsilon_s(\rH)}.
  \end{eqnarray*}

\item[(4)] {Assume that} $\|u_1\|_{\Upsilon_s(\rH)}\wedge \|u_2\|_{\Upsilon_s(\rH)}\geq m+1$. {In this case, we have}
  \begin{eqnarray*}
      K(s)=0.
  \end{eqnarray*}

\end{itemize}
Hence we infer that
\begin{eqnarray}\label{eq K}
  K(s)
&\leq&
  \frac{1}{4} \|U(s)\|^2_{\rV}    +
     C\|U\|^2_{\Upsilon_s(\rH)}\\
      && \Big[
       |u_1(s)|^2_{\rH}\|u_1(s)\|^2_{\rV}\cdot \mathds{1}_{[0,m+1]}(\|u_1\|_{\Upsilon_s(\rH)})
       +
       |u_2(s)|^2_{\rH}\|u_2(s)\|^2_{\rV}\cdot \mathds{1}_{[0,m+1]}(\|u_2\|_{\Upsilon_s(\rH)})\nonumber
       \Big].
\end{eqnarray}
Set
$$
\Theta(t):=\sup_{s\in[0,t]}|U(s)|^2_{\rH}+\int_0^t\|U(s)\|^2_{\rV}\,ds, \;\;\; {t \in [0,T]}.
$$
Substituting (\ref{eq K}) into (\ref{eq J1}), and then into (\ref{eq Z 0}), notice that
$$
\|U\|^2_{\Upsilon_s(\rH)}\leq 2 \Theta(s),
$$
we have
\begin{eqnarray}\label{eq local 01}
  \Theta(T)
&\leq&
  C\int_0^T\Theta(s)\Big[
       |u_1(s)|^2_{\rH}\|u_1(s)\|^2_{\rV}\cdot \mathds{1}_{[0,m+1]}(\|u_1\|_{\Upsilon_s(\rH)})\\
       &&\;\;\;\;\;\;\;\;\;\;\;\;\ \ \ \ \ \ \ \ \ \ \ +
       |u_2(s)|^2_{\rH}\|u_2(s)\|^2_{\rV}\cdot \mathds{1}_{[0,m+1]}(\|u_2\|_{\Upsilon_s(\rH)})\nonumber
       \Big]
    ds\nonumber\\
    &&+
    \sup_{t\in[0,T]}|J_2(t)|
    +
    J_3(T).\nonumber
\end{eqnarray}
Gronwall's lemma implies that
\begin{eqnarray}\label{eq local 02}
  \Theta(T)
&\leq&
  \Big(\sup_{t\in[0,T]}|J_2(t)|
    +
    J_3(T)
  \Big)\nonumber
  \\
&&  \cdot
  e^{C\int_0^T\left[
       |u_1(s)|^2_{\rH}\|u_1(s)\|^2_{\rV}\cdot \mathds{1}_{[0,m+1]}(\|u_1\|_{\Upsilon_s(\rH)})
       +
       |u_2(s)|^2_{\rH}\|u_2(s)\|^2_{\rV}\cdot \mathds{1}_{[0,m+1]}(\|u_2\|_{\Upsilon_s(\rH)})\nonumber
       \right]
    ds}\nonumber\\
&\leq&
  C_m \Big(\sup_{t\in[0,T]}|J_2(t)|
    +
    J_3(T)
  \Big).
\end{eqnarray}
By the Burkholder-Davis-Gundy inequality (see Theorem 23.12 in \cite{Kallenberg}) and assumption {\bf (G-H1)} (see (\ref{eqn-Lipschitz-H})), we get in a standard way the following inequality:
\begin{eqnarray}\label{eq J2}
  C_m\mathbb{E}\Big(\sup_{t\in[0,T]}|J_2(t)|\Big)
\leq
  \frac{1}{2}\|U\|^2_{\Lambda_T(\rH)}
  +
  C_mT\|y_1-y_2\|^2_{\Lambda_T(\rH)}.
\end{eqnarray}
Moreover, applying {\bf (G-H1)} again, we have
\begin{eqnarray}\label{eq J3}
  \mathbb{E}\Big(J_3(T)\Big)
\leq
  CT\|y_1-y_2\|^2_{\Lambda_T(\rH)}.
\end{eqnarray}
Summing up the inequalities (\ref{eq local 02}), (\ref{eq J2}), and (\ref{eq J3}), we deduce that
\begin{eqnarray*}
\|U\|^2_{\Lambda_T(\rH)}\leq C_m T \|y_1-y_2\|^2_{\Lambda_T(\rH)}.
\end{eqnarray*}
This proves  inequality (\ref{eq BPFT 01}), and thus the proof of Lemma \ref{lem 2} is complete.
\end{proof}

\vskip 0.2cm

\begin{corollary}\label{cor 1}
Under the assumptions {of} Lemma \ref{lem 1}, the solution of Problem (\ref{eq lemma 1 01}) is unique.

\end{corollary}

\begin{proof}[Proof of Corollary \ref{cor 1}]
Suppose that $Y_1$ and $Y_2$ are two solutions of Problem (\ref{eq lemma 1 01}). By the Lions-Magenes lemma, we infer that for every $t\geq 0$,
\begin{eqnarray}\label{eq cor 01}
&&\hspace{-1truecm}|Y_1(t)-Y_2(t)|^2_{\rH}+2\int_0^t\|Y_1(s)-Y_2(s)\|^2_{\rV}\,ds\nonumber\\
&=&
-2\int_0^t {\fourIdx{}{\rV^\prime}{}{}{\langle\,}} \theta_m(\|Y_1+M\|_{\Upsilon_s(\rH)})\rB(Y_1(s)+M(s))\\
&&\;\;\;\ \ \ \ \ \ \ \ \ \ \ \ -\theta_m(\|Y_2+M\|_{\Upsilon_s(\rH)})\rB(Y_2(s)+M(s)),Y_1(s)-Y_2(s){\fourIdx{}{}{}{\rV}{\,\rangle}}
\,ds.\nonumber
\end{eqnarray}
Set $u_1=Y_1+M$ and $u_2=Y_2+M$. The above equality implies that
\begin{eqnarray}\label{eq cor 02}
&&\hspace{-1truecm}|u_1(t)-u_2(t)|^2_{\rH}+2\int_0^t\|u_1(s)-u_2(s)\|^2_{\rV}\,ds\\
&=&
-2\int_0^t \dual{\langle \theta_m(\|u_1\|_{\Upsilon_s(\rH)})\rB(u_1(s))-\theta_m(\|u_2\|_{\Upsilon_s(\rH)})\rB(u_2(s)),u_1(s)-u_2(s)\rangle}
\,ds.\nonumber
\end{eqnarray}
We observe that the above equality is a special case of equality (\ref{eq Z 0}) with $G\equiv0$.
Therefore, the proof of Lemma \ref{lem 2} implies that $u_1=u_2$. Hence {we infer that } $Y_1=Y_2$.

The proof of  Corollary \ref{cor 1} is complete. \end{proof}
\vskip 0.2cm

Finally, we are ready to finish the proof of the main result in this section. We use the Banach fixed point theorem to prove this result.

\begin{proof}[\textbf{Proof of Theorem \ref{thm-initial data in H}}.]
The proof is divided into three steps.

\begin{itemize}
\item[{\bf Step 1.}] \textbf{Uniqueness.} For the uniqueness part of Theorem \ref{thm-initial data in H}, we refer to \cite{BLZ} or \cite{BHZ}.

\item[{\bf Step 2.}] \textbf{Local existence.} Consider  the following auxiliary problem
\begin{eqnarray}\label{eq SNS jieduan 1}
&& du_n(t)+{\rA}u_n(t)\,dt+\theta_n(\|u_n\|_{\Upsilon_t(\rH)})\rB(u_n(t))\,dt=f(t)\,dt+ \int_{{\rZ}}G(u_n(t-),z)\widetilde{\eta}(dz,dt),\nonumber\\
&& u_n(0)=u_0.
\end{eqnarray}
\todozbdel{As in my comment to Lemma \ref{lem 2} I think that we could need modify the ``noisy" coefficient by multiplying it by a something like we used to multiple the nonlinearity $\rB(u)$, i.e., by the $\theta_n(\|u_n\|_{\Upsilon_t(\rH)})$, which depends on $u_n$.
 }

We choose and fix $T>0$. For any $y\in \Lambda_T(\rH)$, Lemma \ref{lem 1} and Corollary \ref{cor 1} imply that there exists a unique element $u_n=\Phi^y\in\Lambda_T(\rH)$ satisfying
\begin{eqnarray}\label{eq thm 1 SNS jieduan 01}
&& du_n(t)+{\rA}u_n(t)\,dt+\theta_n(\|u_n\|_{\Upsilon_t(\rH)})\rB(u_n(t))\,dt=f(t)\,dt+ \int_{{\rZ}}G(y(t-),z)\widetilde{\eta}(dz,dt),\nonumber\\
&& u_n(0)=u_0.
\end{eqnarray}
\todozbdel{Following my last remark, I think that we could need modify the ``noisy" coefficient in equation \eqref{eq thm 1 SNS jieduan 01} by multiplying it by  $\theta_n(\|y\|_{\Upsilon_t(\rH)})$, which depends on $y$.
 }

Indeed, it is known  that there exists a unique $M\in \Lambda_T(\rH)$ satisfying the following equation
\begin{eqnarray*}
&&dM(t)+{\rA}M(t)\,dt=\int_{{\rZ}}G(y(t-),z)\widetilde{\eta}(dz,dt), \;\; t\geq 0, \;\;\; M(0)=0
\end{eqnarray*}
 and satisfying the following inequality,
$$
\mathbb{E}\big(\sup_{t\in[0,T]}|M(t)|^2_{\rH}\big)+\mathbb{E}\big(\int_0^T\|M(t)\|^2_{\rV}\,dt\big)\leq C_T\Big[\mathbb{E}\big(\sup_{t\in[0,T]}|y(t)|^2_{\rH}\big)+1\Big].
$$
Hence, Lemma \ref{lem 1} and Corollary \ref{cor 1} imply that for any $\omega\in \Omega$, there exists a unique element $X(\omega)\in C([0,T],\rH)\cap L^2([0,T],\rV)$ solving
\begin{eqnarray*}
&& dX(t)+{\rA}X(t)\,dt+\theta_n(\|X+M\|_{\Upsilon_t(\rH)})\rB(X(t)+M(t))\,dt=f(t)\,dt,\nonumber\\
&& X(0)=u_0.
\end{eqnarray*}

One can show that $u$ is a solution to (\ref{eq thm 1 SNS jieduan 01}) iff  $u=X+M$. For uniqueness, we refer to Lemma \ref{lem 2}. Moreover, Lemma \ref{lem 2} implies that there exists a constant $C_n>0$ such that
\begin{eqnarray}\label{eq BPFT 0100}
\|\Phi^{y_1}-\Phi^{y_2}\|_{\Lambda_T(\rH)}^2\leq C_nT\|y_1-y_2\|^2_{\Lambda_T(\rH)},\ \ \ \forall y_1, y_2\in \Lambda_T(\rH).
\end{eqnarray}

Let $T_n=\frac{1}{2C_n}$. In view of inequality (\ref{eq BPFT 0100}) and by using the Banach fixed point theorem, we infer that there exists a unique element $u_n^1\in \Lambda_{T_n}(\rH)$ such that $u_n^1$ is a solution of (\ref{eq SNS jieduan 1}) for $t\in[0,T_n]$.
Repeating the above proof, and observing that the constant $T_n$ does not depend on the initial datum,
 we can find a unique element $u_n^2:=\{u_n^2(t),\ t\in[0,2T_n]\} \in \Lambda_{2T_n}(\rH)$ solving the  following problem
 \begin{eqnarray*}
&& \hspace{-0.8truecm}du_n(t)+{\rA}u_n(t)\,dt+\theta_n(\|u_n\|_{\Upsilon_t(\rH)})\rB(u_n(t))\,dt
\\
&&\ \ =f(t)\,dt+ \int_{{\rZ}}G(u_n(t-),z)\widetilde{\eta}(dz,dt),\ t\in [T_n,2T_n],\nonumber\\
&& \hspace{-0.8truecm}u_n(t)=u_n^1(t),\ t\in[0,T_n].
\end{eqnarray*}
\todozbdel{Following my last remark, I think that we could need modify the ``noisy" coefficient in equation \eqref{eq thm 1 SNS jieduan 01} by multiplying it by  $\theta_n(\|u_n\|_{\Upsilon_t(\rH)})$, which depends on $u_n$.
 }

 It is not difficult to see that $u_n^2$ is a solution of Problem (\ref{eq SNS jieduan 1}) on the time interval $[0,2T_n]$.
 By induction, we can construct a unique element $u_n\in \Lambda_T(\rH)$ which
is a solution of Problem (\ref{eq SNS jieduan 1}) for $t\in[0,T]$, where $T>0$ is arbitrary.
\vskip 0.2cm

{Define} a stopping time
\begin{equation}\label{eqn-tau_n}
\tau_n=\inf\{t\geq0:\,\|u_n\|_{\Upsilon_t(\rH)}{>} n\}.
\end{equation}
By definition, $\theta_n(\|u_n\|_{\Upsilon_t(\rH)})=1$
for any $t\in[0,\tau_n)$, hence $\{u_n(t),\, t\in[0,\tau_n)\}$ is a local solution of Problem (\ref{eq SNS 01}).
Thus, by the uniqueness of solutions to Problem (\ref{eq SNS 01}), we infer that
$$
u_{n+1}(t)=u_{n}(t),\ \ \ t\in[0,\tau_{n}\wedge\tau_{n+1})\ \mathbb{P}\text{-a.s.}.
$$
Hence, the sequence $(\tau_n)_{n=1}^\infty $ is nondecreasing. We set
$\tau_{max}:=\lim_{n\to\infty}\tau_n$, and we observe that $\tau_{max}$ is also a stopping time.

\vskip 0.2cm

Now we can construct a local solution of Problem (\ref{eq SNS 01}) as $\{u(t),\,t\in[0,\tau_{max})\}$ defined by
$$
u(t)=u_n(t),\ \ \ t\in[0,\tau_n).
$$
 Using an argument similar to the proof of Theorem 3.5 in \cite{BHR15}, we can
prove that
\begin{equation}\label{eqn-tau_max}
\lim_{t\toup \tau_{max}}\|u\|_{\Upsilon_t(\rH)}=\infty\text{ on }\{\omega \in \Omega:\,\tau_{max}<\infty\}\ \mathbb{P}\text{-a.s..}
\end{equation}

\item[{\bf Step 3.}] \textbf{Global existence.} We will prove that
\begin{equation}\label{eqn-tau_max=infty}
\mathbb{P}(\tau_{max}=\infty)=1.
\end{equation}
{It is sufficient to prove that for very $T>0$, $\mathbb{P}(\tau_{max} \geq T)=1$. For the rest of this proof we choose and fix $T>0$.} \\
We observe that in this step, we do not use the Lipschitz assumption \eqref{eqn-Lipschitz-H} but only the linear growth assumption (\ref{eqn-linear growth-H}).\\
By {the It\^o} formula, we have
\begin{eqnarray*}
&&\hspace{-1truecm}|u(t\wedge\tau_n)|^2_{\rH}+2\int_0^{t\wedge\tau_n}\|u(s)\|^2_{\rV}\,ds\\
&=&
  |u_0|^2_{\rH}
  +
  2\int_0^{t\wedge\tau_n}\dual{\langle f(s),u(s)\rangle}\,ds
  +
  2\int_0^{t\wedge\tau_n}\int_{{\rZ}}\langle G(u(s-),z),u(s-)\rangle_{\rH}\widetilde{\eta}(dz,ds)\\
  &&+
  \int_0^{t\wedge\tau_n}\int_{{\rZ}}|G(u(s-),z)|^2_{\rH} \eta(dz,ds),\ \forall t\in[0,T].\nonumber
\end{eqnarray*}
Applying the Burkholder-Davis-Gundy inequality,
\begin{eqnarray}\label{eq global H 01}
&& \hspace{-1truecm}\mathbb{E}\Big(\sup_{t\in[0,T]}|u({t\wedge\tau_n})|^2_{\rH}
    +
   \int_0^{T\wedge\tau_n}\|u(s)\|^2_{\rV}\,ds\Big)\nonumber\\
&\leq&
   |u_0|^2_{\rH}
   +
   \int_0^T\|f(s)\|^2_{\rV^\prime}\,ds
   +
    2\mathbb{E}\Big(\sup_{t\in[0,T]}
       \Big|\int_0^{t\wedge\tau_n}\int_{{\rZ}}\langle G(u(s-),z),u(s-)\rangle_{\rH}\widetilde{\eta}(dz,ds)\Big|
         \Big)\nonumber\\
    &&+
    \mathbb{E}\Big(\int_0^{T\wedge\tau_n}\int_{{\rZ}}|G(u(s-),z)|^2_{\rH} \eta(dz,ds)\Big)\\
&\leq&
  |u_0|^2_{\rH}
  +
   \int_0^T\|f(s)\|^2_{\rV^\prime}\,ds
   +
   \frac{1}{2}\mathbb{E}\Big(\sup_{t\in[0,T]}|u({t\wedge\tau_n})|^2_{\rH}\Big)
   \nonumber\\
    && +
  C\int_0^T\mathbb{E}\Big(\sup_{l\in[0,t]}|u(l\wedge\tau_n)|^2_{\rH}\Big)\,dt
  +
  CT.\nonumber
\end{eqnarray}
Applying Gronwall's lemma, we infer that
\begin{eqnarray*}
\mathbb{E}\Big(\sup_{t\in[0,T]}|u(t\wedge\tau_n)|^2_{\rH}
    +
   \int_0^{T\wedge\tau_n}\|u(s)\|^2_{\rV}\,ds\Big)
   \leq
   C_T(1+|u_0|^2_{\rH}+ \int_0^T\|f(s)\|^2_{\rV^\prime}\,ds).
\end{eqnarray*}
Taking the limit $n\to \infty$, so that  $\tau_n\toup \tau_{max}$, we deduce that
\begin{eqnarray}\label{eq sup es u}
\hspace{-0.2truecm}\mathbb{E}\Big(\sup_{t\in[0,T\wedge\tau_{max})}\!|u(t)|^2_{\rH}
    +
   \int_0^{T\wedge\tau_{max}}\!\|u(s)\|^2_{\rV}\,ds\Big)
   \leq
   C_T(1+|u_0|^2_{\rH}+ \int_0^T\!\|f(s)\|^2_{\rV^\prime}\,ds).
\end{eqnarray}
This lead to $\mathbb{P}$-a.s.,
$$
\sup_{t\in[0,T\wedge\tau_{max})}|u(t)|^2_{\rH}+\int_0^{T\wedge\tau_{max}}\|u(s)\|^2_{\rV}\,ds<\infty.
$$
{The above implies that    \[
\mbox{ the function } [0,T\wedge\tau_{max}) \ni t\ni \|u\|^2_{\Upsilon_t(\rH)}  \mbox{ is bounded } \mathbb{P}\mbox{-almost surely}, \]
what in turn in view of \eqref{eqn-tau_max} implies that $\tau_{max} \geq T$, $\mathbb{P}$-almost surely, as required.}
\end{itemize}

The proof of Theorem \ref{thm-initial data in H} is thus complete.
\end{proof}

\section{
Solutions to SNSEs  with initial data in the space $\rV$}
\label{subsection 2.2}
\vskip 0.3cm

Now we consider SNSEs with more regular data. For this purpose, we formulate the following assumptions.

\begin{con}\label{con1 G} The function
$G:\rV\times {\rZ}\to \rV$ is a measurable map \dela{such that
\[G: (u,z) \in \rV \mbox{ for all }(u,v) \in \rV \times {\rZ}.\]
Assume the corresponding restriction function $\rV\times {\rZ}\to \rV$, denoted by the same notation, is also measurable.
Assume also }that constants $C_1>0$ and $C_2>0$ exist such that
\begin{itemize}
\item[(G-V1)](Lipschitz {in $\rV$})
\begin{equation}\label{eqn-Lipschitz-V}
       \int_{\rZ}\|G(v_1,z)-G(v_2,z)\|^2_{\rV}\nu(dz)\leq C_1\|v_1-v_2\|^2_{\rV},\ \ \ v_1,v_2\in \rV,
      \end{equation}

\item[(G-V2)](Linear growth {in $\rV$})
\begin{equation}\label{eqn-linear growth-V}
       \int_{{\rZ}}\|G(v,z)\|^2_{\rV}\nu(dz)\leq C_1(1+\|v\|^2_{\rV}),\ \ \ \ v\in \rV.
      \end{equation}
\item[(G-VH2)](Linear growth in $\rH$)
\begin{equation}\label{eqn-linear growth-H-2}
       \int_{{\rZ}}|G(v,z)|^2_{\rH}\nu(dz)\leq C_2(1+|v|^2_{\rH}),\;\;\; v\in \rV.
      \end{equation}
\end{itemize}
\end{con}

In this section, we will  prove the following result.
\begin{thm}\label{thm-initial data in V}
Assume that a function $G$ satisfies  \textbf{Assumption \ref{con1 G}}.
Then for all $u_0\in \rV$ and $f\in L^2_{loc}([0,\infty),\rH)$, there exists a unique $\mathbb{F}$-{progressively measurable} process $u$ such that
\begin{itemize}
\item[(1)] $u\in D([0,\infty),\rV)\cap L^2_{loc}([0,\infty),\mathcal{D}({\rA}))$, $\mathbb{P}$-a.s.,

\item[(2)] the following equality holds, for all $t\in[0,\infty)$, $\mathbb{P}$-a.s., in \addaok{$\rH$}:
      \begin{eqnarray}\label{eq V 00}
      \hspace{-1truecm}u(t)=u_0-\int_0^t{\rA}u(s)\,ds - \int_0^t\rB(u(s))\,ds+\int_0^t f(s)\,ds+\int_0^t\int_{{\rZ}}G(u({s-}),z)\widetilde{\eta}(dz,ds).
      \end{eqnarray}
\end{itemize}
\end{thm}

\begin{remark}\label{rem-uniqueness in V} Let us observe that {Assumption \ref{con1 G}}, i.e., the Lipschitz property of $G$ with respect to the $\rV$-norm, does not imply Assumption \ref{con G},
i.e., the Lipschitz property of $G$ with respect to the $\rH$-norm. Hence, the uniqueness part of Theorem \ref{thm-initial data in V} is not a consequence of Theorem \ref{thm-initial data in H},
and we need an independent proof of the uniqueness.
\end{remark}

\begin{remark}\label{Rem2}
In \cite{BHR15}, the authors considered the existence and uniqueness of solutions defined as in Theorem \ref{thm-initial data in V} for stochastic
hydrodynamical systems with L\'evy noise, including 2D Navier-Stokes euations. They assumed {that the function $G$ is globally Lipschitz in the sense that there exists $K>0$ such that} for $p=1,2$,
$$
\int_{\rZ}\|G(v_1,z)-G(v_2,z)\|^{2p}_{\rV}\nu(dz)\leq K\|v_1-v_2\|^{2p}_{\rV}, \;\;\; {v_1,v_2 \in \rV},
$$
and
$$
\int_{\rZ}|G(v_1,z)-G(v_2,z)|^{2p}_{\rH}\nu(dz)\leq K|v_1-v_2|^{2p}_{\rH},\;\;\;  {v_1,v_2 \in \rH}.
$$
It is easy to see that our assumptions are weaker than the above.

\dela{\textcolor[rgb]{1.00,0.00,0.00}{In a forthcoming paper \cite{BPZ_2021-local Lipschitz}, the authors will generalize the results from the present paper to the case when the coefficient $G$ is Lipschitz on balls, i.e., it satisfies the following condition. For every $\hbar>0$ there exists $C_\hbar>0$ such that
\begin{equation}\label{eqn-Lipschitz-V-local}
       \int_{\rZ}\Vert G(v_1,z)-G(v_2,z)\Vert_{\rV}^2\nu(dz)\leq C_\hbar\Vert v_1-v_2\Vert_{\rV}^2,\;\;\; v_1,v_2\in B_\hbar(\rV),
\end{equation}
where $B_\hbar(\rV)$ denotes the ball of radius $\hbar$ in $\rV$.
As in Remark \ref{Rem1}, the case when $G$ is only locally Lipschitz remains open.} \textcolor[rgb]{0.44,0.00,0.94}{From Zhai: It is Theorem \ref{thm-initial data in V local}. }}

Let us also mention that in the Gaussian case, stochastic Navier-Stokes equations, respectively Euler equations, for initial data in the space $\rV$ have been studied in \cite{Mikulevicius_2009} and \cite{GHZ2009}, and respectively in \cite{Brz+Peszat_2001_eu}.
\end{remark}


\begin{remark}
In Section \ref{subsection 2.1}, we proved  two existence results.
The first one, i.e., Theorem \ref{thm-initial data in H}, holds under the global Lipschitz assumptions on the coefficient $G$.
The second
one, i.e., Theorem \ref{thm-initial data in H local}, holds under the assumption that $G$ is Lipschitz on balls on the space $\rH$ and of linear growth.
The bulk of the proof was devoted to the proof of the former result, as the latter one follows from the former by standard procedure. \\
In the same vein, in the present section, we formulate first Theorem \ref{thm-initial data in V} which holds under the assumption that the coefficient $G$ is globally Lipschitz with respect to the space $\rV$.
This result is supplemented by Theorem \ref{thm-initial data in V local} below, in which we assume that the coefficient $G$ is  Lipschitz on balls on the space $\rV$. The latter result can be deduced from the former one
by standard truncation procedure.
\end{remark}


\begin{con}\label{con G Local in V}
 A map $G:\rV\times {\rZ}\rightarrow \rV$ is a measurable map such that
\begin{itemize}
\item[(G-V1-local)](Lipschitz on balls) for every $\hbar>0$, there exists a constant $C_\hbar>0$ such that, for all $v_1,v_2\in\rV$ with $\|v_1\|_\rV\vee\|v_2\|_\rV\leq \hbar$,
\begin{equation}\label{eqn-Lipschitz-V-local}
       \int_{\rZ}\|G(v_1,z)-G(v_2,z)\|^2_{\rV}\nu(dz)\leq C_\hbar\|v_1-v_2\|^2_{\rV},
\end{equation}
\end{itemize}
and the assumptions \textbf{(G-V2)}(Linear growth {in $\rV$}) {and \textbf{(G-VH2)}(Linear growth in $\rH$)} hold.
\end{con}

\vskip 0.2cm

\begin{thm}\label{thm-initial data in V local}
Assume that \textbf{Assumption \ref{con G Local in V}} holds\dela{ and \textbf{(G-H2)} in \textbf{Assumption \ref{con G}} hold}. Then for all $u_0\in \rV$ and $f\in L^2_{loc}([0,\infty),\rH)$, there exists a unique $\mathbb{F}$-{progressively measurable} process $u$ such that
\begin{itemize}
\item[(1)] $u\in D([0,\infty),\rV)\cap L^2_{loc}([0,\infty),\mathcal{D}({\rA}))$, $\mathbb{P}$-a.s.,

\item[(2)] the following equality holds, for all $t\in[0,\infty)$, $\mathbb{P}$-a.s., in $\rH$,
      \begin{eqnarray}\label{eq Zhai local V 01}
      \hspace{-1truecm}u(t)=u_0-\int_0^t{\rA}u(s)\,ds - \int_0^t\rB(u(s))\,ds+\int_0^t f(s)\,ds
      +
      \int_0^t\int_{{\rZ}}G(u({s-}),z)\widetilde{\eta}(dz,ds).
      \end{eqnarray}
\end{itemize}
\end{thm}

\dela{\textcolor[rgb]{0.44,0.00,0.94}{From Zhai: the details:}}
\begin{proof}[Proof of Theorem \ref{thm-initial data in V local}]
The proof is similar to that for Theorem \ref{thm-initial data in H local}.

For any natural number $k \in \mathbb{N}$, we define an auxiliary function $G_k$ by
$$
G_k :\rV \times \rZ \ni (y,z)\mapsto G\Big(\frac{\|y\|_\rV\wedge k}{\|y\|_\rV}y,z\Big) \in \rV,
$$
where we set $\frac{\|y\|_\rV\wedge k}{\|y\|_\rV}=1$ when $y=0$. {Since, by our assumptions, $G$ satisfies \textbf{Assumption \ref{con G Local in V}}, we can easily show that
$G_k$ satisfies \textbf{Assumption \ref{con1 G}}.}
From now we choose and fix $k>\|u_0\|_\rV$.
By Theorem \ref{thm-initial data in V}, we infer that there exists
a unique $\mathbb{F}$-{progressively measurable} process $X^k$ such that
\begin{itemize}
 \item $X^k\in D([0,\infty),\rV)\cap L^2_{loc}([0,\infty),\mathcal{D}({\rA}))$, $\mathbb{P}$-a.s.,
 \item the following equality holds, for all $t\in[0,\infty)$, $\mathbb{P}$-a.s., in $\rH$,
      \begin{eqnarray*}\label{Eq th local H 01}
      \hspace{-1truecm}X^k(t)=u_0-\int_0^t{\rA}X^k(s)\,ds - \int_0^t\rB(X^k(s))\,ds+\int_0^t f(s)\,ds+\int_0^t\int_{{\rZ}}G_k(X^k({s-}),z)\widetilde{\eta}(dz,ds).
      \end{eqnarray*}
\end{itemize}
{Similarly to \eqref{eqn-sigma_k} we} define a stopping time
$$
\sigma_k:=\inf\{t\geq 0:\sup_{s\in[0,t]}\|X^k(s)\|_\rV>k\}.
$$
\deln{where we set $\inf \emptyset= \infty$.} It is not difficult to see that $\sigma_k$ is increasing in $k$, and $X^{k+1}(t)=X^k(t),\, t\in[0,\sigma_k)$.
{We also define a stopping time} $\sigma:=\lim_{k\rightarrow\infty}\sigma_k$.
The property above  enables us to define $u(t)$ for $t\in[0,\sigma)$ as follows:
$$
u(t):=X^k(t),\, t\in[0,\sigma_k).
$$
It is easy to see that $u(t),\, t\in[0,\sigma)$ is a local solution of Problem (\ref{eq Zhai local V 01}). To complete the proof, we need only show that
$\mathbb{P}(\sigma=\infty)=1$. For this purpose, we use Condition \textbf{{(G-VH2)}} from \textbf{Assumption \ref{con G Local in V}}.

\vskip 0.1cm
Following the argument we used in the proof of inequality (\ref{eq sup es u}), we can find $C_T>0$ such that
\begin{eqnarray}\label{eq1 esta u pro Zhai}
\mathbb{E}\Big(\sup_{t\in[0,T\wedge\sigma)}|u(t)|^2_{\rH}\Big)
+
\mathbb{E}\Big(\int_0^{T\wedge\sigma}\|u(t)\|^2_{\rV}\,dt\Big)
\leq
C_T.
\end{eqnarray}
\vskip 0.2cm

Define an additional stopping time $\widetilde{\tau}_N$ by
$$
\widetilde{\tau}_N:=\inf\{t\geq 0: \sup_{s\in[0,t]}|u(s)|^2_{\rH}+\int_0^t\|u(s)\|^2_{\rV}\,ds\geq N\}\wedge T\wedge\sigma,
$$
and set $\tau_{N,k}:=\widetilde{\tau}_N\wedge\sigma_k$.
By {the It\^o} formula and Lemma \ref{lem B baisc prop}, we have, for $t \geq 0$,
\begin{eqnarray*}
&&\hspace{-2truecm}\|u(t)\|^2_{\rV} + 2\int_0^t\|u(s)\|^2_{\mathcal{D}({\rA})}\,ds\nonumber\\
&&\hspace{-2truecm}=
 \|u_0\|^2_{\rV} - 2\int_0^t\langle \rB(u(s)),{\rA}u(s)\rangle_{\rH}\,ds
 +
 2\int_0^t\langle f(s),{\rA}u(s)\rangle_{\rH}\,ds\nonumber\\
 &&\hspace{-2truecm}\ \ \ +
 2\int_0^t\int_{{\rZ}}\langle G(u(s-),z),u(s-)\rangle_{\rV}\widetilde{\eta}(dz,ds)
 +
 \int_0^t\int_{{\rZ}}\|G(u(s-),z)\|^2_{\rV}\eta(dz,ds)\nonumber\\
&&\hspace{-2truecm}\leq
  \|u_0\|^2_{\rV} + \int_0^t\|u(s)\|^2_{\mathcal{D}({\rA})}\,ds +C\int_0^t\|u(s)\|^4_{\rV}|u(s)|^2_{\rH}\,ds+
  2\int_0^t|f(s)|^2_{\rH}\,ds\nonumber\\
  &&\hspace{-2truecm}\ \ \ + 2\int_0^t\int_{{\rZ}}\langle G(u(s-),z),u(s-)\rangle_{\rV}\widetilde{\eta}(dz,ds)
  +
  \int_0^t\int_{{\rZ}}\|G(u(s-),z)\|^2_{\rV}\eta(dz,ds).
\end{eqnarray*}
Applying Gronwall's lemma, we infer that
\begin{eqnarray}\label{eq1 page 13 01 Zhai}
 &&\hspace{-1truecm}\|u(t\wedge \tau_{N,k})\|^2_{\rV} + \int_0^{t\wedge \tau_{N,k}}\|u(s)\|^2_{\mathcal{D}({\rA})}\,ds\nonumber\\
&\leq&
  e^{C\int_0^{t\wedge \tau_{N,k}}|u(s)|^2_{\rH}\|u(s)\|^2_{\rV}\,ds}\nonumber\\
  &&\times \Big(
     \|u_0\|^2_{\rV} + 2\int_0^T|f(s)|^2_{\rH}\,ds
     +
     \sup_{s\in[0, T]}\Big|\int_0^{s\wedge\tau_{N,k}}\int_{{\rZ}}\langle G(u(l-),z),u(l-)\rangle_{\rV}\widetilde{\eta}(dz,dl)\Big|\nonumber\\
     &&\ \ \ \ \ +
     \int_0^{T\wedge\tau_{N,k}}\int_{{\rZ}}\|G(u(s-),z)\|^2_{\rV}\eta(dz,ds)
  \Big)\nonumber\\
&\leq&
  e^{CN^2}\Big(
     \|u_0\|^2_{\rV} + 2\int_0^T|f(s)|^2_{\rH}\,ds
     +
     \sup_{s\in[0, T]}\Big|\int_0^{s\wedge\tau_{N,k}}\int_{{\rZ}}\langle G(u(l-),z),u(l-)\rangle_{\rV}\widetilde{\eta}(dz,dl)\Big|\nonumber\\
     &&\ \ \ \ \ \ \ \ \ \ \ \ \ \ +
     \int_0^{T\wedge\tau_{N,k}}\int_{{\rZ}}\|G(u(s-),z)\|^2_{\rV}\eta(dz,ds)
  \Big),\;\; t\in[0,T].
\end{eqnarray}

By the Burkholder-Davis-Gundy inequality and the assumption \textbf{(G-V2)}, i.e., \eqref{eqn-linear growth-V}, we get
\begin{align}\label{eq1 page 13 02 Zhai}
&\hspace{-1truecm}e^{CN^2}\mathbb{E}\Big(
    \sup_{s\in[0, T]}\Big|\int_0^{s\wedge\tau_{N,k}}\int_{{\rZ}}\langle G(u(l-),z),u(l-)\rangle_{\rV}\widetilde{\eta}(dz,dl)\Big|
     \Big)\nonumber\\
&\leq
Ce^{CN^2}\mathbb{E}\Big(
     \Big|\int_0^{T\wedge\tau_{N,k}}\int_{{\rZ}}\|G(u(s-),z)\|^2_{\rV}\|u(s-)\|^2_{\rV}\eta(dz,ds)\Big|^{\frac{1}{2}}
         \Big)\nonumber\\
&\leq
 \frac{1}{2} \mathbb{E}\Big(\sup_{s\in[0,T]}\|u(s\wedge\tau_{N,k})\|^2_{\rV}\Big)
+
 Ce^{CN^2}\mathbb{E}\Big(
             \int_0^{T\wedge\tau_{N,k}}\int_{{\rZ}}\|G(u(s),z)\|^2_{\rV}\nu(dz)\,ds
           \Big)\nonumber\\
&\leq
 \frac{1}{2} \mathbb{E}\Big(\sup_{s\in[0,T]}\|u(s\wedge\tau_{N,k})\|^2_{\rV}\Big)
+
 Ce^{CN^2}\int_0^{T}\mathbb{E}(1+\|u(s\wedge\tau_{N,k})\|^2_{\rV})\,ds.
\end{align}
Applying the assumption \textbf{(G-V2)} again, we infer that
\begin{eqnarray}\label{eq1 page 13 03 Zhai}
 \mathbb{E}\Big(\int_0^{T\wedge\tau_{N,k}}\int_{{\rZ}}\|G(u(s-),z)\|^2_{\rV}\eta(dz,ds)\Big)
\leq
 C\int_0^{T}\mathbb{E}(1+\|u(s\wedge\tau_{N,k})\|^2_{\rV})\,ds.
\end{eqnarray}
Inserting inequalities (\ref{eq1 page 13 02 Zhai}) and (\ref{eq1 page 13 03 Zhai}) into inequality (\ref{eq1 page 13 01 Zhai}), and then using Gronwall's lemma, we infer that
\begin{eqnarray*}
  \mathbb{E}\Big(\sup_{t\in[0,T]}\|u(t\wedge\tau_{N,k})\|^2_{\rV}\Big) + \mathbb{E}\Big(\int_0^{T\wedge\tau_{N,k}}\|u(s)\|^2_{\mathcal{D}({\rA})}\,ds\Big)
  \leq
  C_{N,T}\Big(
      \|u_0\|^2_{\rV}+\int_0^T|f(s)|^2_{\rH}\,ds + 1
      \Big).
\end{eqnarray*}
Taking the limit $k\to \infty$, we get
\begin{eqnarray*}
 \mathbb{E}\Big(\sup_{t\in[0,T\wedge\widetilde{\tau}_N\wedge\sigma)}\|u(t)\|^2_{\rV}\Big)
 +
 \mathbb{E}\Big(\int_0^{T\wedge\widetilde{\tau}_N\wedge\sigma}\|u(t)\|^2_{\mathcal{D}({\rA})}\,dt\Big)
 \leq
 C_{N,T}\Big(
      \|u_0\|^2_{\rV}+\int_0^T|f(s)|^2_{\rH}\,ds + 1
     \Big).
\end{eqnarray*}
This implies that
\begin{eqnarray}\label{eq2 V global Zhai}
  \sup_{t\in[0,T\wedge\widetilde{\tau}_N\wedge\sigma)}\|u(t)\|^2_{\rV}+\int_0^{T\wedge\widetilde{\tau}_N\wedge\sigma}\|u(t)\|^2_{\mathcal{D}({\rA})}\,dt<\infty,\ \ \ \ \mathbb{P}\text{-a.s..}
\end{eqnarray}

For a fixed $T>0$, we set
$$
\Omega_N:=\{\omega\in\Omega,\ \widetilde{\tau}_N= T\wedge\sigma\}.
$$
Then, $\Omega_N\subset\Omega_{N+1}$. By (\ref{eq1 esta u pro Zhai}) and (\ref{eq2 V global Zhai}), we deduce  that
$$
\lim_{N\to\infty}\mathbb{P}(\Omega_N)=1,
$$
and
\begin{eqnarray*}
  \sup_{t\in[0,T\wedge\sigma)}\|u(t)\|^2_{\rV}+\int_0^{T\wedge\sigma}\|u(t)\|^2_{\mathcal{D}({\rA})}\,dt<\infty,\ \ \ \ \text{on }\Omega_N\ \ \mathbb{P}\text{-a.s..}
\end{eqnarray*}

Hence
\begin{eqnarray*}
  \sup_{t\in[0,T\wedge\sigma)}\|u(t)\|^2_{\rV}+\int_0^{T\wedge\sigma}\|u(t)\|^2_{\mathcal{D}({\rA})}\,dt<\infty,\ \ \ \ \mathbb{P}\text{-a.s.,}
\end{eqnarray*}
which yields that
$$
\mathbb{P}(\sigma\geq T)=1,\ \ \forall T>0.
$$

The proof of Theorem \ref{thm-initial data in V local} is complete.
\end{proof}

Similar to Section \ref{subsection 2.1}, we first introduce  symbols which will be used later. After that, we state three auxiliary results:
Lemmata \ref{lem PDE 1}, \ref{lem PDE 2}, and Corollary \ref{cor PDE 2}.

In this section, we set, for $T\geq 0$,
\begin{equation}
\label{eqn-Upsilon_T-V}
\Upsilon_T(\rV)= D([0,T],\rV)\cap L^2([0,T],\mathcal{D}({\rA})).
\end{equation}
 Note that the definition of the space above differs from \eqref{eqn-Upsilon_T-H}.
It is easy to see that the space $\Upsilon_T(\rV)$ endowed with the norm
\begin{equation}
\label{eqn-Upsilon_T-V-norm}
\|y\|_{\Upsilon_T(\rV)}= \sup_{s\in[0,T]}\|y(s)\|_{\rV} + \Bigl(\int_0^T\|y(s)\|^2_{\mathcal{D}({\rA})}\,ds\Bigr)^{\frac{1}{2}}
\end{equation}
is a Banach space.

Let $\Lambda_T(\rV)$ be the space of all {$\rV$}-valued c\`adl\`ag $\mathbb{F}$-{progressively measurable} processes $y$
 {whose a.a. trajectories belong to the space $\Upsilon_T(\rV)$ and }
such that
\begin{equation}
\label{eqn-Lambda_T-V}
\|y\|_{\Lambda_T(\rV)}^2:=\mathbb{E}\Big(\sup_{s\in[0,T]}\|y(s)\|^2_{\rV}+\int_0^T\|y(s)\|^2_{\mathcal{D}({\rA})}\,ds\Big)<\infty.
\end{equation}
 {We point out that the space $\Lambda_T(\rH)$ introduced earlier around \eqref{eqn-Lambda_T-H} differs from the current space $\Lambda_T(\rV)$}.

Recall that the auxiliary function $\theta_m(\cdot)$ has been introduced in \eqref{eqn-theta_m} (and used, for instance, in Lemma \ref{lem 1}).

We are now ready to state the first of the three auxiliary results we need to prove Theorem \ref{thm-initial data in V}.

\begin{lemma}\label{lem PDE 1}
Assume that $T>0$ and $m\in \mathbb{N}$. Then for all {$u_0 \in \rV$, $f\in L^2([0,T];\rH)$}, and $\zsmall \in\Upsilon_T(\rV)$, there exists a function $\ysmall \in C([0,T],\rV)\cap L^2([0,T],\mathcal{D}({\rA}))$ satisfying
\begin{eqnarray}\label{eq lemma PDE 1 01}
&& \ysmall ^\prime(t)+{\rA}\ysmall (t)+\theta_m(\|\ysmall +\zsmall \|_{\Upsilon_t(\rV)})\rB(\ysmall (t)+\zsmall (t))=f(t),\;\; t \in (0,T),\nonumber\\
&& \ysmall (0)=u_0.
\end{eqnarray}

\end{lemma}

\begin{proof}[Proof of Lemma \ref{lem PDE 1}]

Fix a $T>0$ and $m\in \mathbb{N}$. Also fix {$u_0 \in \rV$, $f\in L^2([0,T];\rH)$} and $\zsmall \in\Upsilon_T(\rV)$.
We use the Picard iterative method again to prove this result.

\vskip 0.2cm
Choose and fix $y_0\in C([0,T],\rV)\cap L^2([0,T],\mathcal{D}({\rA}))$ such that $y_0(0)=u_0$.
For instance, we can take $y_0(t)=e^{-t\rA}u_0$, $t\in [0,T]$.

It is not difficult to prove that, given
$y_{n}\in C([0,T],\rV)\cap L^2([0,T],\mathcal{D}({\rA}))$, $n\in\mathbb{N}$, there exists a unique $y_{n+1}\in C([0,T],\rV)\cap L^2([0,T],\mathcal{D}({\rA}))$, satisfying the following deterministic initial value problem
\begin{eqnarray}\label{App auxi spde PDE 02}
&& y_{n+1}^\prime(t)+{\rA}y_{n+1}(t)+\theta_m(\|y_n+\zsmall \|_{\Upsilon_t(\rV)})\rB(y_n(t)+\zsmall (t),y_{n+1}(t)+\zsmall (t))=f(t),\nonumber\\
&& y_{n+1}(0)=u_0.
\end{eqnarray}

 {We will show that $\{y_n,n\in\mathbb{N}\}$ is a Cauchy sequence in $C([0,T],\rV)\cap L^2([0,T],\mathcal{D}({\rA}))$.}
We now estimate the norm in $C([0,T],\rV)\cap L^2([0,T],\mathcal{D}({\rA}))$ of the difference $y_{n+1}-y_n$, {for $n\geq 1$.}
To do so, set, for $x_i \in C([0,T],\rV)\cap L^2([0,T],\mathcal{D}({\rA}))$, $i=1,2,3,$
$$
\Xi(x_1,x_2,x_3)(s)=\theta_m(\|x_1+\zsmall \|_{\Upsilon_s(\rV)})\rB(x_2(s)+\zsmall (s),x_3(s)+\zsmall (s)),\,\;\; s\in [0,T].
$$
By the Lions-Magenes lemma, we have
\begin{eqnarray}\label{App auxi esti PDE 01}
\|y_{n+1}(t)-y_n(t)\|^2_{\rV}+2\int_0^t\|y_{n+1}(s)-y_n(s)\|^2_{\mathcal{D}({\rA})}\,ds=-2\int_0^t K(s)\,ds,
\end{eqnarray}
where
{\small
\begin{eqnarray*}
K(s)
=\Big\langle \Xi(y_n,y_n,y_{n+1})(s)-\Xi(y_{n-1},y_{n-1},y_{n})(s),
   {\rA}(y_{n+1}(s)-y_n(s))
   \Big\rangle_{\rH}
\end{eqnarray*}
} with $\big\langle \cdot,\cdot\big\rangle_{\rH}$ denoting the scalar product in $\rH$.

\vskip 0.1cm
 {Now fix $s\in [0,T]$.} To estimate $K(s)$, we {consider} {three} cases.
Each case contains a calculation of a certain ``partial" integral $\int_0^t |K(s)|\,ds$.

\vskip 0.1cm
\noindent{\bf Case 1.} {Assume that } $\|y_n+\zsmall \|_{\Upsilon_s(\rV)}\vee \|y_{n-1}+\zsmall \|_{\Upsilon_s(\rV)}\leq m+2$.
\vskip 0.1cm

Then {\small
\begin{eqnarray*}
 |K(s)|
 &\leq&
 |\theta_m(\|y_n+\zsmall \|_{\Upsilon_s(\rV)})-\theta_m(\|y_{n-1}+\zsmall \|_{\Upsilon_s(\rV)})|
\\
&&     \Big|\Big\langle \rB(y_{n-1}(s)+\zsmall (s),y_{n}(s)+\zsmall (s)),{\rA}(y_{n+1}(s)-y_n(s))\Big\rangle_{\rH}\Big|\\
 &&+
 \Big|\Big\langle \rB(y_{n}(s)+\zsmall (s),y_{n+1}(s)-y_n(s)),{\rA}(y_{n+1}(s)-y_n(s))\Big\rangle_{\rH}\Big|\\
 &&+
 \Big|\Big\langle \rB(y_{n}(s)-y_{n-1}(s),y_{n}(s)+\zsmall (s)),{\rA}(y_{n+1}(s)-y_n(s))\Big\rangle_{\rH}\Big|\\
 &=:& I_1(s)+I_2(s)+I_3(s).
\end{eqnarray*}
}
By Lemma \ref{lem B baisc prop} and the definition of $\theta_m$,
\begin{eqnarray}\label{eq es PDE Case 1 I1}
   I_1(s)
&\leq&
   C\|y_n-y_{n-1}\|_{\Upsilon_s(\rV)}|y_{n-1}(s)+\zsmall (s)|^{\frac{1}{2}}_{\rH}\|y_{n-1}(s)+\zsmall (s)\|^{\frac{1}{2}}_{\rV}\\
                 && \|y_{n}(s)+\zsmall (s)\|^{\frac{1}{2}}_{\rV}\|y_{n}(s)+\zsmall (s)\|^{\frac{1}{2}}_{\mathcal{D}({\rA})}
                  \|y_{n+1}(s)-y_n(s)\|_{\mathcal{D}({\rA})}\nonumber\\
&\leq&
  C_m\|y_n-y_{n-1}\|_{\Upsilon_s(\rV)}\|y_{n}(s)+\zsmall (s)\|^{\frac{1}{2}}_{\mathcal{D}({\rA})}\|y_{n+1}(s)-y_n(s)\|_{\mathcal{D}({\rA})}\nonumber\\
&\leq&
  {\eps}\|y_{n+1}(s)-y_n(s)\|^2_{\mathcal{D}({\rA})}
  +
  \frac{C_m}{{\eps}}\|y_n-y_{n-1}\|^2_{\Upsilon_s(\rV)}\|y_{n}(s)+\zsmall (s)\|_{\mathcal{D}({\rA})},\nonumber
\end{eqnarray}
\begin{eqnarray}\label{eq es PDE Case 1 I2}
   I_2(s)
&\leq&
   C|y_{n}(s)+\zsmall (s)|^{\frac{1}{2}}_{\rH}\|y_{n}(s)+\zsmall (s)\|^{\frac{1}{2}}_{\rV}
   \|y_{n+1}(s)-y_n(s)\|^{\frac{1}{2}}_{\rV}
                  \|y_{n+1}(s)-y_n(s)\|^{\frac{3}{2}}_{\mathcal{D}({\rA})}\nonumber\\
&\leq&
  C_m\|y_{n+1}(s)-y_n(s)\|^{\frac{1}{2}}_{\rV}
                  \|y_{n+1}(s)-y_n(s)\|^{\frac{3}{2}}_{\mathcal{D}({\rA})}\nonumber\\
&\leq&
  {\eps}^{4/3}\|y_{n+1}(s)-y_n(s)\|^2_{\mathcal{D}({\rA})}
  +
  \frac{C_m}{{\eps}^4}\|y_{n+1}(s)-y_n(s)\|^2_{\rV},
\end{eqnarray}
and
\begin{eqnarray}\label{eq es PDE Case 1 I3}
   I_3(s)
&\leq&
   C|y_{n}(s)-y_{n-1}(s)|^{\frac{1}{2}}_{\rH}\|y_{n}(s)-y_{n-1}(s)\|^{\frac{1}{2}}_{\rV}\\
   &&\|y_{n}(s)+\zsmall (s)\|^{\frac{1}{2}}_{\rV}\|y_{n}(s)+\zsmall (s)\|^{\frac{1}{2}}_{\mathcal{D}({\rA})}
                  \|y_{n+1}(s)-y_n(s)\|_{\mathcal{D}({\rA})}\nonumber\\
&\leq&
  C_m\|y_{n}(s)-y_{n-1}(s)\|_{\rV}\|y_{n}(s)+\zsmall (s)\|^{\frac{1}{2}}_{\mathcal{D}({\rA})}
                  \|y_{n+1}(s)-y_n(s)\|_{\mathcal{D}({\rA})}\nonumber\\
&\leq&
  {\eps}\|y_{n+1}(s)-y_n(s)\|^2_{\mathcal{D}({\rA})}
  +
  \frac{C_m}{{\eps}}\|y_{n}(s)-y_{n-1}(s)\|^2_{\rV}\|y_{n}(s)+\zsmall (s)\|_{\mathcal{D}({\rA})}.
  \nonumber
\end{eqnarray}

 {Therefore, since
$$\int_0^t\|y_{n}(s)+\zsmall (s)\|_{\mathcal{D}({\rA})}\mathds{1}_{\{\|y_n+\zsmall \|_{\Upsilon_s(\rV)}\leq m+2\}}(s)\,ds\leq C_mt^{\frac{1}{2}},$$
 by (\ref{eq es PDE Case 1 I1})-(\ref{eq es PDE Case 1 I3}), we deduce }\\
\begin{eqnarray}\label{Claim 1 PDE}
&&\hspace{-1.4truecm}\lefteqn{ \int_0^t|K(s)|\mathds{1}_{\{\|y_n+\zsmall \|_{\Upsilon_s(\rV)}\vee \|y_{n-1}+\zsmall \|_{\Upsilon_s(\rV)}\leq m+2\}}(s)\,ds}\\
&\leq&
  (2{\eps}+{\eps}^{4/3})\int_0^t\|y_{n+1}(s)-y_n(s)\|^2_{\mathcal{D}({\rA})}\,ds
  +
  \frac{C_m}{{\eps}^4}t\|y_{n+1}-y_{n}\|^2_{\Upsilon_t(\rV)}\nonumber\\
  &&+
  \frac{C_m}{{\eps}}\|y_n-y_{n-1}\|^2_{\Upsilon_t(\rV)}
  \int_0^t\|y_{n}(s)+\zsmall (s)\|_{\mathcal{D}({\rA})}\mathds{1}_{\{\|y_n+\zsmall \|_{\Upsilon_s(\rV)}\leq m+2\}}(s)\,ds\nonumber\\
&\leq&
  (2{\eps}+{\eps}^{4/3}+\frac{C_m}{{\eps}^{4}}t)
\|y_{n+1}-y_{n}\|^2_{\Upsilon_t(\rV)}
  +
  \frac{C_m}{{\eps}}t^{\frac{1}{2}}\|y_n-y_{n-1}\|^2_{\Upsilon_t(\rV)}, \; t\in [0,T].
  \nonumber
\end{eqnarray}

\vskip 0.1cm

\noindent{\bf Case 2} {Assume that $\|y_n+\zsmall \|_{\Upsilon_s(\rV)}\leq m+2$ and $\|y_{n-1}+\zsmall \|_{\Upsilon_s(\rV)}>m+2$.}

\vskip 0.1cm
Then, the definition of $\theta_m$ implies that
$$
K(s)= \theta_m(\|y_n+\zsmall \|_{\Upsilon_s(\rV)})
     \Big\langle \rB(y_{n}(s)+\zsmall (s),y_{n+1}(s)+\zsmall (s)),{\rA}(y_{n+1}(s)-y_n(s))\Big\rangle_{\rH},
$$
and
\begin{eqnarray}\label{Claim 2 sss}
&&\hspace{-1truecm}\lefteqn{\int_0^t|K(s)|\mathds{1}_{\{\|y_n+\zsmall \|_{\Upsilon_s(\rV)}\leq m+2\}}(s)\mathds{1}_{\{\|y_{n-1}+\zsmall \|_{\Upsilon_s(\rV)}>m+2\}}(s)\,ds}\nonumber
\\
&=&{\int_0^t|K(s)|\mathds{1}_{\{\|y_n+\zsmall \|_{\Upsilon_s(\rV)}\leq m+1\}}(s)\mathds{1}_{\{\|y_{n-1}+\zsmall \|_{\Upsilon_s(\rV)}>m+2\}}(s)\,ds}.
\end{eqnarray}

For any $s\in[0,t]$ such that
$$
\mathds{1}_{\{\|y_n+\zsmall \|_{\Upsilon_s(\rV)}\leq m+1\}}(s)\mathds{1}_{\{\|y_{n-1}+\zsmall \|_{\Upsilon_s(\rV)}>m+2\}}(s)=1,
$$
we have
\begin{eqnarray}\label{Claim 2 s}
\|y_n-y_{n-1}\|_{\Upsilon_s(\rV)}\geq \|y_{n-1}+\zsmall \|_{\Upsilon_s(\rV)}-\|y_n+\zsmall \|_{\Upsilon_s(\rV)} \geq 1,
\end{eqnarray}
and by Lemma \ref{lem B baisc prop},
\begin{eqnarray}\label{eq esta PDE Case 2 01}
  |K(s)|
&\leq&
   C|y_{n}(s)+\zsmall (s)|^{\frac{1}{2}}_{\rH}\|y_{n}(s)+\zsmall (s)\|^{\frac{1}{2}}_{\rV}\nonumber\\
                 && \|y_{n+1}(s)+\zsmall (s)\|^{\frac{1}{2}}_{\rV}\|y_{n+1}(s)+\zsmall (s)\|^{\frac{1}{2}}_{\mathcal{D}({\rA})}
                  \|y_{n+1}(s)-y_n(s)\|_{\mathcal{D}({\rA})}\nonumber\\
&\leq&C_m \|y_{n+1}(s)+\zsmall (s)\|^{\frac{1}{2}}_{\rV}\|y_{n+1}(s)+\zsmall (s)\|^{\frac{1}{2}}_{\mathcal{D}({\rA})}
                  \|y_{n+1}(s)-y_n(s)\|_{\mathcal{D}({\rA})}\\
&\leq&
   {\eps} \|y_{n+1}(s)-y_n(s)\|^2_{\mathcal{D}({\rA})}
   +
   \frac{C_m}{{\eps}} \|y_{n+1}(s)+\zsmall (s)\|_{\rV}\|y_{n+1}(s)+\zsmall (s)\|_{\mathcal{D}({\rA})}.\nonumber
\end{eqnarray}
For the second term of the above inequality, we have
\begin{eqnarray}\label{eq Case 02 sec term}
&&\hspace{-1.3truecm}\lefteqn{\frac{C_m}{{\eps}} \|y_{n+1}(s)+\zsmall (s)\|_{\rV}\|y_{n+1}(s)+\zsmall (s)\|_{\mathcal{D}({\rA})}}\\
&&\hspace{-0.9truecm}\leq
   \frac{C_m}{{\eps}} (\|y_{n}(s)+\zsmall (s)\|_{\rV}+\|y_{n+1}(s)-y_{n}(s)\|_{\rV})
   (\|y_{n}(s)+\zsmall (s)\|_{\mathcal{D}({\rA})}
   \lefteqn{ +\|y_{n+1}(s)-y_{n}(s)\|_{\mathcal{D}({\rA})})}
  \nonumber\\
&&\hspace{-0.9truecm}\leq
   \frac{C_m}{{\eps}}\Big(
            \|y_{n}(s)+\zsmall (s)\|_{\mathcal{D}({\rA})}\|y_n-y_{n-1}\|^2_{\Upsilon_s(\rV)}
            +
            \|y_{n+1}(s)-y_{n}(s)\|_{\mathcal{D}({\rA})}\|y_n-y_{n-1}\|_{\Upsilon_s(\rV)}\nonumber\\
            &&\hspace{-0.9truecm}\ \ \ +
            \|y_{n+1}(s)-y_{n}(s)\|_{\rV}\|y_{n}(s)+\zsmall (s)\|_{\mathcal{D}({\rA})}
             \lefteqn{
            +
            \|y_{n+1}(s)-y_{n}(s)\|_{\rV}\|y_{n+1}(s)-y_{n}(s)\|_{\mathcal{D}({\rA})}
           \Big)}
           \nonumber\\
&&\hspace{-0.9truecm}\leq
   2{\eps} \|y_{n+1}(s)-y_{n}(s)\|^2_{\mathcal{D}({\rA})}
   +
   \frac{C_m}{{\eps}^3}\Big(\|y_n-y_{n-1}\|^2_{\Upsilon_s(\rV)}+\|y_{n+1}(s)-y_{n}(s)\|^2_{\rV}\Big)\nonumber\\
   &&\hspace{-0.9truecm}\ \ \ \lefteqn{ +
   \frac{C_m}{{\eps}}\Big(\|y_{n}(s)+\zsmall (s)\|_{\mathcal{D}({\rA})}\|y_n-y_{n-1}\|^2_{\Upsilon_s(\rV)}
             +
              \|y_{n+1}(s)-y_{n}(s)\|_{\rV}\|y_{n}(s)+\zsmall (s)\|_{\mathcal{D}({\rA})}
           \Big).}\nonumber
\end{eqnarray}
In the second ``$\leq$" in (\ref{eq Case 02 sec term}), we have used (\ref{Claim 2 s}) and $\|y_{n}(s)+\zsmall (s)\|_{\rV}\leq \|y_{n}+\zsmall \|_{\Upsilon_s(\rV)}\leq m+1$.

Considering (\ref{Claim 2 sss}), (\ref{eq esta PDE Case 2 01}), and (\ref{eq Case 02 sec term}) together, we deduce
\begin{eqnarray}\label{Claim 2 PDE}
&&\hspace{-1truecm}\lefteqn{\int_0^t|K(s)|\mathds{1}_{\{\|y_n+\zsmall \|_{\Upsilon_s(\rV)}\leq m+2\}}(s)\mathds{1}_{\{\|y_{n-1}+\zsmall \|_{\Upsilon_s(\rV)}>m+2\}}(s)\,ds}
\\
&\leq&
  3{\eps}\int_0^t\|y_{n+1}(s)-y_{n}(s)\|^2_{\mathcal{D}({\rA})}\,ds
  +
  \frac{C_m}{{\eps}^3} t \sup_{s\in[0,t]}\|y_{n+1}(s)-y_{n}(s)\|^2_{\rV} \nonumber\\
  &&+
  \frac{C_m}{{\eps}^3}\sup_{s\in[0,t]}\|y_{n}-y_{n-1}\|^2_{\Upsilon_t(\rV)} t\nonumber\\
  &&+
  \frac{C_m}{{\eps}}\sup_{s\in[0,t]}\|y_{n+1}(s)-y_{n}(s)\|_{\rV}
 \nonumber \\ &&\hspace{1truecm}\lefteqn{\int_0^t\|y_{n}(s)+\zsmall (s)\|_{\mathcal{D}({\rA})}\mathds{1}_{\{\|y_n+\zsmall \|_{\Upsilon_s(\rV)}\leq m+1\}}(s)\mathds{1}_{\{\|y_{n-1}+\zsmall \|_{\Upsilon_s(\rV)}>m+2\}}(s)\,ds}
 \nonumber\\
  &&+
  \frac{C_m}{{\eps}}\|y_{n}-y_{n-1}\|^2_{\Upsilon_t(\rV)}
  \int_0^t\|y_{n}(s)+\zsmall (s)\|_{\mathcal{D}({\rA})}\mathds{1}_{\{\|y_n+\zsmall \|_{\Upsilon_s(\rV)}\leq m+1\}}(s)\,ds\nonumber\\
&\leq&
  \|y_{n+1}-y_{n}\|^2_{\Upsilon_t(\rV)}\Big(3{\eps}+\frac{C_m}{{\eps}^3}t+{\eps}^2\Big)
  +
  \frac{C_m}{{\eps}^3}t\|y_{n}-y_{n-1}\|^2_{\Upsilon_t(\rV)}\nonumber\\
  &&+
  \frac{C_m}{{\eps}^4}\Big(\int_0^t\|y_{n}(s)+\zsmall (s)\|_{\mathcal{D}({\rA})}\mathds{1}_{\{\|y_n+\zsmall \|_{\Upsilon_s(\rV)}\leq m+1\}}(s)\mathds{1}_{\{\|y_{n-1}+\zsmall \|_{\Upsilon_s(\rV)}>m+2\}}(s)\nonumber \\
  &&\hspace{4truecm}\lefteqn{
  \cdot \|y_{n}-y_{n-1}\|_{\Upsilon_s(\rV)}\,ds\Big)^2}\nonumber\\
  &&+
  \frac{C_m}{{\eps}}\|y_{n}-y_{n-1}\|^2_{\Upsilon_t(\rV)}
  \Big(\int_0^t\|y_{n}(s)+\zsmall (s)\|^2_{\mathcal{D}({\rA})}\mathds{1}_{\{\|y_n+\zsmall \|_{\Upsilon_s(\rV)}\leq m+1\}}(s)\,ds\Big)^{\frac{1}{2}}t^{\frac{1}{2}}\nonumber\\
&\leq&
  \|y_{n+1}-y_{n}\|^2_{\Upsilon_t(\rV)}\Big(3{\eps}+\frac{C_m}{{\eps}^3}t+{\eps}^2\Big)
  +
  \frac{C_m}{{\eps}^3}t\|y_{n}-y_{n-1}\|^2_{\Upsilon_t(\rV)}\nonumber\\
  &&+
  \frac{C_m}{{\eps}^4}\|y_{n}-y_{n-1}\|^2_{\Upsilon_t(\rV)}t\Big(\int_0^t\|y_{n}(s)+\zsmall (s)\|^2_{\mathcal{D}({\rA})}\mathds{1}_{\{\|y_n+\zsmall \|_{\Upsilon_s(\rV)}\leq m+1\}}(s) ds\Big)\nonumber\\
  &&+
  \frac{C_m}{{\eps}}\|y_{n}-y_{n-1}\|^2_{\Upsilon_t(\rV)}t^{\frac{1}{2}}\nonumber\\
&\leq&
  \|y_{n+1}-y_{n}\|^2_{\Upsilon_t(\rV)}\Big(3{\eps}+\frac{C_m}{{\eps}^3}t+{\eps}^2\Big)
  +
\|y_{n}-y_{n-1}\|^2_{\Upsilon_t(\rV)}\Big(\frac{C_m}{{\eps}^3}t+\frac{C_m}{{\eps}}t^{\frac{1}{2}}+\frac{C_m}{{\eps}^4}t\Big).\nonumber
\end{eqnarray}
In the second ``$\leq$" in (\ref{Claim 2 PDE}), we have used (\ref{Claim 2 s}).

\vskip 0.1cm
\noindent {\bf Case 3:} Assume that $\|y_n+\zsmall \|_{\Upsilon_s(\rV)}> m+2$ and $\|y_{n-1}+\zsmall \|_{\Upsilon_s(\rV)}\leq m+2$.

\vskip 0.1cm
The definition of $\theta_m$ implies that
$$
K(s)= -\theta_m(\|y_{n-1}+\zsmall \|_{\Upsilon_s(\rV)})
     \Big\langle \rB(y_{n-1}(s)+\zsmall (s),y_{n}(s)+\zsmall (s)),{\rA}(y_{n+1}(s)-y_n(s))\Big\rangle_{\rH},
$$
and
\begin{eqnarray}\label{Claim 3 pss}
&&\hspace{-1.4truecm}\int_0^t|K(s)|\mathds{1}_{\{\|y_n+\zsmall \|_{\Upsilon_s(\rV)}> m+2\}}(s)\mathds{1}_{\{\|y_{n-1}+\zsmall \|_{\Upsilon_s(\rV)}\leq m+2\}}(s)\,ds\nonumber\\
&=&
\int_0^t|K(s)|\mathds{1}_{\{\|y_n+\zsmall \|_{\Upsilon_s(\rV)}> m+2\}}(s)\mathds{1}_{\{\|y_{n-1}+\zsmall \|_{\Upsilon_s(\rV)}\leq m+1\}}(s)\,ds.
\end{eqnarray}

For any $s\in[0,t]$ such that
$$
\mathds{1}_{\{\|y_n+\zsmall \|_{\Upsilon_s(\rV)}> m+2\}}(s)\mathds{1}_{\{\|y_{n-1}+\zsmall \|_{\Upsilon_s(\rV)}\leq m+1\}}(s)=1,
$$
we have
\begin{eqnarray}\label{Claim 3 pp}
\|y_n-y_{n-1}\|_{\Upsilon_s(\rV)}\geq \|y_n+\zsmall \|_{\Upsilon_s(\rV)}-\|y_{n-1}+\zsmall \|_{\Upsilon_s(\rV)} \geq 1,
\end{eqnarray}
and by Lemma \ref{lem B baisc prop} again,
\begin{eqnarray}\label{eq esta PDE Case 3 01}
  |K(s)|
&\leq&
   C|y_{n-1}(s)+\zsmall (s)|^{\frac{1}{2}}_{\rH}\|y_{n-1}(s)+\zsmall (s)\|^{\frac{1}{2}}_{\rV}\nonumber\\
                 && \|y_{n}(s)+\zsmall (s)\|^{\frac{1}{2}}_{\rV}\|y_{n}(s)+\zsmall (s)\|^{\frac{1}{2}}_{\mathcal{D}({\rA})}
                  \|y_{n+1}(s)-y_n(s)\|_{\mathcal{D}({\rA})}\\
&\leq&C_m \|y_{n}(s)+\zsmall (s)\|^{\frac{1}{2}}_{\rV}\|y_{n}(s)+\zsmall (s)\|^{\frac{1}{2}}_{\mathcal{D}({\rA})}
                  \|y_{n+1}(s)-y_n(s)\|_{\mathcal{D}({\rA})}\nonumber\\
&\leq&
   {\eps} \|y_{n+1}(s)-y_n(s)\|^2_{\mathcal{D}({\rA})}
   +
   \frac{C_m}{{\eps}} \|y_{n}(s)+\zsmall (s)\|_{\rV}\|y_{n}(s)+\zsmall (s)\|_{\mathcal{D}({\rA})}.\nonumber
\end{eqnarray}
Using similar arguments as (\ref{eq Case 02 sec term}),
we have
\begin{eqnarray}\label{eq Case 03 sec term}
&&\hspace{-0.5truecm}\lefteqn{\frac{C_m}{{\eps}} \|y_{n}(s)+\zsmall (s)\|_{\rV}\|y_{n}(s)+\zsmall (s)\|_{\mathcal{D}({\rA})}}\\
&&\hspace{-0.5truecm}\leq
   \frac{C_m}{{\eps}} (\|y_{n-1}(s)\!+\!\zsmall (s)\|_{\rV}\!+\!\|y_{n}(s)\!-\!y_{n-1}(s)\|_{\rV})
   (\|y_{n-1}(s)\!+\!\zsmall (s)\|_{\mathcal{D}({\rA})}\!+\!\|y_{n}(s)\!-\!y_{n-1}(s)\|_{\mathcal{D}({\rA})})\nonumber\\
&&\hspace{-0.5truecm}\leq
   \frac{C_m}{{\eps}}\Big(
            \|y_{n-1}(s)+\zsmall (s)\|_{\mathcal{D}({\rA})}\|y_n-y_{n-1}\|^2_{\Upsilon_s(\rV)}
            +
            \|y_{n}(s)-y_{n-1}(s)\|_{\mathcal{D}({\rA})}\|y_n-y_{n-1}\|_{\Upsilon_s(\rV)}\nonumber\\
            &&\hspace{-1.5truecm}\ \ \ \ \ \ \ \ \ \ \ +
            \|y_{n}(s)-y_{n-1}(s)\|_{\rV}\|y_{n-1}(s)+\zsmall (s)\|_{\mathcal{D}({\rA})}
            +
            \|y_{n}(s)-y_{n-1}(s)\|_{\rV}\|y_{n}(s)-y_{n-1}(s)\|_{\mathcal{D}({\rA})}
           \Big)\nonumber\\
&&\hspace{-0.5truecm}\leq
   2{\eps} \|y_{n}(s)-y_{n-1}(s)\|^2_{\mathcal{D}({\rA})}
   +
   \frac{C_m}{{\eps}^3}\Big(\|y_n-y_{n-1}\|^2_{\Upsilon_s(\rV)}+\|y_{n}(s)-y_{n-1}(s)\|^2_{\rV}\Big)\nonumber\\
   &&\hspace{-0.5truecm}\ \ \ \ +
   \frac{C_m}{{\eps}}\|y_{n-1}(s)+\zsmall (s)\|_{\mathcal{D}({\rA})}\|y_n-y_{n-1}\|^2_{\Upsilon_s(\rV)}
           \nonumber.
\end{eqnarray}
In the second and third ``$\leq$" of (\ref{eq Case 03 sec term}), we have used (\ref{Claim 3 pp}) and
$$
\|y_{n-1}(s)+\zsmall (s)\|_{\rV}\leq\|y_{n-1}+\zsmall \|_{\Upsilon_s(\rV)}\leq m+1\ \text{and} \ \ \|y_{n}(s)-y_{n-1}(s)\|_{\rV}\leq \|y_n-y_{n-1}\|^2_{\Upsilon_s(\rV)}.
$$

 Considering (\ref{Claim 3 pss}), (\ref{eq esta PDE Case 3 01}), and (\ref{eq Case 03 sec term}) together, and using the same idea as in (\ref{Claim 2 PDE}), we deduce
\begin{eqnarray}\label{Claim 3 PDE}
&&\hspace{-1truecm}\int_0^t|K(s)|\mathds{1}_{\{\|y_n+\zsmall \|_{\Upsilon_s(\rV)}> m+2\}}(s)\mathds{1}_{\{\|y_{n-1}+\zsmall \|_{\Upsilon_s(\rV)}\leq m+2\}}(s)\,ds\\
&\leq&
  {\eps}\|y_{n+1}-y_{n}\|^2_{\Upsilon_t(\rV)}
  +
\|y_{n}-y_{n-1}\|^2_{\Upsilon_t(\rV)}\Big(3{\eps}+\frac{C_m}{{\eps}^3}t+\frac{C_m}{{\eps}}t^{\frac{1}{2}}\Big).\nonumber
\end{eqnarray}

We have now finished the estimates for $K(s)$ in the three cases.
\vskip 0.2cm

The statements made in (\ref{Claim 1 PDE}), (\ref{Claim 2 PDE}), and (\ref{Claim 3 PDE}), combined with equality (\ref{App auxi esti PDE 01}), allow us to arrive at the following.
For all ${\eps}>0$ and $t {\in [0,T]}$,
\begin{eqnarray}\label{App auxi esti PDE star 01}
&&\hspace{-1truecm}\lefteqn{  \|y_{n+1}(t)-y_n(t)\|^2_{\rV}+2\int_0^t\|y_{n+1}(s)-y_n(s)\|^2_{\mathcal{D}({\rA})}\,ds}\\
&\leq&
   C_m\|y_{n+1}-y_{n}\|^2_{\Upsilon_t(\rV)}\Big({\eps}+{\eps}^{4/3}+\frac{t}{{\eps}^4}+{\eps}^2+\frac{t}{{\eps}^3}\Big)
   +
   C_m\|y_{n}-y_{n-1}\|^2_{\Upsilon_t(\rV)}\Big({\eps}+\frac{t}{{\eps}^4}+\frac{t^{\frac12}}{{\eps}}+\frac{t}{{\eps}^3}\Big).\nonumber
\end{eqnarray}
Since, {by the definition of the space $\Upsilon_t(\rV)$},
$$
\frac{1}{2}\|y_{n+1}-y_n\|^2_{\Upsilon_t(\rV)}\leq \sup_{s\in[0,t]}\|y_{n+1}(s)-y_n(s)\|^2_{\rV}+\int_0^t\|y_{n+1}(s)-y_n(s)\|^2_{\mathcal{D}({\rA})}\,ds,
$$
we can choose ${\eps}$ and $t_0>0$ small enough, such that for all $n\geq 1$,
 \begin{eqnarray}\label{App auxi esti PDE star 01}
\|y_{n+1}-y_{n}\|^2_{\Upsilon_{t_0}(\rV)}
\leq
   \frac{1}{3}\|y_{n}-y_{n-1}\|^2_{\Upsilon_{t_0}(\rV)}.
\end{eqnarray}
This implies that $\{y_n,n\in\mathbb{N}\}$ is a Cauchy sequence in $C([0,t_0],\rV)\cap L^2([0,t_0],\mathcal{D}({\rA}))$.
Therefore, it has a unique limit in that space which we denote by $\ysmall ^1$. Let us note that it is rather standard (if not obvious) to prove that $\ysmall ^1$ is a solution of (\ref{eq lemma PDE 1 01}) on the time interval $[0,t_0]$.
Observe that the constant $t_0$ does not depend on the initial data.

Next, we consider
\begin{eqnarray}\label{eq lemma PDE 1 02}
&& \ysmall ^\prime(t)+{\rA}\ysmall (t)+\theta_m(\|\ysmall +\zsmall \|_{\Upsilon_t(\rV)})\rB(\ysmall (t)+\zsmall (t))=f(t),\ t>t_0,\nonumber\\
&& \ysmall (t)=\ysmall ^1(t),\ t\in[0,t_0].
\end{eqnarray}
 Repeating the above arguments, we can solve
(\ref{eq lemma PDE 1 02}) on interval $[0,2t_0]$, and denote its solution by $\ysmall ^2:=\{\ysmall ^2(t),\ t\in[0,2t_0]\}$. It is not difficult to prove that
 $\ysmall ^2$ is a solution of (\ref{eq lemma PDE 1 01}) on the time interval $[0, 2t_0]$.
Then, by induction, we can solve (\ref{eq lemma PDE 1 01})
on $[0,3t_0]$, $[0,4t_0]$, and so on. We finally
obtain a solution $\ysmall \in C([0,T],\rV)\cap L^2([0,T],\mathcal{D}({\rA}))$ of (\ref{eq lemma PDE 1 01}){ for any fixed $T>0$}.
The proof of Lemma \ref{lem PDE 1} is complete.

\end{proof}

\vskip 0.1cm

Although the uniqueness of solutions to equation \eqref{eq lemma PDE 1 01} follows from the existence proof, for completeness' sake,  we give an independent proof of this property (see also Corollary \ref{cor PDE 2}). The following lemma is a preliminary step in this direction.
It will also be used later.

Let us recall that the space $\Lambda_T(\rV)$ (and its norm) was defined around equality \eqref{eqn-Lambda_T-V}.

\begin{lemma}\label{lem PDE 2} {Assume that $n\in\mathbb{N}$ and $T>0$.}
Assume that {for all $u_0\in \rV$, $f\in L^2([0,T];\rH)$ and } $y\in \Lambda_T(\rV)$, there exists an element $u=\Phi^y\in \Lambda_T(\rV)$ satisfying
\begin{eqnarray}\label{eq SNS PDE jieduan 2}
&& du(t)+{\rA}u(t)\,dt+\theta_n(\|u\|_{\Upsilon_t(\rV)})\rB(u(t))\,dt=f(t)\,dt + \int_{{\rZ }}G(y(t-),z)\widetilde{\eta}(dz,dt),\nonumber\\
&& u(0)=u_0.
\end{eqnarray}
Then there exist a positive constant $C$ and a function $L_{n}: (0,\infty) \to (0,\infty) $ such that $\lim_{T\to 0}L_{n}(T)=1$ and
\begin{eqnarray}\label{eq BPFT PDE 01}
\|\Phi^{y_1}-\Phi^{y_2}\|_{\Lambda_T(\rV)}^2\leq CL^2_{n}(T)T\|y_1-y_2\|^2_{\Lambda_T(\rV)},\ \ \ \forall y_1, y_2\in \Lambda_T(\rV).
\end{eqnarray}

\end{lemma}

\begin{proof}[Proof of Lemma \ref{lem PDE 2}] Choose and fix $n\in\mathbb{N}$ and $T>0$.
Assume that  $u_0\in \rV$, $f\in L^2([0,T];\rH)$ and  $y_1,y_2\in \Lambda_T(\rV)$.

For simplicity, let us set $u_1=\Phi^{y_1}$, $u_2=\Phi^{y_2}$, and $\usmall =u_1-u_2$.
By {the It\^o} formula, we have
\begin{eqnarray}\label{eq1 Z 0}
&&\hspace{-1truecm}\|\usmall (t)\|^2_{\rV}+2\int_0^t\|\usmall (s)\|^2_{\mathcal{D}({\rA})}\,ds\nonumber\\
&=&
 -2\int_0^t\Big\langle \theta_n(\|u_1\|_{\Upsilon_s(\rV)}) \rB(u_1(s))-\theta_n(\|u_2\|_{\Upsilon_s(\rV)}) \rB(u_2(s)),{\rA}\usmall (s)\Big\rangle_{\rH}\,ds\nonumber\\
 &&+
 2\int_0^t\int_{{\rZ}}\Big\langle G(y_1(s-),z)-G(y_2(s-),z),\usmall (s-)\Big\rangle_{\rV} \widetilde{\eta}(dz,ds)\nonumber\\
 &&+
 \int_0^t\int_{{\rZ}}\|G(y_1(s-),z)-G(y_2(s-),z)\|^2_{\rV} \eta(dz,ds)\nonumber\\
&=:&
 J_1(t)+J_2(t)+ J_3(t).
\end{eqnarray}

For the first term, $J_1$, we have
\begin{equation}\label{eq1 J1}|J_1(t)|
\leq
 \frac{1}{2} \int_0^t\|\usmall (s)\|^2_{\mathcal{D}({\rA})}\,ds
 +
 2 \int_0^t\,K(s)ds
\end{equation}
where
\begin{equation*}\label{eqn-K(s)}
K(s):=\Big|\theta_n(\|u_1\|_{\Upsilon_s(\rV)}) \rB(u_1(s))-\theta_n(\|u_2\|_{\Upsilon_s(\rV)})\rB(u_2(s))\Big|^2_{\rH},\;\;\; s\in [0,T].
\end{equation*}

To find suitable bounds on $K$, we consider four cases.
By the property  of $\theta_n$ (see (\ref{eqn-theta_m})) and the Minkowski inequality, we have the following estimates.
\begin{itemize}
\item[(\textbf{1})] For $\|u_1\|_{\Upsilon_s(\rV)}\vee\|u_2\|_{\Upsilon_s(\rV)}\leq n+1$, we have
  \begin{eqnarray*}
&&\hspace{-1.6truecm}K(s)
   \leq
      C \Big[
        |\rB(u_1(s))-\rB(u_2(s))|^2_{\rH}
      +
       \Big|\theta_n(\|u_1\|_{\Upsilon_s(\rV)})-\theta_n(\|u_2\|_{\Upsilon_s(\rV)})\Big|^2|\rB(u_2(s))|^2_{\rH}
       \Big]\\
   &&\hspace{-0.6truecm}\leq
     C \Big[
       |u_1(s)|_{\rH}\|u_1(s)\|_{\rV}\|u(s)\|_{\rV}\|u(s)\|_{\mathcal{D}({\rA})}\!+\!\|u_2(s)\|_{\rV}\|u_2(s)\|_{\mathcal{D}({\rA})}|u(s)|_{\rH}\|u(s)\|_{\rV}
       \Big]\\
       &&\hspace{-0.6truecm}\ \ \ +
     C|u_2(s)|_{\rH}\|u_2(s)\|^2_{\rV}\|u_2(s)\|_{\mathcal{D}({\rA})}\|u\|^2_{\Upsilon_s(\rV)}\\
   &&\hspace{-0.6truecm}\leq
     \frac{1}{4} \|u(s)\|^2_{\mathcal{D}({\rA})}\\
     &&\hspace{-0.6truecm}\ \ \ +
     C\|u\|^2_{\Upsilon_s(\rV)}
       \Big[
       |u_1(s)|^2_{\rH}\|u_1(s)\|^2_{\rV}+\|u_2(s)\|_{\rV}\|u_2(s)\|_{\mathcal{D}({\rA})}+\|u_2(s)\|^3_{\rV}\|u_2(s)\|_{\mathcal{D}({\rA})}
       \Big].
  \end{eqnarray*}

\item[(\textbf{2})] For $\|u_1\|_{\Upsilon_s(\rV)}\leq n+1$ and $\|u_2\|_{\Upsilon_s(\rV)}\geq n+1$, we have
  \begin{eqnarray*}
      K(s)\!\!\!&=&\!\!\!
          \Big|\theta_n(\|u_1\|_{\Upsilon_s(\rV)}) \rB(u_1(s))\Big|^2_{\rH}\\
        \!\!\!&=&\!\!\!
          \Big|\theta_n(\|u_1\|_{\Upsilon_s(\rV)})-\theta_n(\|u_2\|_{\Upsilon_s(\rV)})\Big|^2| \rB(u_1(s))|^2_{\rH}\\
        &\leq&\!\!\!
          C|u_1(s)|_{\rH}\|u_1(s)\|^2_{\rV}\|u_1(s)\|_{\mathcal{D}({\rA})}\|u\|^2_{\Upsilon_s(\rV)}.
  \end{eqnarray*}

\item[(\textbf{3})] For $\|u_1\|_{\Upsilon_s(\rV)}\geq n+1$ and $\|u_2\|_{\Upsilon_s(\rV)}\leq n+1$, similar to Case (2), we get
  \begin{eqnarray*}
      K(s)
        \leq
          C|u_2(s)|_{\rH}\|u_2(s)\|^2_{\rV}\|u_2(s)\|_{\mathcal{D}({\rA})}\|u\|^2_{\Upsilon_s(\rV)}.
  \end{eqnarray*}

\item[(\textbf{4})] For $\|u_1\|_{\Upsilon_s(\rV)}\wedge \|u_2\|_{\Upsilon_s(\rV)}\geq n+1$, we have
  \begin{eqnarray*}
      K(s)=0.
  \end{eqnarray*}

\end{itemize}
Hence,
\begin{eqnarray}\label{eq1 K}
  K(s)
\leq
  \frac{1}{4} \|u(s)\|^2_{\mathcal{D}({\rA})}
     +
     C\|u\|^2_{\Upsilon_s(\rV)}\Xi(s), \;\;\; s\in [0,T],
\end{eqnarray}
where
\begin{eqnarray*}
\Xi(s)\!\!\!&:=&\!\!\!\Big(|u_1(s)|^2_{\rH}\|u_1(s)\|^2_{\rV}+\|u_1(s)\|^3_{\rV}\|u_1(s)\|_{\mathcal{D}({\rA})}\Big)
\mathds{1}_{[0,n+1]}(\|u_1\|_{\Upsilon_s(\rV)})\nonumber\\
       &&\hspace{-0.2truecm}+
       \Big(\|u_2(s)\|_{\rV}\|u_2(s)\|_{\mathcal{D}({\rA})}+\|u_2(s)\|^3_{\rV}\|u_2(s)\|_{\mathcal{D}({\rA})}\Big) \mathds{1}_{[0,n+1]}(\|u_2\|_{\Upsilon_s(\rV)}), \;\;\; s\in [0,T].
\end{eqnarray*}
Set
$$
\Theta(t):=\sup_{s\in[0,t]}\|u(t)\|^2_{\rV}+\int_0^t\|u(s)\|^2_{\mathcal{D}({\rA})}\,ds.
$$
Substituting (\ref{eq1 K}) into (\ref{eq1 J1}), and then into (\ref{eq1 Z 0}), and noting that
$$
\|u\|^2_{\Upsilon_s(\rV)}\leq 2 \Theta(s), \;\;\; s\in [0,T],
$$
we deduce that
\begin{eqnarray}\label{eq1 local 01}
  \Theta(t)
\leq
  C\int_0^t\Theta(s)\Xi(s)
    ds
    +
    \sup_{s\in[0,t]}|J_2(s)|
    +
    J_3(t), \;\;\; t\in [0,T].
\end{eqnarray}
Then, Gronwall's lemma implies that
\begin{eqnarray}\label{eq1 local 02}
  \Theta(T)
&\leq&
  \Big(\sup_{t\in[0,T]}|J_2(t)|
    +
    J_3(T)
  \Big)
  \cdot
  e^{C\int_0^T\Xi(s)
    ds}\nonumber\\
&\leq&
  \Big(\sup_{t\in[0,T]}|J_2(t)|
    +
    J_3(T)
  \Big)\times e^{C\Big(n^4T+n^4 T^{\frac{1}{2}}+n^2 T^{\frac{1}{2}}\Big)}.
\end{eqnarray}
Set
\begin{equation}\label{eqn-L_n(T)}
L_{n}(T)=e^{C\Big(n^4T+n^4 T^{\frac{1}{2}}+n^2 T^{\frac{1}{2}}\Big)}.
\end{equation}
By the Burkholder-Davis-Gundy inequality and Assumption \ref{con1 G}, we have
\begin{eqnarray}\label{eq1 J2}
  L_{n}(T)\mathbb{E}\Big(\sup_{t\in[0,T]}|J_2(t)|\Big)
\leq
  \frac{1}{2}\|u\|^2_{\Lambda_T(\rV)}
  +
  CL_{n}(T)^2T\|y_1-y_2\|^2_{\Lambda_T(\rV)}.
\end{eqnarray}
and
\begin{eqnarray}\label{eq11 J3}
  \mathbb{E}\Big(J_3(T)\Big)
\leq
  CT\|y_1-y_2\|^2_{\Lambda_T(\rV)}.
\end{eqnarray}
Summing up (\ref{eq1 local 02}), (\ref{eq1 J2}), and (\ref{eq11 J3}), we have
\begin{eqnarray*}
\|u\|^2_{\Lambda_T(\rV)}\leq CL^2_{n}(T) T \|y_1-y_2\|^2_{\Lambda_T(\rV)}.
\end{eqnarray*}

Notice that in view of the definition \eqref{eqn-L_n(T)}, $L_{n}(T)\to 1$ as $T\to 0$.

The proof of Lemma \ref{lem PDE 2} is thus complete.

\end{proof}

Using arguments similar to those for Corollary \ref{cor 1}, by Lemma \ref{lem PDE 2}, we have
\begin{corollary}\label{cor PDE 2}
Under the same assumptions as in Lemma \ref{lem PDE 1}, the solution of Problem (\ref{eq lemma PDE 1 01}) is unique.
\end{corollary}
\vskip 0.2cm

\noindent

Now we are in a position to prove Theorem \ref{thm-initial data in V}. But before we do so, for the benefit of a reader, let us make the following remark.
Conditions  (2.15) and (2.16), { the auxiliary problems (3.4)\footnote{Let us remark that (3.4) in \cite{BHR15} should read as follows
$$
du_n(t)+\rA u_n(t)dt=-(\rm B_n^Tv)(t)dt+\int_\rZ G(z,v(t-))\widetilde{\eta}(dz,dt).
$$}
and (3.13) and   their proofs of the existence and uniqueness of solutions in \cite{BHR15},}  correspond to properties (\ref{eqn-theta_m}), {the auxiliary problems (\ref{eq1 SNS jieduan 2})
 and (\ref{eq1 SNS jieduan 1}) and their proofs of the existence and uniqueness of solutions}  in the present paper.

\begin{proof}[\bf Proof Theorem \ref{thm-initial data in V}]

We also use the Banach fixed point theorem to give this proof in three steps.
\vskip 0.3cm

\noindent
{\bf Step 1: Local existence}. Consider the following auxiliary problem
\begin{eqnarray}\label{eq1 SNS jieduan 1}
&& du_n(t)+{\rA}u_n(t)\,dt+\theta_n(\|u_n\|_{\Upsilon_t(\rV)})\rB(u_n(t))\,dt=f(t)\,dt+ \int_{{\rZ}}G(u_n(t-),z)\widetilde{\eta}(dz,dt),\nonumber\\
&& u_n(0)=u_0.
\end{eqnarray}

Choose and fix $T>0$. For any $y\in \Lambda_T(\rV)$, there exists a unique element $u=\Phi^y$ such that $u \in D([0,T],\rV)\cap L^2([0,T],\mathcal{D}({\rA}))=\Upsilon_T(\rV)$, $\mathbb{P}$-a.s. and
\begin{eqnarray}\label{eq1 SNS jieduan 2}
&& du(t)+{\rA}u(t)\,dt+\theta_n(\|u\|_{\Upsilon_t(\rV)})\rB(u(t))\,dt=f(t)\,dt + \int_{{\rZ}}G(y(t-),z)\widetilde{\eta}(dz,dt),\nonumber\\
&& u(0)=u_0.
\end{eqnarray}

This result can be seen as follows. It is known that there exists a unique $\mathbb{F}$-{progressively measurable} process $\zbig\in \Upsilon_T(\rV)$ satisfying the following stochastic Langevin equation
\begin{eqnarray}\label{eq1 OU}
&&d\zbig (t)+{\rA}\zbig (t)\,dt=\int_{{\rZ}}G(y(t-),z)\widetilde{\eta}(dz,dt),\\
&&\zbig(0)=0.\nonumber
\end{eqnarray}
Moreover, this process, called an Ornstein-Uhlenbeck process, satisfies
\begin{eqnarray*}
\mathbb{E}\Big(\sup_{t\in[0,T]}\|\zbig (t)\|^2_{\rV}\Big)
+
\mathbb{E}\Big(\int_0^T\|\zbig(t)\|^2_{\mathcal{D}({\rA})}\,dt\Big)
\leq
CT\Big(\mathbb{E}\Big(\sup_{t\in[0,T]}\|y(t)\|^2_{\rV}\Big)+1\Big).
\end{eqnarray*}

For any $\omega\in \Omega$, consider the following deterministic PDE:
\begin{eqnarray*}
&& d\xsmall(t)+{\rA}\xsmall(t)+\theta_n(\|\xsmall+\zbig\|_{\Upsilon_t(\rV)})\rB(\xsmall(t)+\zbig(t))\,dt=f(t)\,dt,\nonumber\\
&& \xsmall(0)=u_0.
\end{eqnarray*}
By Lemma \ref{lem PDE 1} and Corollary \ref{cor PDE 2},
this PDE has a unique solution $\xsmall\in \Upsilon_T(\rV)$.
One can show that a process $u$ defined by $u=\xsmall+\zbig$ is a solution to (\ref{eq1 SNS jieduan 2}).
For the uniqueness, we refer to Lemma \ref{lem PDE 2}.

\vskip 0.3cm

Now we prove that $u \in \Lambda_T(\rV)$.
Applying {the It\^o} formula,
\begin{eqnarray}\label{eq1 Pn 01}
&& \hspace{-1.4truecm}\|u(t)\|^2_{\rV} + 2\int_0^t\|u(s)\|^2_{\mathcal{D}({\rA})}\,ds\nonumber\\
&=&
  \|u_0\|^2_{\rV} - 2\int_0^t\Big\langle\theta_n(\|u\|_{\Upsilon_s(\rV)})\rB(u(s)),{\rA}u(s)\Big\rangle_{\rH}\,ds\nonumber\\
  &&+
  2\int_0^t\langle f(s),{\rA}u(s)\rangle_{\rH}\,ds
  +
  2 \int_0^t\int_{{\rZ}}\Big\langle G(y(s-),z),u(s)\Big\rangle_{\rV}\widetilde{\eta}(dz,ds)\nonumber\\
  &&+
  \int_0^t\int_{{\rZ}}\|G(y(s-),z)\|^2_{\rV}\eta(dz,ds)\nonumber\\
&=& \sum_{i=1}^5 J_i(t).
\end{eqnarray}
By Lemma \ref{lem B baisc prop}, we have
\begin{eqnarray}\label{eq1 J2 01 }
   |J_2(t)|
 &\leq&
   \frac{1}{2}\int_0^t\|u(s)\|^2_{\mathcal{D}({\rA})}\,ds + 2\int_0^t|\theta_n(\|u\|_{\Upsilon_s(\rV)})\rB(u(s))|^2_{\rH}\,ds \nonumber\\
 &\leq&
   \frac{1}{2}\int_0^t\|u(s)\|^2_{\mathcal{D}({\rA})}\,ds
   +
   C\int_0^t \theta^2_n(\|u\|_{\Upsilon_s(\rV)}) |u(s)|_{\rH}\|u(s)\|_{\rV}^2\|u(s)\|_{\mathcal{D}({\rA})}\,ds\nonumber\\
 &\leq&
   \frac{1}{2}\int_0^t\|u(s)\|^2_{\mathcal{D}({\rA})}\,ds
   +
   Cn^4t^{\frac{1}{2}}.
\end{eqnarray}

It is easy to see
\begin{eqnarray}\label{eq1 J3}
  J_3(t)
\leq
  \frac{1}{2}\int_0^t\|u(s)\|^2_{\mathcal{D}({\rA})}\,ds
   +
  2\int_0^t|f(s)|^2_{\rH}\,ds,\;\; t\in [0,T],
\end{eqnarray}
and
\begin{eqnarray}\label{eq1 J5}
  \mathbb{E}(J_5(T))
\leq
  CT\Big(1+\mathbb{E}\Big(\sup_{s\in[0,T]}\|y(s)\|^2_{\rV}\Big)\Big).
\end{eqnarray}
By the Burkholder-Davis-Gundy inequality, we get
\begin{eqnarray}\label{eq1 J4}
  \mathbb{E}\Big(\sup_{t\in[0,T]}|J_4(t)|\Big)
\leq
  \frac{1}{2}\mathbb{E}\Big(\sup_{t\in[0,T]}\|u(t)\|^2_{\rV}\Big)
  +
  CT\Big(1+\mathbb{E}\Big(\sup_{s\in[0,T]}\|y(s)\|^2_{\rV}\Big)\Big).
\end{eqnarray}
Combining (\ref{eq1 Pn 01})--(\ref{eq1 J4}), we have
\begin{eqnarray*}
  \mathbb{E}\Big(\sup_{t\in[0,T]}\|u(t)\|^2_{\rV}\Big) + \mathbb{E}\Big(\int_0^T\|u(s)\|^2_{\mathcal{D}({\rA})}\,ds\Big)
\leq
  C_{n,T}\Big(\|u_0\|^2_{\rV}+1+\mathbb{E}\Big(\sup_{s\in[0,T]}\|y(s)\|^2_{\rV}\Big)\Big).
\end{eqnarray*}
We have obtained that $u \in \Lambda_T(\rV)$, and this implies that
$$
\Phi^{\cdot}: \Lambda_T(\rV)\to\Lambda_T(\rV)
$$
is well-defined.

\vskip 0.3cm
By Lemma \ref{lem PDE 2}, there exist a positive constant $C$ and $L_{n}(T)\to 1$ as $T\to 0$ such that
\begin{eqnarray}
\|\Phi^{y_1}-\Phi^{y_2}\|_{\Lambda_T(\rV)}^2\leq CL^2_{n}(T)T\|y_1-y_2\|^2_{\Lambda_T(\rV)},\ \ \ \forall y_1, y_2\in \Lambda_T(\rV).
\end{eqnarray}

 Using arguments similar to the proof of Theorem \ref{thm-initial data in H}, we can construct a
 unique element $u_n\in\Lambda_T(\rV)$ for any $T>0$ such that $u_n$ is a solution of (\ref{eq1 SNS jieduan 1}). However, we do not know whether
 $u_n$ is the unique solution of (\ref{eq1 SNS jieduan 1}).

{Define} a stopping time $\tau_n$ by  \begin{equation}\label{eqn-tau_n}
                                              \tau_n=\inf\{t\geq0:\,\|u_n\|_{\Upsilon_t(\rV)}> n\}.
                                            \end{equation}
By the definition of $\theta_n$, we have  $\theta_n(\|u_n\|_{\Upsilon_t(\rV)})=1$
for any $t\in[0,\tau_n)$, hence $\{u_n(t),\, t\in[0,\tau_n)\}$ is a local solution of Problem (\ref{eq V 00}).

\vskip 0.3cm

\noindent
\textbf{Step 2: Local uniqueness.} We need a proof of the uniqueness not relying on the uniqueness from Theorem \ref{thm-initial data in H}; see Remark \ref{rem-uniqueness in V}.

\vskip 0.2cm
Assume that
$(U_1(t)$, $t\in[0,\sigma_1))$
and
$(U_2(t)$, $t\in[0,\sigma_2))$
are two local solutions of (\ref{eq V 00}).
{Choose and fix $R>0$.}
Define
\begin{align*}
\sigma^i_{R}&=\inf\{t>0,\ \|U_i\|_{\Upsilon_t(\rV)}> R\}\wedge\sigma_i,\ \ i=1,2.
\\
\sigma&=\sigma_1\wedge\sigma_2,\ \ \ \ \sigma_{R}=\sigma^1_{R}\wedge\sigma^2_{R}.
\end{align*}
It is known that $\sigma_i,\ \sigma_R^i,\ i=1,2$, $\sigma$, and $\sigma_R$ are stopping times.

Now we prove that
\begin{eqnarray}\label{eq1 unique 00}
U_1=U_2,\ \ \text{on } [0,\sigma).
\end{eqnarray}
\vskip 0.3cm

Let $M(t)=U_1(t)-U_2(t)$. By {the It\^o} formula,
\begin{eqnarray}\label{eq1 unique 01}
&&  \hspace{-1.4truecm}\|M(t)\|^2_{\rV} + 2\int_0^t\|M(s)\|^2_{\mathcal{D}({\rA})}\,ds\nonumber\\
&=&
  -2\int_0^t\langle \rB(U_1(s))-\rB(U_2(s)), {\rA}M(s)\rangle_{\rH}\,ds\nonumber\\
  &&+
  2\int_0^t\int_{{\rZ}}\langle G(U_1(s-),z)-G(U_2(s-),z),M(s-)\rangle_{\rV}\widetilde{\eta}(dz,ds)\nonumber\\
  &&+
  \int_0^t\int_{{\rZ}}\|G(U_1(s-),z)-G(U_2(s-),z)\|^2_{\rV}\eta(dz,ds)\nonumber\\
&=&
  I_1(t)+I_2(t)+I_3(t).
\end{eqnarray}
By Lemma \ref{lem B baisc prop},
\begin{eqnarray}\label{eq1 I1}
  |I_1(t)|
&\leq&
  \frac{1}{2}\int_0^t\|M(s)\|^2_{\mathcal{D}({\rA})}\,ds + 2\int_0^t\|\rB(U_1(s))-\rB(U_2(s))\|^2_{\rH}\,ds\nonumber\\
&\leq&
  \frac{1}{2}\int_0^t\|M(s)\|^2_{\mathcal{D}({\rA})}\,ds + C\int_0^t|U_1(s)|_{\rH}\|U_1(s)\|_{\rV}\|M(s)\|_{\rV}\|M(s)\|_{\mathcal{D}({\rA})}\,ds\nonumber\\
  &&+
  C\int_0^t\|U_2(s)\|_{\rV}\|U_2(s)\|_{\mathcal{D}({\rA})}\|M(s)\|_{\rV}\|M(s)\|_{\rH}\,ds\\
&\leq&
  \int_0^t\|M(s)\|^2_{\mathcal{D}({\rA})}\,ds
  +
  C\int_0^t\|M(s)\|_{\rV}^2\Big[
              |U_1(s)|_{\rH}^2\|U_1(s)\|_{\rV}^2+\|U_2(s)\|_{\mathcal{D}({\rA})}\|U_2(s)\|_{\rV}
             \Big]ds.\nonumber
\end{eqnarray}
In view of inequality \eqref{eq1 I1}, by Gronwall's lemma applied to equality (\ref{eq1 unique 01}), we infer that for all $t\in[0,T]$,
\begin{eqnarray}\label{eq1 unique 03}
&&  \hspace{-1.2truecm}\|M(t\wedge\sigma_R)\|^2_{\rV} + \int_0^{t\wedge\sigma_R}\|M(s)\|^2_{\mathcal{D}({\rA})}\,ds\nonumber\\
&\leq&
   e^{C\int_0^{t\wedge\sigma_R}
              |U_1(s)|_{\rH}^2\|U_1(s)\|_{\rV}^2+\|U_2(s)\|_{\mathcal{D}({\rA})}\|U_2(s)\|_{\rV}
             ds}
   \Big[
     \sup_{s\in[0,T]}|I_2(s\wedge\sigma_R)| + I_3(T\wedge\sigma_R)
   \Big]\nonumber\\
&\leq&
  e^{CTR^4+CR^2T^{\frac{1}{2}}}
  \Big[
     \sup_{s\in[0,T]}|I_2(s\wedge\sigma_R)| + I_3(T\wedge\sigma_R)
   \Big].
\end{eqnarray}
Next, by the Burkholder-Davis-Gundy inequality and Assumption \ref{con1 G} we infer that that for any $\delta>0$,
\begin{eqnarray}\label{eq1 expec I2}
  \mathbb{E}\Big(\sup_{s\in[0,T]}|I_2(s\wedge\sigma_R)|\Big)
\leq
  \delta \mathbb{E}\Big(\sup_{s\in[0,T]}\|M(s\wedge\sigma_R)\|^2_{\rV}\Big)
  +
  C_\delta \mathbb{E}\Big(\int_0^{T\wedge\sigma_R}\|M(s)\|^2_{\rV}\,ds\Big)
\end{eqnarray}
and
\begin{eqnarray}\label{eq1 expec I3}
  \mathbb{E}\Big(I_3(T\wedge\sigma_R)\Big)
\leq
  C\mathbb{E}\Big(\int_0^{T\wedge\sigma_R}\|M(s)\|^2_{\rV}\,ds\Big).
\end{eqnarray}

Combining inequalities (\ref{eq1 unique 03}), (\ref{eq1 expec I2}), and (\ref{eq1 expec I3}), and by applying Gronwall's lemma, we deduce that
\begin{eqnarray*}
\mathbb{E}\Big(\sup_{t\in[0,T]}\|M(t\wedge\sigma_R)\|^2_{\rV}\Big) + \mathbb{E}\Big(\int_0^{T\wedge\sigma_R}\|M(s)\|^2_{\mathcal{D}({\rA})}\,ds\Big)
=
0.
\end{eqnarray*}
Since $\lim_{R\toup \infty}\sigma_R=\sigma$, by taking first the limit as $R\toup \infty$ and then the limit as $T\toup \infty$,
we infer that
\begin{eqnarray*}
\mathbb{E}\Big(\sup_{t\in[0,\sigma)}\|M(t)\|^2_{\rV}\Big) + \mathbb{E}\Big(\int_0^{\sigma}\|M(s)\|^2_{\mathcal{D}({\rA})}\,ds\Big)
=
0,
\end{eqnarray*}
what implies the uniqueness of the local solution.

\vskip 0.3cm
\noindent
\textbf{Step 3. Global existence.} {Let us recall that $\tau_n$
has been defined in \eqref{eqn-tau_n}. By  \textbf{Step 2} we infer  that the sequence $(\tau_n)_{n=1}^\infty$} is nondecreasing and
$$
u_{n+1}(t)=u_n(t),\ \ \ t\in[0,\tau_n),\ \mathbb{P}\text{-a.s..}
$$

{Put }
$\tau_{max}:=\lim_{n\to\infty}\tau_n$. {By \cite[Proposition 2.1.2]{Ethier+Kurtz_1986}, the random time   $\tau_{\max}$ is a stopping time.}
As in the proof of Theorem \ref{thm-initial data in H}, we can define a process $\{u(t),\,t\in[0,\tau_{max})\}$
$$
u(t)=u_n(t),\ \ \ t\in[0,\tau_n),
$$
This process is a local solution of (\ref{eq V 00}), and it satisfies, {see \eqref{eqn-tau_max},}
$$
\lim_{t \toup \tau_{max}}\|u\|_{\Upsilon_t(\rV)}=\infty\text{ on }\{\omega:\,\tau_{max}<\infty\}\ \mathbb{P}\text{-a.s..}
$$
\deljz{From Zhai:}\addjzok{Using an argument similar to the one used in the proof of Theorem \ref{thm-initial data in V local}, we can prove that
$$
\mathbb{P}(\tau_{max}=\infty)=1.
$$
This concludes the proof of the global existence, and hence
 Theorem \ref{thm-initial data in V} is thus established.}
\end{proof}

\section{Large deviation principle (LDP) }\label{sec-LDP}

Fix $T>0$. In this section, we establish a Freidlin-Wentzell LDP for Problem (\ref{eq SNS 01}) on space $\Upsilon_T(\rV)$ defined in
\eqref{eqn-Upsilon_T-V}, i.e.,
$$
\Upsilon_T(\rV):=D([0,T],\rV)\cap L^2([0,T],\mathcal{D}({\rA})).$$
In the following, the space $D([0,T],\rV)$ is equipped with the Skorohod topology.

\subsection{Description of the problem and the statement of the main result}
We first introduce the problem and then state the precise assumptions on the coefficients, followed by the main result.
\vskip 0.2cm

Let us recall that ${\rZ}$ is a locally compact Polish space.  We set
\begin{eqnarray*}{\rZ}_T=[0,T]\times{\rZ}, \;\; \rY=\rZ \times [0,\infty)
 \mbox { and } \rY_T&=&[0,T]\times{\rZ}\times[0,\infty).
\end{eqnarray*}
 We set $\MT=\mathcal{M} ({\rZ}_T)$ be the space of all {nonnegative} measures $\vartheta$ on $({\rZ}_T,\mathcal{B}({\rZ}_T))$ such that
$\vartheta(K)<\infty$ for every compact subset $K$ of ${\rZ}_T$.

We endow the set $\MT$ with the weakest topology, denoted by $\mathcal{T}(\MT)$, such that for every $g\in C_c({\rZ}_T)$
(where by $C_c({\rZ}_T)$, we denote the space of real continuous functions on ${\rZ}_T$ with compact support), the map \[M_T \ni \vartheta\mapsto \int_{{\rZ}_T}g(z,s)\vartheta(dz,ds)\in \mathbb{R} \] is continuous.
\vskip 0.2cm
Analogously, we define $\MTb=\mathcal{M} (\rY_T)$ and $\mathcal{T}(\MTb)$.

It is known, see Section 1  of \cite{Budhiraja-Dupuis-Maroulas},
that both $\bigl(\MT, \mathcal{T}(\MT)\bigr)$
and $\bigl(\MTb, \mathcal{T}(\MTb)\bigr)$
are Polish spaces.

\vskip 0.2cm
In the present paper, we denote
$$
\bar{\Omega}=\MTb,\; \G:=\mathcal{T}(\MTb).
$$
\vskip 0.2cm
Fix a $\sigma$-finite measure $\nu$ on $(\rZ,\mathcal{B}(\rZ))$  such that
$\nu(K)<\infty$ for every compact subset $K$ of ${\rZ}$. By \cite{Ikeda-Watanabe} Section I.8,
there exists a unique probability measure $\Q$ on {$(\bar{\Omega},\G)$} on which the canonical/identity map 
\[\Nbar: \bar{\Omega} \ni \mbar \mapsto \mbar \in \MTb \]
 is a Poisson random measure (PRM) on $\rY_T$ with intensity measure $\Leb(dt)\otimes \nu(dz)\otimes\Leb(dr)$, over the probability space $(\bar{\Omega},\G,\Q)$.
\vskip 0.2cm
We also introduce the following notation: 
\begin{eqnarray*} \G_{t} &=& \mbox{ the } \Q\text{-}\mbox{completion of }\sigma\{\Nbar((0,s]\times {O}): s\in[0,t],\ {O}\in\mathcal{B}(\dela{\mathbb{Y}} {\rY})\}, \;\; t\in [0,T],\\
 \Gb&=&(\G_t)_{t\in[0,T]},\\
\Pc &=& \mbox{ the }\Gb\text{-}\mbox{predictable }\sigma\text{-}\mbox{field
on }[0,T]\times \bar{\Omega},\\
\bar{\mathbb{A}}&=&
\mbox{the class of all $(\Pc\otimes\mathcal{B}({\rZ}))$-measurable functions
$\varphi:{\rZ}_T\times\bar{\Omega}\to[0,\infty)$.}\label{page-Abar}
\end{eqnarray*}

It can be shown, that $\Nbar$ is a time-homogenous PRM on
 $\rY$, with intensity measure $\Leb(dt)\otimes \nu(dz)\otimes\Leb(dr)$, over the (filtered) probability space $(\bar{\Omega},\G,\Gb,\Q)$; see Appendix \ref{sec-A-PRM}. {The corresponding compensated PRM is denoted by $\widetilde{\Nbar}$. }

 For every function $\varphi\in\bar{\mathbb{A}}$, let us define a
counting process $\mathrm{N}^\varphi$ on ${\rZ}$ by
  \begin{eqnarray}\label{Jump-representation}
   \mathrm{N}^\varphi((0,t]\times {O})&:=&\int_{(0,t]\times {O} \times (0,\infty)}\mathds{1}_{[0,\varphi(s,z)]}(r)\Nbar(ds,dz,dr),
   \\
&=&\int_{(0,t]\times {O} \times (0,\infty)}\mathds{1}_{\{(s,z,r): r \leq \varphi(s,z)\}}(s,z,r) \Nbar(ds,dz,dr)
\nonumber \\
&=&\int_{(0,t]\times {O} \times (0,\infty)}\mathds{1}_{[r,\infty)}(\varphi(s,z)) \Nbar(ds,dz,dr),\;\; t\in[0,T],{O}\in\mathcal{B}({\rZ}).
\nonumber
  \end{eqnarray}
\addaok{We observe that
\[
\mathrm{N}^\varphi: \bar{\Omega} \to \mathcal{M}(\rZ_T)=\MT.
\]
}

Analogously, we define a process $\widetilde{\mathrm{N}}^\varphi$, \addaok{ i.e.,
  \begin{eqnarray}\label{Jump-representation-compensated}
   \widetilde{\mathrm{N}}^\varphi((0,t]\times {O})&:=&\int_{(0,t]\times {O} \times (0,\infty)}\mathds{1}_{[0,\varphi(s,z)]}(r)\widetilde{\Nbar}(ds,dz,dr),
   \\
&=&\int_{(0,t]\times {O} \times (0,\infty)}\mathds{1}_{\{(s,z,r): r \leq \varphi(s,z)\}}(s,z,r) \widetilde{\Nbar}(ds,dz,dr)
\nonumber \\
&=&\int_{(0,t]\times {O} \times (0,\infty)}\mathds{1}_{[r,\infty)}(\varphi(s,z)) \widetilde{\Nbar}(ds,dz,dr),\;\; t\in[0,T],{O}\in\mathcal{B}({\rZ}).
\nonumber
  \end{eqnarray}
For any Borel function $g: \rZ_T \to [0,\infty)$,

  \begin{eqnarray}\label{Jump-representaton-integral}
 \int_{(0,t] \times \rZ}  g(s,z) \, \widetilde{\mathrm{N}}^\varphi(ds,dz) &=&
 \int_{(0,t]\times {\rZ} \times (0,\infty)}\mathds{1}_{[0,\varphi(s,z)]}(r) g(s,z) \widetilde{\Nbar}(ds,dz,dr).
\end{eqnarray}

Note that if $\varphi$ is a constant function $a$ with value $a \in [0,\infty)$, then
  \begin{eqnarray*}
  \mathrm{N}^a((0,t]\times {O})&=& \Nbar \bigl((0,t] \times O\times (0,a] \bigr), \;\; t\in[0,T],{O}\in\mathcal{B}({\rZ}),\\
  \widetilde{\mathrm{N}}^a((0,t]\times {O})&=&\widetilde{\Nbar}\bigl((0,t] \times O\times (0,a] \bigr),\;\; t\in[0,T],{O}\in\mathcal{B}({\rZ}).
  \end{eqnarray*}
We finish this introduction with the following two simple observations.
\begin{proposition}\label{prop-PRM-N^a} In the above framework, for every $a>0$,
 the map
\begin{eqnarray}\label{eqnN^a}
  \mathrm{N}^a: \bar{\Omega} \to \mathcal{M}(\rZ_T)=\MT
 \end{eqnarray}
 is a Poisson random measure on $\rZ_T$ with intensity measure $\Leb(dt)\otimes a\nu(dz)$ and $\widetilde{\mathrm{N}}^a$ is equal to the corresponding compensated Poisson random measure.
 \end{proposition}
\begin{proposition}\label{prop-PRM-b} In the above framework, suppose that
two functions $\varphi, \psi\in\bar{\mathbb{A}}$,
a number $t_0\in[0,T]$,
and
a Borel set $O \subset \rZ$
are such that
\[ \varphi(s,z,\omega)=\psi(s,z,\omega), \mbox{ for } (s,z,\omega) \in [0,t_0]\times O \times \bar{\Omega}.\]
Then
\begin{eqnarray}\label{eqn-PRM-b}
  \mathrm{N}^\varphi((0,t]\times C)=  \mathrm{N}^\psi((0,t]\times C), \mbox{ for } t\in [0,t_0], C\in O\cap \mathcal{B}(\rZ).
 \end{eqnarray}
\end{proposition}
}
\vskip 0.1cm
Let us fix ${\eps}>0$, $u_0\in \rV$ and $f\in L^2([0,T];\rH)$.
Consider the following SPDE on the given probability space $(\bar{\Omega},\G,\Gb,\Q)$
\begin{equation}\label{eq LDP u epsilon 00}
\begin{split}
 du^{\eps}(t)&+{\rA}u^{\eps}(t)\,dt + \rB(u^{\eps}(t))\,dt
 =
  f(t)\,dt + {\eps}\int_{{\rZ}}G(u^{\eps}(t-),z)\widetilde{\mathrm{N}}^{1/{\eps}}(dz,dt),\ \ t\in[0,T],\\
 u^{\eps}(0)&=u_0.
 \end{split}
\end{equation}
By Theorem \ref{thm-initial data in V local} we infer that there exists a unique solution $u^{\eps}$ to Problem (\ref{eq LDP u epsilon 00}) whose trajectories a.s. belong to the space ${\UT}$
(see \eqref{eqn-Upsilon_T-V}).
In particular, $u^\eps$ induces an $\UT$-valued random variable. In this section, we aim to establish an LDP
for the laws of family $\{u^{\eps}\}_{{\eps}>0}$ on $\UT$.

\vskip 0.2cm

For our main result, we use the following notation.

Denote, for $N>0$,
  \begin{eqnarray}\label{S_N}
   S^N&=&\Big\{g:{\rZ}_T\to[0,\infty): g \mbox{ is Borel measurable and }\,L_T(g)\leq N\Big\},
   \\
\mathbb{S}&=&\cup_{N\geq 1}S^N, \nonumber
  \end{eqnarray}
  where for a  Borel measurable function $g:{\rZ}_T\to[0,\infty)$,
 we set
\begin{equation}\label{L_T}
L_T(g):=\int_0^T\int_{\rZ}\Big(g(t,z)\log g(t,z)-g(t,z)+1\Big)\nu(dz)\,dt.
\end{equation}
A function $g\in S^N$ can be identified with a measure $\nu^g\in\MT$, defined by
  \begin{eqnarray*}
   &&\nu^g({O})=\int_{O} g(t,z)\nu(dz)\,dt,\ \ {O}\in\mathcal{B}({\rZ}_T).
  \end{eqnarray*}
This identification induces a topology on $S^N$, under which $S^N$ is a compact space
(see \cite[Appendix]{Budhiraja-Chen-Dupuis}).
Throughout this section, we use this topology on $S^N$.

\vskip 0.1cm

Let us finally define
\begin{eqnarray*}
\mathcal{H}&:=&\Big\{h:{\rZ}\to\mathbb{R}: h \mbox{ is Borel measurable and there exists } \delta>0: \\
&&\hspace{3truecm}\lefteqn{ \int_{\Gamma}e^{\delta h^2(z)}\nu(dz)<\infty
 \mbox{ for all } \Gamma\in\mathcal{B}({\rZ}): \nu(\Gamma)<\infty\ \Big\},}
\end{eqnarray*}
and
\begin{eqnarray*}
L^2(\nu):=\{h:{\rZ}\to[0,\infty): h \mbox{ is Borel measurable and } \int_\rZ h^2(z)\nu(dz)<\infty\}.
\end{eqnarray*}

In many parts of this section, we use the following assumption.

\begin{con}\label{con LDP} \addjzok{There exist functions $L_\hbar,\,L_i\in \mathcal{H}\cap L^2(\nu)$, for $\hbar>0$, $i=2,3$, such that}
\begin{itemize}

\item[(LDP-01)]\addjzok{(Lipschitz on balls) for every $\hbar>0$, $v_1,\,v_2\in\rV$ with $\|v_1\|_\rV\vee\|v_2\|_\rV\leq \hbar$,
\begin{eqnarray*}
 \|G(v_1,z)-G(v_2,z)\|_{\rV}\leq L_\hbar(z)\|v_1-v_2\|_{\rV},\ z\in{\rZ},
\end{eqnarray*}}
\item[(LDP-02)] (Linear growth in $\rV$)
\begin{eqnarray*}
 \|G(v,z)\|_{\rV}\leq L_2(z)(1+\|v\|_{\rV}),\ v\in \rV,\ z\in{\rZ},
\end{eqnarray*}

\item[(LDP-03)] (Linear growth in $\rH$)
\begin{eqnarray*}
 |G(v,z)|_{\rH}\leq L_3(z)(1+|v|_{\rH}),\ v\in {\rV},\ z\in{\rZ}.
\end{eqnarray*}
\end{itemize}
\end{con}

\noindent
\begin{remark}\label{rem-con LDP}
A word of warning is due here. Quite often, the Lipschitz property is formulated differently. See, for instance, inequality \eqref{eqn-Lipschitz-H} in Assumption \ref{con G}.
\end{remark}

\begin{remark}
\dela{\addjz{Because $L_\hbar,L_2 \in L^2(\nu)$, the first two parts of \textbf{Assumption} \ref{con LDP} imply \textbf{Assumption} \ref{con G Local in V}.}}
Because the functions $L_\hbar$, $L_2$, {and $L_3$} belong to $L^2(\nu)$, the \dela{first two parts of} \textbf{Assumption} \ref{con LDP} {implies} \textbf{Assumption} \ref{con G Local in V}.
\end{remark}

We now state the main result of this section. We use  convention that  $\inf(\emptyset)=\infty$.
\begin{theorem}\label{thm-LDP}
Assume that \textbf{Assumption \ref{con LDP}} holds, $f\in L^2([0,T];\rH)$, and $u_0\in \rV$.
Then the family $\{u^{\eps}\}_{{\eps}>0}$
 satisfies an LDP on $\UT$ with the good rate function $I$ defined by
 \begin{eqnarray}\label{eq3 rate function}
  I(k):=\inf\Big\{L_T(g): g\in\mathbb{S}, u^g=k \Big\},\ \ \ k\in\UT,
 \end{eqnarray}
where for $g\in\mathbb{S}$, $u^g$ is the unique solution of the following deterministic PDE
\begin{equation}\label{eq3 LDP deter 00}
 \begin{split}
 \frac{du^g(t)}{dt}&+{\rA}u^g(t) + \rB(u^g(t)) = f(t) + \int_{{\rZ}}G(u^g(t),z)(g(t,z)-1)\nu(dz),\\
 u^g(0)&=u_0.
 \end{split}
\end{equation}
\end{theorem}
\vskip 0.3cm
\begin{remark}

By Theorem \ref{thm-initial data in H local}, and using similar arguments as in the proof of  Theorem \ref{thm-LDP}, it is not difficult to improve
the results on Freidlin-Wentzell-type LDP for strong solutions in the probabilistic sense
of 2D SNSEs driven by L\'evy processes of jump type. We only state the result here, and omit the proof.
\begin{con}\label{con LDP Prob} \addjzok{There exist functions $\Upsilon_\hbar\in \mathcal{H}\cap L^2(\nu)$ for $\hbar>0$ and $\Upsilon\in \mathcal{H}\cap L^2(\nu)$ such that}
\begin{itemize}

\item[(LDP-01-P)]\addjzok{(Lipschitz on balls in $\rH$) for every $\hbar>0$, $v_1,\,v_2\in\rH$ with $\|v_1\|_\rH\vee\|v_2\|_\rH\leq \hbar$,
\begin{eqnarray*}
 \|G(v_1,z)-G(v_2,z)\|_{\rH}\leq \Upsilon_\hbar(z)\|v_1-v_2\|_{\rH},\ z\in{\rZ},
\end{eqnarray*}}
\item[(LDP-02-P)] (Linear growth in $\rH$)
\begin{eqnarray*}
 \|G(v,z)\|_{\rH}\leq \Upsilon(z)(1+\|v\|_{\rH}),\ v\in \rH,\ z\in{\rZ}.
\end{eqnarray*}
\end{itemize}
\end{con}
\begin{theorem}
Assume that \textbf{Assumption \ref{con LDP Prob}} holds, $f\in L^2([0,T];\rV')$, and $u_0\in \rH$.
Then the family $\{u^{\eps}\}_{{\eps}>0}$
 satisfies an LDP on $D([0,T],\rH)\cap L^2([0,T],\rV)$ with the good rate function $J$ defined by
 \begin{eqnarray*}
  J(k):=\inf\Big\{L_T(g): g\in\mathbb{S}, u^g=k \Big\},\ \ \ k\in D([0,T],\rH)\cap L^2([0,T],\rV),
 \end{eqnarray*}
where for $g\in\mathbb{S}$, $u^g$ is the unique solution of the following deterministic PDE
\begin{equation*}
 \begin{split}
 \frac{du^g(t)}{dt}&+{\rA}u^g(t) + \rB(u^g(t)) = f(t) + \int_{{\rZ}}G(u^g(t),z)(g(t,z)-1)\nu(dz),\\
 u^g(0)&=u_0.
 \end{split}
\end{equation*}

\end{theorem}

Let us point out that the results (see \cite{Dong-Xiong-Zhai-Zhang, Xiong Zhai, Zhai-Zhang}) on this topic assume that the global Lipschitz condition in $\rH$ with Condition (LDP-02-P) holds,  i.e.,
\begin{itemize}
\item \addjzok{(Global Lipschitz in $\rH$) There exist functions $\widetilde{\Upsilon}\in \mathcal{H}\cap L^2(\nu)$ such that, for every $v_1,\,v_2\in\rH$,
\begin{eqnarray*}
 \|G(v_1,z)-G(v_2,z)\|_{\rH}\leq \widetilde{\Upsilon}(z)\|v_1-v_2\|_{\rH},\ z\in{\rZ}.
\end{eqnarray*}}
\end{itemize}

\end{remark}

Before we can embark on the proof of Theorem \ref{thm-LDP}, we need to establish the well-posedness of equation (\ref{eq3 LDP deter 00}). This is a consequence of the following result, whose proof is postponed to Appendix \ref{sec-B}.

\begin{lemma}\label{lem3 LDP deter 1}
Assume that $N\in\mathbb{N}$. Then, for all $u_0\in \rV$, $f\in L^2([0,T],\rH)$, and $g\in S^N$, there exists a unique  solution $u^g \in C([0,T],\rV)\cap L^2([0,T],\mathcal{D}({\rA}))$ of Problem (\ref{eq3 LDP deter 00}). Moreover, for any $\rho>0$ and $R>0$,
there exists a positive constant $C_N=C_{N,\rho,R}$ such that for every $g \in S^N$ and all $u_0\in \rV$ and $f\in L^2([0,T],\rH)$, such that
$\Vert u_0\Vert_{\rV} \leq \rho $ and $\vert f \vert_{L^2([0,T],\rH)} \leq R$,
the following estimate is satisfied
\begin{eqnarray}\label{eq3 LDP deter pro}
             \sup_{t\in[0,T]}\|u^g(t)\|^2_{\rV}+\int_0^T\|u^g(t)\|^2_{\mathcal{D}({\rA})}\,dt\leq C_N.
\end{eqnarray}
\end{lemma}

\vskip 0.3cm

\begin{proof}[{\bf Proof of Theorem \ref{thm-LDP}}]

By applying \dela{\addjz{Theorem \ref{thm-initial data in V local}}} Theorem \ref{thm-initial data in V local} to Problem (\ref{eq LDP u epsilon 00}), in view of  \cite[Theorem 8]{HY Zhao}, we infer that there exists
a family of $\{\mathcal{G}^{\eps}\}_{{\eps}>0}$, where
\[ \mathcal{G}^{\eps}: \MT \to \UT \mbox{ is a measurable map}\]
such that for every $\eps>0$, the following condition holds.
\begin{itemize}
 \item[] If $\eta$ is a  {time-homogenous} Poisson random measure {on $\rZ$ with intensity ${\eps}^{-1}\nu(dz)$}, i.e.,  {a Poisson random measure on $\rZ_T$} with intensity $\Leb(dt) \otimes{\eps}^{-1}\nu(dz)$, on a stochastic basis $(\Omega^1,\mathcal{F}^1,\mathbb{P}^1,\mathbb{F}^1)$ with $\mathbb{F}^1 =\{\mathcal{F}^1_t,t\in[0,T]\}$ satisfying the usual conditions, then the process $Y^{\eps}$ defined by
$$
Y^{\eps}:=\mathcal{G}^{\eps}({\eps}\eta)
$$
is the unique solution of
\begin{equation}
\begin{split}
 dY^{\eps}(t)&+{\rA}Y^{\eps}(t)\,dt + \rB(Y^{\eps}(t))\,dt
\\ &=
  f(t)\,dt
  + {\eps}\int_{{\rZ}}G(Y^{\eps}(t-),z)(\eta(dz,dt)-{\eps}^{-1}\nu(dz)\,dt),\\
 Y^{\eps}(0)&=u_0\dela{\in \rV}.
 \end{split}
\end{equation}
\end{itemize}
The statements in the condition means, that $Y^{\eps}$ induces (in a natural way) an $\mathbb{F}^1$-{progressively measurable} process {(for which we do not introduce a separate notation)} which satisfies
\begin{enumerate}
 \item[(a1)] the trajectories of $Y^{\eps}$ belong to $\UT$ $\mathbb{P}^1$-a.s.,
 \item[(a2)] the following equality holds \text{ in } $\rH$: for all $t\in[0,T]$, $\mathbb{P}^1$-a.s.:
  \begin{equation}
  \begin{split}
   Y^{\eps}(t)
   &=
   u_0-\int_0^t{\rA}Y^{\eps}(s)\,ds - \int_0^t\rB(Y^{\eps}(s))\,ds\\
    &\ \ +
    \int_0^tf(s)\,ds + {\eps}\int_0^t\int_{{\rZ}}G(Y^{\eps}(s-),z)(\eta(dz,ds)-{\eps}^{-1}\nu(dz)\,ds).
    \end{split}
\end{equation}
\end{enumerate}

Therefore, since by Proposition \ref{prop-PRM-N^a}, $N^{{\eps}^{-1}}$ is a Poisson random measure on $\rZ_T$ with intensity measure $\Leb(dt)\otimes \eps^{-1}\nu(dz)$, we deduce the following result which will be used later on.

\begin{corollary}\label{cor-solution to eq LDP u epsilon 00}
In the above framework, the unique solution of Problem (\ref{eq LDP u epsilon 00})
on the probability space $(\bar{\Omega},\G,\Gb,\Q)$ is given by the following equality:
\begin{eqnarray}\label{eqn-u^eps}
u^{\eps}:=\mathcal{G}^{\eps}({\eps} N^{{\eps}^{-1}}).
\end{eqnarray}
\end{corollary}
\vskip 0.3cm
Moreover, Lemma \ref{lem3 LDP deter 1} implies that, for every $g\in\mathbb{S}$, there is a unique solution $u^g\in \UT$ of equation (\ref{eq3 LDP deter 00}).
This allows us to define a map
\begin{eqnarray}\label{eq3 define G0}
\mathcal{G}^0:\, \mathbb{S}\ni g\mapsto u^g\in\UT.
\end{eqnarray}
\vskip 0.3cm

We apply Theorem 2.4 of \cite{Budhiraja-Chen-Dupuis} to finish the proof of Theorem \ref{thm-LDP}. According to \cite{Budhiraja-Chen-Dupuis}, it is sufficient to verify two claims. The first one is the following.

\noindent
\textbf{Claim-LDP-1} For all $ N\in\mathbb{N}$, if $g_n,\ g\in S^N$ are such that $g_n\to g$ as
 $n\to\infty$, then
   $$
     \mathcal{G}^0({g_n})\to \mathcal{G}^0(
     {g})\text{ \emph{i.e.,} } u^{g_n}\to u^g\quad\text{in}\quad \UT.
   $$

To state the second claim, we introduce additional notations.

Let us fix an increasing sequence $\{K_n\}_{
n=1,2,\cdots}$
 of compact subsets of ${\rZ}$ such that
\begin{equation}\label{eqn-K_n}
 \cup _{n=1}^\infty K_n={\rZ}.
 \end{equation}
 Define
  \begin{equation}
  \label{eqn-A_b}
\begin{split}
    \bar{\mathbb{A}}_{b}
     &=
       \bigcup_{n=1}^\infty\Big\{\varphi\in\bar{\mathbb{A}}:
                  \varphi(t,z,\omega)\in[\frac{1}{n},n],
                  \mbox{ if }\ (t,z,\omega)\in [0,T]\times K_n\times \bar{\Omega}\ \\
                 &\ \ \ \ \ \ \ \ \ \ \ \ \ \ \ \
                   \mbox{ and }\ \varphi(t,z,\omega)=1,\ \mbox{ if }\ (t,z,\omega)\in [0,T]\times K_n^c\times \bar{\Omega}
       \Big\},
       \end{split}
  \end{equation}
where the class $\bar{\mathbb{A}}$ was introduced on page \pageref{page-Abar}. We also define the following notation:
\begin{equation}\label{eqn-U^N}
\begin{split}
\mathcal{U}^N&:=\{\varphi\in \bar{\mathbb{A}}_{b}:\, \varphi(\cdot,\cdot,\omega)\in S^N, \mbox{ for } \Q\mbox{-a.a. } \omega \in \bar{\Omega}\},\\
\mathcal{U}&= \bigcup_{N=1}^\infty \mathcal{U}^N.
\end{split}
\end{equation}

\noindent
\textbf{Claim-LDP-2.} For all $N\in\mathbb{N}$, if ${{\eps}_n}\to 0$ and $\varphi_{{\eps}_n},\ \varphi\in\mathcal{U}^N$
are such that $\varphi_{{\eps}_n}$ converges in law  to $\varphi$, then
   $$
     \mathcal{G}^{{\eps}_n}\Big({{\eps}_n}
     N^{{{\eps}_n}^{-1}\varphi_{{\eps}_n}}\Big)\text{ converges in law to }
     \mathcal{G}^0({\varphi})\text{ in } \UT.
   $$

\vskip 0.3cm

The verification of \textbf{Claim-LDP-1} will be given in Proposition \ref{prop-1st continuity lemma} in the following subsection.
\textbf{Claim-LDP-2} will be established in Proposition \ref{prop3 02} in the final section.
Assuming these claims have been proven,
the proof of Theorem \ref{thm-LDP} is complete.

\end{proof}

\vskip 0.3cm

\subsection{The first continuity lemma}
\label{subsec-1st continuity lemma}

To verify \textbf{Claim-LDP-1},
it is sufficient to prove the following result.
\begin{proposition}[The first continuity lemma]\label{prop-1st continuity lemma}
For all $ N\in\mathbb{N}$, let $g_n,\ g\in S^N$ be such that $g_n\to g$ in $S^N$ as
 $n\to\infty$. Then
   $$
     \mathcal{G}^0({g_n})\to \mathcal{G}^0(
     {g})\quad\text{in}\quad \UT.
   $$
\end{proposition}

\begin{proof}[Proof of Proposition \ref{prop-1st continuity lemma}]
Recall the definition of $\mathcal{G}^0$ in (\ref{eq3 define G0}). Set $u^{g_n}$ be the solution of (\ref{eq3 LDP deter 00}) with $g$
replaced by $g_n$. For simplicity, set $u_n=u^{g_n}=\mathcal{G}^0({g_n})$ and $u=u^g=\mathcal{G}^0({g})$. To prove
our result, we prove that
$$
u_n\to u\ \ \text{in }\UT.
$$
\vskip 0.3cm

Fix $\alpha\in(0,\frac{1}{2})$.
Let $W^{\alpha,2}([0,T],\rV^\prime)$ be the Sobolev space consisting of all $h\in L^2([0,T],\rV^\prime)$ satisfying
$$
\int_0^T\int_0^T\frac{\|h(t)-h(s)\|^2_{\rV^\prime}}{|t-s|^{1+2\alpha}}\,dtds<\infty,
$$
endowed with the norm
\begin{equation}\label{eqn-W^alpha2-norm}
\|h\|^2_{W^{\alpha,2}([0,T],\rV^\prime)}=\int_0^T\|h(t)\|^2_{\rV^\prime}\,dt + \int_0^T\int_0^T\frac{\|h(t)-h(s)\|^2_{\rV^\prime}}{|t-s|^{1+2\alpha}}\,dtds.
\end{equation}

 By Lemma \ref{lem3 LDP deter 1} and using arguments similar to the proof of (4.8) in \cite{Zhai-Zhang}, we can deduce that
\begin{eqnarray}\label{eq3 prop 1 W 00}
\sup_{n\geq 1}\|u_n\|^2_{W^{\alpha,2}([0,T],\rV^\prime)}
\leq
\widetilde{C}_{N}<\infty.
\end{eqnarray}
Moreover, since by  \cite[Theorem 2.1]{Flandoli+Gatarek_1995} (see also \cite{Temam_2001}), the embedding
\begin{equation}
\label{eqn-comapct embedding}
L^2([0,T],\mathcal{D}({\rA}))\cap W^{\alpha,2}([0,T],\rV^\prime)\hookrightarrow L^2([0,T],\rV)
\end{equation}
is compact.
By Lemma \ref{lem3 LDP deter 1} and (\ref{eq3 prop 1 W 00}),
we infer that there exists $\tilde{u}\in L^2([0,T],\mathcal{D}({\rA}))\cap L^\infty([0,T],\rV)$ and a subsequence  (for simplicity, we also denote it by $u_n$)
such that
\begin{itemize}
\item[(P1)] $u_n\to \tilde{u}$ weakly in $L^2([0,T],\mathcal{D}({\rA}))$,

\item[(P2)] $u_n\to \tilde{u}$ in the weak-$^\ast$ topology of $L^\infty([0,T],\rV)$, \dela{\addjz{}}
{and
$$
\sup_{n\geq 1}\sup_{t\in[0,T]}\|u_n(t)\|_\rV+\sup_{t\in[0,T]}\|\tilde{u}(t)\|_\rV=:\hbar_0<\infty,
$$}

\item[(P3)] $u_n\to \tilde{u}$ strongly in $L^2([0,T],\rV)$.

\end{itemize}
\vskip 0.3cm

Now we prove that the limit function $\tilde{u}$ is a solution of equation (\ref{eq3 LDP deter 00}).
By the uniqueness of this solution, we infer $\tilde{u}=u=u^g$. The proof seems to be
classical, but it is not, because of the nonstandard terms.
\vskip 0.2cm

Let $\psi$ be a continuously differentiable {$\rV$-valued} function on $[0,T]$ with $\psi(T)=0$.
We multiply  $u_n(t)$ scalarly in $\rH$ by $\psij(t)$, and then
integrate by parts. This leads to the following equation:
\begin{eqnarray}\label{eq3 prop1 Xn lim 00}
&&\hspace{-1.4truecm}-\int_0^T\langle u_n(t),\psij^\prime(t)\rangle_{\rH}\,dt + \int_0^T\langle u_n(t),\psij(t)\rangle_{\rV}\,dt\nonumber\\
&=&
 \langle u_0,\psij(0)\rangle_{\rH}-\int_0^T\dual{\langle \rB(u_n(t)),\psij(t)\rangle}
 \,dt
 +
 \int_0^T\langle f(t),\psij(t)\rangle_{\rH}\,dt\nonumber\\
 &&+
 \int_0^T\int_{{\rZ}}\langle G(u_n(t),z),\psij(t)\rangle_{\rH}(g_n(t,z)-1)\nu(dz)\,dt.
\end{eqnarray}
Keeping in mind properties (P1), (P2), and (P3) from above and using a argument similar to the proof in \cite[Theorem III.3.1]{Temam_2001}, we see that
\begin{eqnarray}\label{eq3 prop1 Xn lim 01}
&&\lim_{n\to\infty}\Big[-\int_0^T\langle u_n(t),\psij^\prime(t)\rangle_{\rH}\,dt + \int_0^T\langle u_n(t),\psij(t)\rangle_{\rV}\,dt\nonumber\\
&&\ \ \ \ \ \ \ \ -
 \langle u_0,\psij(0)\rangle_{\rH}+\int_0^T\dual{\langle \rB(u_n(t)),\psij(t)\rangle}\,dt
 -
 \int_0^T\langle f(t),\psij(t)\rangle_{\rH}\,dt\Big]\nonumber\\
 &=&
 -\int_0^T\langle \tilde{u}(t),\psij^\prime(t)\rangle_{\rH}\,dt + \int_0^T\langle \tilde{u}(t),\psij(t)\rangle_{\rV}\,dt\nonumber\\
&&-
 \langle u_0,\psij(0)\rangle_{\rH}+\int_0^T\dual{\langle \rB(\tilde{u}(t)),\psij(t)\rangle}\,dt
 -
 \int_0^T\langle f(t),\psij(t)\rangle_{\rH}\,dt.
 \end{eqnarray}
What concerns us is the last term in (\ref{eq3 prop1 Xn lim 00}).
Since  $g_n\to g$ in $S^N$, by Lemma 3.11 in \cite{Budhiraja-Chen-Dupuis}, we infer that
\begin{eqnarray}\label{eq3 prop1 Xn lim last 01}
&&\hspace{-1.4truecm}\lim_{n\to\infty} \int_0^T\int_{{\rZ}}\langle G(\tilde{u}(t),z),\psij(t)\rangle_{\rH}(g_n(t,z)-1)\nu(dz)\,dt\nonumber\\
&=&
\int_0^T\int_{{\rZ}}\langle G(\tilde{u}(t),z),\psij(t)\rangle_{\rH}(g(t,z)-1)\nu(dz)\,dt.
\end{eqnarray}
\vskip 0.2cm

Next, for $\delta>0$, we set
$$
A_{n,\delta}:=\{t\in[0,T],\ \|u_n(t)-\tilde{u}(t)\|_{\rV}\geq \delta\}.
$$
Since $u_n\to \tilde{u}$ strongly in $L^2([0,T],\rV)$, {by applying the Chebyshev inequality we infer that}
\begin{eqnarray}\label{eq3 An delta}
\lim_{n\to \infty}\Leb_{[0,T]}(A_{n,\delta})
\leq
\lim_{n\to \infty}\frac{ {1}}{\delta {^2}}\int_0^T\|u_n(t)-\tilde{u}(t)\|^2_{\rV}\,dt
=
0,
\end{eqnarray}
where $\Leb_{[0,T]}$ is the Lebesgue measure on $[0,T]$.

{Fix $\delta>0$.
Then, \addjzok{by Assumption \ref{con LDP} and assertion (P2)}, we infer that
\begin{eqnarray}\label{eq 20180624-00}
&&\hspace{-2truecm}\Big|\int_0^T\int_{{\rZ}}\langle G(u_n(t),z)-G(\tilde{u}(t),z),\psij(t)\rangle_{\rH}(g_n(t,z)-1)\nu(dz)\,dt\Big|\nonumber\\
&\leq&
     \vert \psi\vert_{L^\infty([0,T];\rV)} \int_0^T\|u_n(t)-\tilde{u}(t)\|_{\dela{\rV} {\rV}}\int_{{\rZ}}\addjzok{L_{\hbar_0}(z)}|g_n(t,z)-1|\nu(dz)\,dt\nonumber\\
&\leq&
  2\hbar_0 \vert \psi\vert_{L^\infty([0,T];\rV)}  \int_{A_{n,\delta}}\int_{{\rZ}}\addjzok{L_{\hbar_0}(z)}|g_n(t,z)-1|\nu(dz)\,dt\nonumber\\
   &&+
\delta\vert \psi\vert_{L^\infty([0,T];\rV)} \int_0^T\int_{{\rZ}}\addjzok{L_{\hbar_0}(z)}|g_n(t,z)-1|\nu(dz)\,dt,
\end{eqnarray}
\addjzok{where ${\hbar_0}$ is the positive constant appearing in (P2).}}

In what follows, we use the following result;
see  \cite[(3.3) of Lemma 3.1]{Zhai-Zhang},
\cite[ Remark 2]{Yang-Zhai-Zhang},
or \cite[(3.5) of Lemma 3.4]{Budhiraja-Chen-Dupuis}.
\addjzok{\begin{center}
For every function $\Im\in \mathcal{H}\cap L^2(\nu)$ and  every ${\eps}>0$ there exists $\beta>0$ such that for every $O\in\mathcal{B}([0,T])$ with $\Leb_{[0,T]}({O})\leq \beta $, the following inequality holds
\begin{eqnarray}\label{eq3 small {A}}
   \sup_{h\in S^N} \int_{O}\int_{{\rZ}}\Im(z)|h(s,z)-1|\nu(dz)\,ds\leq {\eps}.
\end{eqnarray}
\end{center}}

Hence, by (\ref{eq3 An delta})--(\ref{eq3 small {A}}) and {(\ref{eq3 star local})},
we have
 {\begin{eqnarray}\label{eq3 prop1 Xn lim last 02}
&& \hspace{-2truecm}\limsup_{n\to\infty}
 \Big|\int_0^T\int_{{\rZ}}\langle G(u_n(t),z)-G(\tilde{u}(t),z),\psij(t)\rangle_{\rH}(g_n(t,z)-1)\nu(dz)\,dt\Big|
 \\
 &\leq& \delta \addjzok{C_{\hbar_0,N}} \vert \psi\vert_{L^\infty([0,T];\rV)}.\nonumber
\end{eqnarray}}
Since $\delta>0$ can be chosen arbitrarily small, this implies
\begin{eqnarray}\label{eq3 prop1 Xn lim last 03}
\lim_{n\to\infty}\Big|\int_0^T\int_{{\rZ}}\langle G(u_n(t),z)-G(\tilde{u}(t),z),\psij(t)\rangle_{\rH}(g_n(t,z)-1)\nu(dz)\,dt\Big|
=
0.
\end{eqnarray}

We observe that the above proof of (\ref{eq3 prop1 Xn lim last 02}) yields the following stronger result, which we use later on. \addjzok{For every $\Im\in \mathcal{H}\cap L^2(\nu)$,
\begin{eqnarray}\label{eq3 star00}
\lim_{n\to\infty}\sup_{k\in S^N}\int_0^T\|u_n(s)-\tilde{u}(s)\|_{\rV}\int_{{\rZ}}\Im(z)|k(s,z)-1|\nu(dz)\,ds=0.
\end{eqnarray}}

Combining (\ref{eq3 prop1 Xn lim 01}), (\ref{eq3 prop1 Xn lim last 01}), and (\ref{eq3 prop1 Xn lim last 03}), we arrive at the following:
\begin{eqnarray}\label{eq3 prop1 X lim 00}
&&\hspace{-1.4truecm}\lefteqn{-\int_0^T\langle \tilde{u}(t),\psij^\prime(t)\rangle_{\rH}\,dt + \int_0^T\langle \tilde{u}(t),\psij(t)\rangle_{\rV}\,dt}\nonumber\\
&=&
 \langle u_0,\psij(0)\rangle_{\rH}-\int_0^T\dual{\langle \rB(\tilde{u}(t)),\psij(t)\rangle}\,dt
 +
 \int_0^T\langle f(t),\psij(t)\rangle_{\rH}\,dt\nonumber\\
 &&+
 \int_0^T\int_{{\rZ}}\langle G(\tilde{u}(t),z),\psij(t)\rangle_{\rH}(g(t,z)-1)\nu(dz)\,dt.
\end{eqnarray}
From here, following an argument as in the proof of  Theorems 3.1 and 3.2 from  \cite[Sect. 3, Chapter III]{Temam_2001},
we can conclude that $\tilde{u}$ is a solution of equation (\ref{eq3 LDP deter 00}) as claimed, and then by the uniqueness of this solution, $\tilde{u}=u=u^g.$

\vskip 0.1cm
In the final stage of our proof of Proposition \ref{prop-1st continuity lemma}, we prove that
\[u_n\to u \mbox{ in } C([0,T],\rV)\cap L^2([0,T],\mathcal{D}({\rA})).\]

For this purpose, let $\Zbig_n=u_n-u$. Then by \cite[Lemma III.1.2]{Temam_2001} (in the space $\rV$) we get, for all $t\in [0,T]$,
\begin{eqnarray}\label{eq3 Zn 00}
&&\hspace{-1.3truecm}\lefteqn{\|\Zbig_n(t)\|^2_{\rV} + 2\int_0^t\|\Zbig_n(s)\|^2_{\mathcal{D}({\rA})}\,ds}\\
&=&
 -2\int_0^t\langle \rB(u_n(s))-\rB(u(s)),{\rA}\Zbig_n(s)\rangle_{\rH}\,ds\nonumber\\
 &&+
 2\int_{0}^t\int_{{\rZ}}\Big\langle
                 G(u_n(s),z)(g_n(s,z)-1)-G(u(s),z)(g(s,z)-1),\Zbig_n(s)
                \Big\rangle_{\rV}
 \nu(dz)\,ds\nonumber\\
&\leq&
  \frac{1}{2}\int_0^t\|\Zbig_n(s)\|^2_{\mathcal{D}({\rA})}\,ds + 2\int_0^t|\rB(u_n(s)-\rB(u(s))|^2_{\rH}\,ds\nonumber\\
  &&+
 {\addjzok{2\int_0^t\|\Zbig_n(s)\|^2_{\rV}\int_{{\rZ}}L_{\hbar_0}(z)|g_n(s,z)-1|\nu(dz)}\,ds}\nonumber\\
  &&+
  2\int_0^t\|\Zbig_n(s)\|_{\rV}(1+\|u(s)\|_{\rV})\int_{{\rZ}}L_2(z)(|g_n(s,z)-1|+|g(s,z)-1|)\nu(dz)\,ds.
  \nonumber
\end{eqnarray}

By Lemma \ref{lem B baisc prop}, for all $s\in [0,T]$,
\begin{eqnarray}\label{eq3 prop1 B}
 &&\hspace{-1truecm}\lefteqn{ |\rB(u_n(s))-\rB(u(s))|^2_{\rH}}\nonumber\\
&\leq&
   2\Big(|\rB(u_n(s),\Zbig_n(s))|^2_{\rH}+|\rB(\Zbig_n(s),u(s)|^2_{\rH}\Big)\nonumber\\
&\leq&
  C\Big(
   |u_n(s)|_{\rH}\|u_n(s)\|_{\rV}\|\Zbig_n(s)\|_{\rV}\|\Zbig_n(s)\|_{\mathcal{D}({\rA})} + |\Zbig_n(s)|_{\rH}\|\Zbig_n(s)\|_{\rV}\|u(s)\|_{\rV}\|u(s)\|_{\mathcal{D}({\rA})}
  \Big)\nonumber\\
&\leq&
  \frac{1}{4}\|\Zbig_n(s)\|_{\mathcal{D}({\rA})}^2 + C \|\Zbig_n(s)\|^2_{\rV}\Big(\|u_n(s)\|^4_{\rV}+\|u(s)\|_{\rV}\|u(s)\|_{\mathcal{D}({\rA})}\Big).
\end{eqnarray}
Substituting (\ref{eq3 prop1 B}) into (\ref{eq3 Zn 00}), and by Lemma \ref{lem3 LDP deter 1}, since $u \in L^\infty([0,T];\rV)$, we obtain
\begin{eqnarray}\label{eq3 Zn 01}
&&\hspace{-1truecm}\lefteqn{ \|\Zbig_n(t)\|^2_{\rV} + \int_0^t\|\Zbig_n(s)\|^2_{\mathcal{D}({\rA})}\,ds}
\nonumber\\
&\leq&
 \int_0^t\|\Zbig_n(s)\|^2_{\rV} \Big(
              C\|u_n(s)\|^4_{\rV}+C\|u(s)\|_{\rV}\|u(s)\|_{\mathcal{D}({\rA})}+ {\addjzok{2\int_{{\rZ}}L_{\hbar_0}(z)|g_n(s,z)-1|\nu(dz)}}
             \Big)
   ds\nonumber\\
   &&+
   C\sup_{h\in S^N}\int_0^T\|\Zbig_n(s)\|_{\rV}\int_{{\rZ}}L_2(z)|h(s,z)-1|\nu(dz)\,dt,\;\;\; t\in [0,T] .
\end{eqnarray}

By Gronwall's lemma,
Lemma \ref{lem3 LDP deter 1} and (\ref{eq3 star}) imply that
\begin{eqnarray*}
&&\hspace{-1truecm}\lefteqn{ \sup_{t\in[0,T]}\|\Zbig_n(t)\|^2_{\rV} + \int_0^T\|\Zbig_n(t)\|^2_{\mathcal{D}({\rA})}\,dt}\\
&&
\leq e^{C\int_0^T\Big(
              \|u_n(s)\|^4_{\rV}+\|u(s)\|_{\rV}\|u(s)\|_{\mathcal{D}({\rA})}+ {\addjzok{\int_{{\rZ}}L_{\hbar_0}(z)|g_n(s,z)-1|\nu(dz)}}
             \Big)
        ds}\\
        &&\ \ \ \ \ \cdot\sup_{h\in S^N}\int_0^T\|\Zbig_n(s)\|_{\rV}\int_{{\rZ}}L_2(z)|h(s,z)-1|\nu(dz)\,dt\\
&&\leq
C_{N,T}\sup_{h\in S^N}\int_0^T\|\Zbig_n(s)\|_{\rV}\int_{{\rZ}}L_2(z)|h(s,z)-1|\nu(dz)\,dt,\;\; t\in [0,T].
\end{eqnarray*}
Note that the integral $\int_0^T\bigl(\|u_n(s)\|^4_{\rV}+\|u(s)\|_{\rV}\|u(s)\|_{\mathcal{D}({\rA})} \bigr)\, ds$ is finite in view of Lemma \ref{lem3 LDP deter 1}.

Therefore, by (\ref{eq3 star00})
\begin{eqnarray}\label{eq LDP claim 1}
\lim_{n\to\infty}\sup_{t\in[0,T]}\|\Zbig_n(t)\|^2_{\rV} + \int_0^T\|\Zbig_n(t)\|^2_{\mathcal{D}({\rA})}\,dt=0.
\end{eqnarray}
\vskip 0.2cm

The proof of Proposition \ref{prop-1st continuity lemma} is thus complete.

\end{proof}



\section{A generalization of the Girsanov Theorem}\label{sec-Girsanov}

The aim of this section is to establish a certain generalization of the Girsanov Theorem. This result will then be used in Section \ref{sec-LDP-verification} to verify \textbf{Claim-LDP-2}. At the beginning of this section, we state and prove Lemma \ref{Lem Q4 01}. This is followed by the second part in which we prove Theorem \ref{thm-Girsanov}, which is the main result of this section.


Before formulating the main result of this section, i.e.,  Theorem \ref{thm-Girsanov}, we give Lemma \ref{Lem Q4 01} below.
In order to formulate it, let us recall
the sets $K_n$ that were introduced on page \pageref{eqn-K_n},
and let us introduce, for $n\in\mathbb{N}$, the following set:
\begin{eqnarray}\label{eqn-barA_bn}
\bar{\mathbb{A}}_{b,n}&=&\Big\{\varphi\in\bar{\mathbb{A}}:
                  \varphi(t,z,\omega)\in[\frac{1}{n},n]
                  \ \mbox{ if }\ (t,z,\omega)\in [0,T]\times K_n\times \bar{\Omega}\ \nonumber\\
                 &&\ \ \ \ \ \ \ \ \ \ \ \ \ \ \ \
                   \mbox{ and }\ \varphi(t,z,\omega)=1\ \mbox{ if } (t,z,\omega)\in [0,T]\times K_n^c\times \bar{\Omega}
       \Big\}.
\end{eqnarray}

Note that with the notation defined in \eqref{eqn-A_b},
we have the following equality:
 $$\bar{\mathbb{A}}_{b}=\bigcup_{n=1}^\infty \bar{\mathbb{A}}_{b,n}.$$

The proof of Lemma 2.4 in \cite{Budhiraja-Dupuis-Maroulas} implies the following result, but since that paper's authors did not give the details, we present a detailed proof.
This result is important in proving Theorem \ref{thm-Girsanov}.

\begin{lemma}\label{Lem Q4 01}
Assume that $n\in \mathbb{N}$ and that
 $\varphi\in \bar{\mathbb{A}}_{b,n}$.
Then, there exists an $\bar{\mathbb{A}}_{b,n}$-valued sequence $\bigl(\psi_m\bigr)_{m\in\mathbb{N}}$ such that the following properties are satisfied.
\begin{itemize}
\item[{\rm\textbf{(R1)}}] {For every $m$}, there exist $l\in\mathbb{N}$ and $n_1,\cdots,n_l\in\mathbb{N}$,  a partition $0 =t_0<t_1<\cdots<t_l=T$ and  families
 \begin{eqnarray*}
 &&\Xbig_{ij}, \;\; i=1,\cdots,l,\; j=1,\cdots,n_i,\\
  &&E_{ij}, \;\; i=1,\cdots,l,\; j=1,\cdots,n_i,
 \end{eqnarray*}
 such that
the $ \Xbig_{ij}$ are $[\frac1n,n]$-valued,  $\G_{t_{i-1}}$-measurable random
variables,
 and, for each $i=1,\cdots,l$, $\bigl(E_{ij}\bigr)_{j=1}^{n_i}$ is a measurable partition of the set $K_n$, such that
\begin{eqnarray}\label{eqn-psi_m}
\psi_m(t,z,\omega)&=&\mathds{1}_{\{0\}}(t)+\sum_{i=1}^l\sum_{j=1}^{n_i}\mathds{1}_{(t_{i-1},t_i]}(t) \Xbig_{ij}(\omega)\mathds{1}_{E_{ij}}(z)+\mathds{1}_{K_n^c}(z)\mathds{1}_{(0,T]}(t)
\\&& \mbox{ for all } (t,z,\omega) \in [0,T]\times \rZ\times \bar{\Omega}.
\nonumber
\end{eqnarray}
\item[{\rm \textbf{(R2)}}] $\lim_{m\to\infty}\int_0^T|\psi_m(t,z,\omega)-\varphi(t,z,\omega)|\,dt=0$, for $\nu\otimes\Q$-a.a. $(z,\omega)\in \rZ\times \bar{\Omega}$.
\end{itemize}
\end{lemma}

\begin{proof}[Proof of Lemma \ref{Lem Q4 01}]
Fix $n\in \mathbb{N}$ and
 $\varphi\in \bar{\mathbb{A}}_{b,n}$.

First, let us remark that, ``$\varphi_k$" on page 729 line -6, in the proof of \cite[Lemma 2.4]{Budhiraja-Dupuis-Maroulas}
 should read as follows
\begin{equation}\label{eqn-phi_k}
\varphi_k(t,z,\omega)=k(\frac{1}{k}-t)^+ +k\int_{(t-1/k)^+}^t\varphi(s,z,\omega)\,ds.
\end{equation}
One can check that
\begin{equation}\label{eqn-phi_k-2}
\varphi_k(t,z,\omega)=1, \text{ on }(t,z,\omega)\in [0,T]\times K_n^c\times\bar{\Omega}.
\end{equation}

In the proof of \cite[Lemma 2.4]{Budhiraja-Dupuis-Maroulas}, the authors  proved the following three assertions.
\begin{itemize}
 \item [(L1)] The process $\varphi_k$ defined in \eqref{eqn-phi_k} satisfies the following three properties:
   \begin{itemize}
    \item [(L1.1)] $\lim_{k\to\infty}\int_0^T|\varphi_k(t,z,\omega)-\varphi(t,z,\omega)|\,dt=0$, $\nu\otimes\Q$-a.s. $(z,\omega)\in \rZ\times \bar{\Omega}$,
    \item [(L1.2)] $\varphi_k\in\bar{\mathbb{A}}_{b,n}$,
    \item [(L1.3)] the function $[0,\infty) \ni t\mapsto \varphi_{k}(t,z,\omega)$ is continuous for $\nu\otimes\Q$-a.s. $(z,\omega)\in \rZ\times \bar{\Omega}$.
   \end{itemize}

 \item [(L2)] If, for $k,q\in\mathbb{N}$, we set
    $$
     \varphi_k^q(t,z,\omega)=\mathds{1}_{\{0\}}(t)+\sum_{m=0}^{\llcorner qT\lrcorner}\varphi_k(\frac{m}{q},z,\omega)\mathds{1}_{(m/q,(m+1)/q]}(t),
     \ \ (t,z,\omega)\in[0,T]\times{\rZ}\times\bar{\Omega}.
    $$
    then
    \begin{itemize}
      \item [(L2.1)] $\varphi_k^q\in\bar{\mathbb{A}}_{b,n}$,

      \item [(L2.2)]
        $$
        \lim_{q\to\infty}\int_0^T|\varphi_k^q(t,z,\omega)-\varphi_k(t,z,\omega)|\,dt=0,\ \nu\otimes\Q\text{-a.s.}\ (z,\omega)\in \rZ\times \bar{\Omega}.
        $$
    \end{itemize}

 \item [(L3)] For all $k,q\in\mathbb{N}$, there exists an $\bar{\mathbb{A}}_{b,n}$-valued sequence of processes $\bigl(\varphi_k^{q,r}\bigr)_{k=1}^\infty$ such that
       \begin{itemize}
        \item [(L3.1)] for every $k$, $\varphi_k^{q,r}$ satisfies condition {\rm\textbf{(R1)}},
        \item [(L3.2)]
          $
          \lim_{r\to\infty}\int_0^T|\varphi_k^{q,r}(t,z,\omega)-\varphi_k^q(t,z,\omega)|\,dt=0,\ \nu\otimes\Q\text{-a.s.}\ (z,\omega)\in \rZ\times \bar{\Omega}.
          $
\end{itemize}

\end{itemize}

Note that      claim (L2.2) follows easily from Claim (L1.3).

In order to prove assertion (L3), we repeat the following argument from the proof of \cite[Lemma 2.4]{Budhiraja-Dupuis-Maroulas}.

Note that for fixed $q$ and $m$, $g(z,\omega)=\varphi_k(\frac{m}{q},z,\omega)$ is a $\mathcal{B}({\rZ})\otimes\G_{m/q}$-measurable map with values in $[1/n,n]$ and $g(z,\omega)=1$ for $z\in K_n^c$.
By a standard approximation procedure,
 one can find $\mathcal{B}({\rZ})\otimes\G_{m/q}$-measurable maps $g_r,\ r\in\mathbb{N}$
with the following properties: $g_r(z,\omega)=\sum_{j=1}^{a(r)}c_j^r(\omega)\mathds{1}_{E_j^r}(z)$ for $z\in K_n$, where for each $r$, $\{E_j^r\}_{j=1}^{a(r)}$ is some measurable partition of $K_n$ and for all $j,r,$ $c_j^r(\omega)\in[1/n,n]$ a.s.; $g_r(z,\omega)=1$ for $z\in K_n^c$;
$g_r\to g$, as $r\to \infty$,  $\nu\otimes\Q$-a.s..

Having established the above, it is easy to see that assertion (L3.2) holds.

\vskip 0.2cm
Moreover, these three assertions imply our result in Lemma \ref{Lem Q4 01}. This can be seen as follows.

First of all, it is easy to see that, for any $k,q,r\in\mathbb{N}$
\begin{eqnarray}\label{eq 2018-06-26}
\varphi(t,z,\omega)=\varphi_k(t,z,\omega)=\varphi_k^q(t,z,\omega)=\varphi_k^{q,r}(t,z,\omega)=1,\ \ \text{for }(t,z,\omega)\in [0,T]\times K_n^c\times\bar{\Omega}.
\end{eqnarray}
Hence, we only need to consider the case of $z\in K_n$.

Set
$$
\Omega_1=\Big\{
      (z,\omega)\in K_n\times\bar{\Omega}:\ \lim_{k\to\infty}\int_0^T|\varphi_k(t,z,\omega)-\varphi(t,z,\omega)|\,dt=0
     \Big\}.
$$
Then assertion (L1.1) implies
\begin{eqnarray}\label{Lem Q4 eq step 1 00}
(\nu\otimes\Q)(K_n\times\bar{\Omega} \setminus\Omega_1)=0.
\end{eqnarray}

For simplicity, in the following we set ${O}^c=K_n\times\bar{\Omega}\setminus{O}$ for any ${O}\subset K_n\times\bar{\Omega}$, and keep in mind
that $\nu(K_n)<\infty$.

\vskip 0.2cm
Let
$$
I_m^i=\bigcap_{k\geq i} \Big\{
      (z,\omega)\in K_n\times\bar{\Omega}:\ \int_0^T|\varphi_k(t,z,\omega)-\varphi(t,z,\omega)|\,dt\leq 1/m
     \Big\}.
$$
It is easy to see that
$$
I_m^i\subset I_{m}^{i+1}, \;\; \bigcup_{i=1}^{\infty}I_m^i=:\Omega_{1,m} \mbox{ and } \Omega_{1}\subset\Omega_{1,m}.
$$
Considering (\ref{Lem Q4 eq step 1 00}) with this,  there exists $i_m$ such that
\begin{eqnarray}\label{Lem Q4 eq step 1 01}
(\nu\otimes\Q)((I_m^{i_m})^c)\leq \frac{1}{m^2}.
\end{eqnarray}
\vskip 0.2cm

Set
$$
\Omega_{2,k}=\Big\{
      (z,\omega)\in K_n\times\bar{\Omega}:\ \lim_{q\to\infty}\int_0^T|\varphi_k^q(t,z,\omega)-\varphi_k(t,z,\omega)|\,dt=0
     \Big\},
$$
and
$$
\Omega_{3,k,q}=\Big\{
      (z,\omega)\in K_n\times\bar{\Omega}:\ \lim_{r\to\infty}\int_0^T|\varphi_k^{q,r}(t,z,\omega)-\varphi_k^q(t,z,\omega)|\,dt=0
     \Big\}.
$$
Let
$$
II_m^{k,i}=\cap_{q\geq i} \Big\{
      (z,\omega)\in K_n\times\bar{\Omega}:\ \int_0^T|\varphi_k^q(t,z,\omega)-\varphi_k(t,z,\omega)|\,dt\leq \frac{1}{m}
     \Big\},
$$
and
$$
III_m^{k,q,i}=\cap_{r\geq i} \Big\{
      (z,\omega)\in K_n\times\bar{\Omega}:\ \int_0^T|\varphi_k^{q,r}(t,z,\omega)-\varphi_k^q(t,z,\omega)|\,dt\leq \frac{1}{m}
     \Big\}.
$$
Using arguments similar to (\ref{Lem Q4 eq step 1 01}), there exist $j_{m,k}$ and $l_{m,k,q}$ such that
\begin{eqnarray}\label{Lem Q4 eq step 1 02}
(\nu\otimes\Q)((II_m^{k,j_{m,k}})^c)\leq \frac{1}{m^2},
\end{eqnarray}
and
\begin{eqnarray}\label{Lem Q4 eq step 1 03}
(\nu\otimes\Q)((III_m^{k,q,l_{m,k,q}})^c)\leq \frac{1}{m^2}.
\end{eqnarray}

By the definition of $I_m^k,\ II_m^{k,q},\ III_m^{k,q,r}$,
\begin{eqnarray*}
 &&\hspace{-2truecm}\lefteqn{ I_m^k\cap II_m^{k,q}\cap III_m^{k,q,r}}\\
&\subset&
 \Big\{(z,\omega)\in K_n\times \bar{\Omega}:
  \int_0^T|\varphi_k(t,z,\omega)-\varphi(t,z,\omega)|\,dt\leq \frac{1}{m}
  \Big\}\\
  &&\cap
  \Big\{(z,\omega)\in K_n\times \bar{\Omega}:\int_0^T|\varphi_k^{q}(t,z,\omega)-\varphi_k(t,z,\omega)|\,dt\leq \frac{1}{m} \Big\}\\
   &&\cap
   \Big\{(z,\omega)\in K_n\times \bar{\Omega}: \int_0^T|\varphi_k^{q,r}(t,z,\omega)-\varphi_k^q(t,z,\omega)|\,dt\leq \frac{1}{m} \Big\},
\end{eqnarray*}
so, for any $(z,\omega)\in I_m^k\cap II_m^{k,q}\cap III_m^{k,q,r}$,
\begin{eqnarray}\label{Lem Q4 eq step 1 04}
 \int_0^T|\varphi_k^{q,r}(t,z,\omega)-\varphi(t,z,\omega)|\,dt\leq \frac{3}{m}.
\end{eqnarray}
\vskip 0.2cm

For
\begin{equation}\label{eqn-kqj}
\mbox{$k=i_m$, $q=j_{m,k}=j_{m,i_m}$, $r=l_{m,k,q}=l_{m,i_m,j_{m,i_m}}$},
\end{equation}
set \[\psi_m=\varphi_k^{q,r} \]
and define
$$
\Omega_0=\cup_{i=1}\cap_{m\geq i}
       \Big\{
          I_m^k\cap II_m^{k,q}\cap III_m^{k,q,r}\ 
       \Big\}.
$$
{In the following, we prove that $\psi_m$ satisfies {\rm\textbf{(R1)}} and {\rm\textbf{(R2)}}, while keeping in mind (\ref{eq 2018-06-26}).
{The claim for \rm\textbf{(R1)}} is obvious.}
For \rm\textbf{(R2)}, it is easy to see $\Omega_0\subset K_n\times\bar{\Omega}$,
thus, we only need to prove the following two results
\begin{eqnarray}\label{Lem Q4 eq step 1 05}
&&(\nu\otimes\Q)(\Omega_0^c)\,=\,0,\\
 && \lim_{m\to\infty}\int_0^T|\psi_m(t,z,\omega)-\varphi(t,z,\omega)|\,dt,\ \text{on } (z,\omega)\in\Omega_0.
 \label{Lem Q4 eq step 1 05-b}
\end{eqnarray}

We use convention \eqref{eqn-kqj} again to get
$$\cap_{m= i}^\infty
       \Big\{
          I_m^k\cap II_m^{k,q}\cap III_m^{k,q,r}\ 
       \Big\}\toup \Omega_0 \mbox{ as } k \to \infty,
$$
and so, by (\ref{Lem Q4 eq step 1 01})-(\ref{Lem Q4 eq step 1 03}), we infer that for every $i\in\mathbb{N}$,
\begin{eqnarray}\label{Lem Q4 eq step 1 06}
&&\hspace{-1truecm}\lefteqn{(\nu\otimes\Q)(\Omega_0^c)
\leq
  (\nu\otimes\Q)\Big(\Big(\cap_{m\geq i}
       \Big\{
          I_m^k\cap II_m^{k,q}\cap III_m^{k,q,r}\ 
       \Big\}\Big)^c\Big)}\nonumber\\
&\leq&
   \sum_{m=i}^\infty
   \Big(
     (\nu\otimes\Q)((I_m^{k})^c)
     +
     (\nu\otimes\Q)((II_m^{k,q})^c)
     +
     (\nu\otimes\Q)((III_m^{k,q,r})^c)
   \Big)\\
    &\leq&
  \sum_{m=i}^\infty\frac{3}{m^2}.\nonumber
\end{eqnarray}
Since $  \sum_{m=i}^\infty\frac{3}{m^2} \to 0$ as $i\to\infty$,
we conclude the proof of equality \eqref{Lem Q4 eq step 1 05}.

Next, we prove claim \eqref{Lem Q4 eq step 1 05-b}.

For this, fix $(z,\omega)\in \Omega_0$. Then, using the convention \eqref{eqn-kqj} again, there exists $i_0 \in \mathbb{N}$ such that
$$
(z,\omega)\in \bigcap_{m=i_0}^\infty
       \Big\{
          I_m^k\cap II_m^{k,q}\cap III_m^{k,q,r}\ 
       \Big\}.
$$
Thus by (\ref{Lem Q4 eq step 1 04}), we infer that
\begin{eqnarray}\label{Lem Q4 eq step 1 07}
 \int_0^T|\psi_m(t,z,\omega)-\varphi(t,z,\omega)|\,dt\leq \frac{3}{m},\ \ \mbox{ for all } m\geq i_0.
\end{eqnarray}
Hence \eqref{Lem Q4 eq step 1 05-b} follows,
and therefore we have proved that {\rm\textbf{(R2)}} holds.
Consequently, the proof of Lemma \ref{Lem Q4 01} is complete.
\end{proof}

\vskip 0.3cm

Fix $\eps>0$ and $\varphi_{\eps}\in\bar{\mathbb{A}}_b$,
and set $\psi_{\eps}=1/{\varphi_{\eps}}$.
In view of the definition \eqref{eqn-A_b} of the set $\bar{\mathbb{A}}_b$ , it is easy
 to see that $\psi_{\eps}\in\bar{\mathbb{A}}_b$. In particular, there exists a natural number $n\in\mathbb{N}$ such that
 $$
 \psi_{\eps}(t,z,\omega)\in[\frac{1}{n},n]\ \mbox{ if }\ (t,z,\omega)\in[0,T]\times K_n\times\bar{\Omega},
 $$
where $K_n$ is a compact subset of ${\rZ}$ from \eqref{eqn-K_n} and
 $$
 \psi_{\eps}(t,z,\omega)=1\ \mbox{ if }\ (t,z,\omega)\in[0,T]\times K^c_n\times\bar{\Omega}.
 $$
 Combining the fact that $\nu(K_n)<\infty$,
 we infer the following four assertions grouped for convenience in Theorem \ref{thm-Girsanov}.

\begin{theorem}\label{thm-Girsanov}
In the framework introduced above, the following hold.
\begin{itemize}
\item[(S1)] The process $\mathcal{M}^{\eps}_t(\psi_{\eps})$, $t\in[0,T]$, defined by
\begin{eqnarray}
\label{eqn-M^eps}
&&\hspace{-3truecm}\lefteqn{ \mathcal{M}^{\eps}_t(\psi_{\eps})=\exp\Big(
                     \int_{(0,t]\times{\rZ}\times[0,{\eps}^{-1}\varphi_{\eps}(s,z)]}\log(\psi_{\eps}(s,z))\Nbar(ds,dz,dr)}
                     \\&&\ \ \ \ \ \ \ +
                     \int_{(0,t]\times{\rZ}\times[0,{\eps}^{-1}\varphi_{\eps}(s,z)]}\Big(-\psi_{\eps}(s,z)+1\Big)\nu(dz)\,dsdr
                       \Big)\nonumber\\
                  &&\hspace{-1.5truecm}=\exp\Big(
                     \int_{(0,t]\times K_n\times[0,{\eps}^{-1}\varphi_{\eps}(s,z)]}\log(\psi_{\eps}(s,z))\Nbar(dsdzdr)
                     \nonumber \\&&\ \ \ \ \ \ \ +
                     \int_{(0,t]\times K_n\times[0,{\eps}^{-1}\varphi_{\eps}(s,z)]}\Big(-\psi_{\eps}(s,z)+1\Big)\nu(dz)\,dsdr
\nonumber                       \Big), \;\; t\in [0,T],
\end{eqnarray}
is an $\Gb$-martingale {on $(\bar{\Omega},\G,\Gb,\Q)$}.

\item[(S2)] The formula
$$
\mathbb{P}^{\eps}_T({O})=\int_{O}\mathcal{M}^{\eps}_T(\psi_{\eps})\,d\Q,\ \ \forall {O}\in \G
$$
defines a probability measure on $(\bar{\Omega},\G)$.

\item[(S3)] The measures $\Q$ and $\mathbb{P}^{\eps}_T$ are equivalent.

\item[(S4)] The laws on {$\MT$} of  the following two random variables are equal: (i) ${\eps} N^{{\eps}^{-1}\varphi_{\eps}}$ defined on probability space $(\bar{\Omega},\G,\Gb,\mathbb{P}^{\eps}_T)$
and (ii) ${\eps} N^{{\eps}^{-1}}$ defined on probability space $(\bar{\Omega},\G,\Gb,\Q)$.
 \end{itemize}
\end{theorem}
\todozbdel{ZB to JZ: I added the below and above in blue. }
\addaok{Note: Recall that the two processes appearing in Assertion (S4) were introduced in equality \eqref{Jump-representation}.}

\vskip 0.2cm
{Although, Theorem \ref{thm-Girsanov} is a ``standard result", see for instance  \cite[Theorem III.3.24]{Jacod-Shiryaev} for the semimartingales and  \cite[Theorem 3.10.21]{BICHTELER} for the Poisson point processes; it seems hard to find an accessible reference in the literature which  would work under our conditions. Therefore, we give a detailed proof of this result in our situation.}

\begin{proof}[Proof of Theorem \ref{thm-Girsanov}]

Since assertion (S2) is implied by assertion (S1), we only prove assertions (S1), (S3), and (S4). We divide the proof into three steps.

\begin{proof}[\bf{Step 1}]
Assume that $\varphi_{\eps}$ is a step process
(see Lemma \ref{Lem Q4 01});
i.e.,
there exist $l,n_1,\cdots,n_l\in\mathbb{N}$ and
a partition \[0=t_0<t_1<\cdots<t_l=T,\]
$[\frac1n,n]$-valued random variables \[ \Xbig_{ij}, \;\; i=1,\cdots,l, \;\; j=1,\cdots,n_i,\]
such that $\Xbig_{ij}$ is $\G_{t_{i-1}}$-measurable,
and,
for each $i=1,\cdots,l$,
a disjoint measurable partition $\bigl(E_{ij}\bigr)_{j=1,\cdots,n_i}$ of the set $K_n$,
 such that for all $(t,z,\omega) \in [0,T]\times \rZ\times \bar{\Omega}$,
$$
\varphi_{\eps}(t,z,\omega)=\mathds{1}_{\{0\}}(t)+\sum_{i=1}^l\sum_{j=1}^{n_i}\mathds{ 1}_{(t_{i-1},t_i]}(t) \Xbig_{ij}(\omega)\mathds{1}_{E_{ij}}(z)+\mathds{1}_{K_n^c}(z)\mathds{1}_{(0,T]}(t).
$$

Then, for any ${O}\in\mathcal{B}({\rZ})$, we have
\begin{eqnarray*}
  &&\hspace{-3.1truecm}\lefteqn{ N^{{\eps}^{-1}\varphi_{\eps}}(T,{O})=N^{{\eps}^{-1}\varphi_{\eps}}(t_l,{O})}\\
&=&
 \int_0^{t_l}\int_{O}\int_0^\infty \mathds{1}_{(0,{\eps}^{-1}\varphi_{\eps}(s,z)]}(r) \Nbar(ds,dz,dr)\\
&=&
  N^{{\eps}^{-1}\varphi_{\eps}}(t_{l-1},{O})\\
  &&+
  \int_{t_{l-1}}^{t_l}\int_{O\cap {K_n^c}}\int_0^\infty \mathds{1}_{(0,{{\eps}^{-1}\varphi_{\eps}(s,z)}]}(r) \Nbar(ds,dz,dr)\\
  &&+
  \sum_{j=1}^{n_l}\int_{t_{l-1}}^{t_l}\int_{O\cap {E_{lj}}}\int_0^\infty \mathds{1}_{(0, {{\eps}^{-1}\varphi_{\eps}(s,z)]}}(r) \Nbar(ds,dz,dr)\\
&=&
 N^{{\eps}^{-1}\varphi_{\eps}}(t_{l-1},{O})\\
  &&+
  \int_{t_{l-1}}^{t_l}\int_{O\cap {K_n^c}}\int_0^\infty \mathds{1}_{(0, {{\eps}^{-1}}]}(r) \Nbar(ds,dz,dr)\\
  &&+
  \sum_{j=1}^{n_l}\int_{t_{l-1}}^{t_l}\int_{O\cap {E_{lj}}}\int_0^\infty \mathds{1}_{(0, {{\eps}^{-1} \Xbig_{lj}}]}(r) \Nbar(ds,dz,dr).
\end{eqnarray*}
{Moreover, with the process $\mathcal{M}^{\eps}_t(\psi_{\eps})$ defined in formula \eqref{eqn-M^eps}, we have}
\begin{eqnarray}\label{eq Porp M}
&&\hspace{-2.5truecm}\mathcal{M}^{\eps}_T(\psi_{\eps})
=\exp\Big(
                     \int_{(0,T]\times{\rZ}\times[0,{\eps}^{-1}\varphi_{\eps}(s,z)]}\log(\psi_{\eps}(s,z))\Nbar(ds,dz,dr)\nonumber\\
                     &&\hspace{1.2truecm}+
                     \int_{(0,T]\times{\rZ}\times[0,{\eps}^{-1}\varphi_{\eps}(s,z)]}\Big(-\psi_{\eps}(s,z)+1\Big)\nu(dz)\,dsdr
                       \Big)\nonumber\\
&&\hspace{-0.9truecm}=\exp\Big(
                     \int_{(0,T]\times K_n\times[0,{\eps}^{-1}\varphi_{\eps}(s,z)]}\log(\psi_{\eps}(s,z))\Nbar(ds,dz,dr)\nonumber\\
                     &&\hspace{1.2truecm}+
                     \int_{(0,T]\times K_n\times[0,{\eps}^{-1}\varphi_{\eps}(s,z)]}\Big(-\psi_{\eps}(s,z)+1\Big)\nu(dz)\,dsdr
                       \Big)\nonumber\\
&&\hspace{-0.9truecm}=
  \exp\Big(
                     \int_{(0,t_{l-1}]\times K_n\times[0,{\eps}^{-1}\varphi_{\eps}(s,z)]}\log(\psi_{\eps}(s,z))\Nbar(ds,dz,dr)\nonumber\\
                     &&\hspace{1.2truecm}+
                     \int_{(0,t_{l-1}]\times K_n\times[0,{\eps}^{-1}\varphi_{\eps}(s,z)]}\Big(-\psi_{\eps}(s,z)+1\Big)\nu(dz)\,dsdr
                       \Big)\nonumber\\
 &&\cdot
  \exp\Big( \sum_{j=1}^{n_l} {\Big[}
                   \int_{{(t_{l-1},t_l]}}\int_{E_{lj}}\int_{(0,{\eps}^{-1}\Xbig_{lj} ]}\log(\frac{1}{\Xbig_{lj} })\Nbar(ds,dz,dr)\nonumber\\
                   &&\hspace{3truecm}+
                   \int_{{(t_{l-1},t_l]}}\int_{E_{lj}}\int_{(0,{\eps}^{-1}\Xbig_{lj} ]}\Big(-\frac{1}{\Xbig_{lj} }+1\Big)\nu(dz)\,dsdr
               {\Big]}
    \Big)\nonumber\\
&&\hspace{-0.9truecm}=\mathcal{M}^{\eps}_{t_{l-1}}(\psi_{\eps})
  \cdot
  \exp\Big( \sum_{j=1}^{n_l} {\Big[}
                   \int_{{(t_{l-1},t_l]}}\int_{E_{lj}}\int_{(0,{\eps}^{-1}\Xbig_{lj} ]}\log(\frac{1}{\Xbig_{lj} })\Nbar(ds,dz,dr)\nonumber\\
                   &&\hspace{3truecm}+
                   \int_{{(t_{l-1},t_l]}}\int_{E_{lj}}\int_{(0,{\eps}^{-1}\Xbig_{lj} ]}\Big(-\frac{1}{\Xbig_{lj} }+1\Big)\nu(dz)\,dsdr
               {\Big]}
    \Big).
\end{eqnarray}

Hence,
{ for any $\xi\in \mathbb{R}$, we have}
\begin{eqnarray}\label{eq Q4 01}
  &&\hspace{-1truecm}\lefteqn{\mathbb{E}^{\Q}\Big(e^{\sqrt{-1} \xi N^{{\eps}^{-1}\varphi_{\eps}}(T,{O})}\cdot \mathcal{M}^{\eps}_T(\psi_{\eps})\Big)}\nonumber\\
&=&
  \mathbb{E}^{\Q}\Big[
      \mathbb{E}^{\Q}
          \Big(
           e^{\sqrt{-1} \xi N^{\frac{\varphi_{\eps}}{\eps}}(T,{O})}\cdot \mathcal{M}^{\eps}_T(\psi_{\eps})| \G_{t_{l-1}}
          \Big)
                \Big]\nonumber\\
&=&
 \mathbb{E}^{\Q}\Big[
         e^{\sqrt{-1} \xi N^{{\eps}^{-1}\varphi_{\eps}}(t_{l-1},{O})}
      \cdot
         \mathcal{M}^{\eps}_{t_{l-1}}(\psi_{\eps})
      \cdot
         {Y(\cdot;\xi,t_{l-1},t_l)}
                \Big],
\end{eqnarray}
where {$Y= Y(\cdot;\xi,t_{l-1},t_l)$ is defined by }
\begin{eqnarray*}
 Y(\cdot;\xi,t_l,t_{l-1})
 &\!\!\!:=\!\!\!&\mathbb{E}^{\Q}\Big(
    \exp {\Big(}\sqrt{-1} \xi\Big(
           \int_{t_{l-1}}^{t_l}\int_{O\cap {K_n^c}}\int_0^\infty \mathds{1}_{(0, {{\eps}^{-1}}]}(r) \Nbar(ds,dz,dr)\\
  &&\ \ \ \ \ \ \ \ \ \ \ \ \ \ +
  \sum_{j=1}^{n_l}\int_{(t_{l-1},t_l]}\int_{O\cap {E_{lj}}}\int_0^\infty \mathds{1}_{(0, {{\eps}^{-1} \Xbig_{lj} }]}(r) \Nbar(ds,dz,dr)\Big) {\Big)}\\
  &&
  \ \ \ \ \ \ \cdot \exp\Big( \sum_{j=1}^{n_l} {\Big[}
                   \int_{{(t_{l-1},t_l]}}\int_{E_{lj}}\int_{(0,{\eps}^{-1}\Xbig_{lj} ]}\log(\frac{1}{\Xbig_{lj} })\Nbar(ds,dz,dr)\\
                   &&\ \ \ \ \ \ \ \ \ \ \ \ \ \ \  +
                   \int_{{(t_{l-1},t_l]}}\int_{E_{lj}}\int_{(0,{\eps}^{-1}\Xbig_{lj} ]}\Big(-\frac{1}{\Xbig_{lj} }+1\Big)\nu(dz)\,dsdr
               {\Big]}
    \Big)
           | \G_{t_{l-1}}
          \Big). \dela{(\omega)}
 \end{eqnarray*}

By assumptions, each $\Xbig_{lj}$, $j=1,2,\cdots,n_l$ {is $ \G_{t_{l-1}}$-measurable,
so by the properties of the conditional expectation, we infer that,  $\Q$-a.s. }
\begin{eqnarray*}
Y(\omega,\xi,t_l,t_{l-1})=K(\omega,\xi, \Xbig_{l1}(\omega),\Xbig_{l2}(\omega),\cdots,\Xbig_{ln_l}(\omega),t_l,t_{l-1})\nonumber
\end{eqnarray*}
where a random variable
$K(\omega,\xi, a_1,a_2,\cdots,a_{n_l},t_l,t_{l-1})$ is defined by
\begin{eqnarray*}
 &&\hspace{-1truecm}\lefteqn{ K(\cdot,\xi, a_1,a_2,\cdots,a_{n_l},t_l,t_{l-1})}\\
 &:=&
 \mathbb{E}^{\Q}\Big(
    \exp {\Big(}\sqrt{-1} \xi\Big(
           \int_{t_{l-1}}^{t_l}\int_{O\cap {K_n^c}}\int_0^\infty \mathds{1}_{(0, {{\eps}^{-1}}]}(r) \Nbar(ds,dz,dr)\\
  &&\ \ \ \ \ \ \ \ \ \ \ \ \ \ +
  \sum_{j=1}^{n_l}\int_{(t_{l-1},t_l]}\int_{O\cap {E_{lj}}}\int_0^\infty \mathds{1}_{(0, {{\eps}^{-1} a_j }]}(r) \Nbar(ds,dz,dr)\Big) {\Big)}\\
  &&
  \ \ \ \ \ \ \cdot \exp\Big( \sum_{j=1}^{n_l} {\Big[}
                   \int_{{(t_{l-1},t_l]}}\int_{E_{lj}}\int_{(0,{\eps}^{-1}a_j ]}\log(\frac{1}{a_j })\Nbar(ds,dz,dr)\\
                   &&\ \ \ \ \ \ \ \ \ \ \ \ \ +
                   \int_{{(t_{l-1},t_l]}}\int_{E_{lj}}\int_{(0,{\eps}^{-1}a_j ]}\Big(-\frac{1}{a_j }+1\Big)\nu(dz)\,dsdr
               {\Big]}
    \Big)
           | \G_{t_{l-1}}
          \Big)\dela{(\omega)}.
\end{eqnarray*}
Note that for any positive constants $a_1,a_2,\cdots,a_{n_l}$, {we have the following identity}
\begin{eqnarray*}
 &&\hspace{-1truecm}\lefteqn{ K(\cdot,\xi, a_1,a_2,\cdots,a_{n_l},t_l,t_{l-1})}\\ &=&
 \exp\Big(\sum_{j=1}^{n_l}(t_{l}-t_{l-1})\nu(E_{lj}){\eps}^{-1}a_j\Big(-\frac{1}{a_j }+1\Big)\Big)\\
 &&\cdot
 \mathbb{E}^{\Q}\Big(
    \exp \Big(\sqrt{-1} \xi\Big(
           \int_{t_{l-1}}^{t_l}\int_{O\cap {K_n^c}}\int_0^\infty \mathds{1}_{(0,{\eps}^{-1}]}(r) \Nbar(ds,dz,dr)\Big)\Big)\Big)\\
 &&\cdot
 \mathbb{E}^{\Q}\Big(\exp\Big(\sum_{j=1}^{n_l}(\sqrt{-1} \xi+\log(\frac{1}{a_j}))
\int_{(t_{l-1},t_l]}\int_{O\cap {E_{lj}}}\int_0^\infty \mathds{1}_{(0,{\eps}^{-1} a_j]}(r) \Nbar(ds,dz,dr)
    \Big)\\
 &&\cdot
 \mathbb{E}^{\Q}\Big(\exp\Big(\sum_{j=1}^{n_l}\log(\frac{1}{a_j})
\int_{(t_{l-1},t_l]}\int_{O^c\cap {E_{lj}}}\int_0^\infty \mathds{1}_{(0,{\eps}^{-1} a_j]}(r) \Nbar(ds,dz,dr)
    \Big)
 \\
&=&
 \exp\Big((t_l-t_{l-1})\nu({O}){\eps}^{-1}[e^{\sqrt{-1}\xi}-1]\Big).
\end{eqnarray*}
{Summing up, we infer that for $\omega \in \Omega$, $\Q$-a.s., we have }
\begin{eqnarray}\label{eq Q4 star}
Y(\dela{\omega}\cdot,\xi,t_l,t_{l-1})
=\exp\Big((t_l-t_{l-1})\nu({O}){\eps}^{-1}[e^{\sqrt{-1}\xi}-1]\Big).
\end{eqnarray}

In particular, we infer that $Y(\omega,0,t_l,t_{l-1})=1$, and hence by identity \eqref{eq Q4 01}, we have
$$
\mathbb{E}^{\Q}\Big(\mathcal{M}^{\eps}_{t_l}(\psi_{\eps})|\G_{t_{l-1}}\Big)
=
\mathcal{M}^{\eps}_{t_{l-1}}(\psi_{\eps})Y(\omega,0,t_l,t_{l-1})
=
\mathcal{M}^{\eps}_{t_{l-1}}(\psi_{\eps}).
$$
Furthermore, by employing the above argument, we can easily verify
$$
\mathbb{E}^{\Q}\Big(\mathcal{M}^{\eps}_{t_l}(\psi_{\eps})|\G_{t}\Big)
=
\mathcal{M}^{\eps}_{t}(\psi_{\eps}),\ t\in[0,T].
$$
This implies that the process $\{\mathcal{M}^{\eps}_t(\psi_{\eps}),\ t\geq 0\}$ is an $\Gb$-martingale {on $(\bar{\Omega},\G,\Gb,\Q)$}. Hence we infer that $\mathbb{P}^{\eps}_T$ is a well-defined probability measure.

\vskip 0.2cm
Inserting identity (\ref{eq Q4 star}) into identity (\ref{eq Q4 01}), we arrive at
\begin{eqnarray*}
 \mathbb{E}^{\mathbb{P}^{\eps}_T}\Big(e^{\sqrt{-1} \xi N^{{\eps}^{-1}\varphi_{\eps}}(T,{O})}\Big)
=
 \mathbb{E}^{\Q}\Big[
         e^{\sqrt{-1} \xi N^{{\eps}^{-1}\varphi_{\eps}}(t_{l-1},{O})}
      \cdot
         \mathcal{M}^{\eps}_{t_{l-1}}(\psi_{\eps})
                \Big]
                e^{(t_l-t_{l-1})\nu({O}){\eps}^{-1}[e^{\sqrt{-1}\xi}-1]}.
\end{eqnarray*}
By induction, we can get
\begin{eqnarray*}
 \mathbb{E}^{\mathbb{P}^{\eps}_T}\Big(e^{\sqrt{-1} \xi N^{{\eps}^{-1}\varphi_{\eps}}(T,{O})}\Big)
=
                \exp\Big(T\nu({O}){\eps}^{-1}[e^{\sqrt{-1}\xi}-1]\Big).
\end{eqnarray*}

We have proved that if $\varphi_{\eps}$ is a step process, then the law of ${\eps} N^{{\eps}^{-1}\varphi_{\eps}}$ on $(\bar{\Omega},\G,\Gb,\mathbb{P}^{\eps}_T)$ is equal to the law of ${\eps} N^{{\eps}^{-1}}$ {on $(\bar{\Omega},\G,\Gb,\Q)$}.
The proof of \textbf{Step 1} is now over. \end{proof}

\vskip 0.2cm

\begin{proof}[\bf{Step 2. The general case}] Let us assume that $\varphi_{\eps}\in\bar{\mathbb{A}}_b$. Then
the definition \eqref{eqn-A_b} of $\bar{\mathbb{A}}_b$ implies that there exists $n\in\mathbb{N}$ such that $\varphi_{\eps}\in\bar{\mathbb{A}}_{b,n}$.
Hence, by Lemma \ref{Lem Q4 01}, there exists a sequence $\psi_m\in\bar{\mathbb{A}}_{b,n}$, $m=1,2,\cdots$ satisfying conditions {\rm\textbf{(R1)}} and {\rm\textbf{(R2)}} from that lemma with $\varphi$ replaced by $\varphi_{\eps}$.

Applying Step 1 to function $\psi_m$, we get
\begin{itemize}
 \item for any ${O}\in\mathcal{B}({\rZ})$ and $\nu({O})<\infty$,
   \begin{eqnarray}\label{Eq Q4 step 1 01}
     \mathbb{E}^{\Q}\Big(\exp\Big(\sqrt{-1}\xi N^{{\eps}^{-1}\psi_m}(T,{O})\Big)\mathcal{M}^{\eps}_T(\frac{1}{\psi_m}) \Big)
     =
     \exp\Big(T\nu({O}){\eps}^{-1}(e^{\sqrt{-1}\xi}-1)\Big),
   \end{eqnarray}

 \item for any $0\leq t_1<t_2\leq T$,
   \begin{eqnarray}\label{Eq Q4 step 1 02}
     \mathbb{E}^{\Q}\Big(\mathcal{M}^{\eps}_{t_2}(\frac{1}{\psi_m}) |\G_{t_1}\Big)
     =
     \mathcal{M}^{\eps}_{t_1}(\frac{1}{\psi_m}),\ \ \ \Q\text{-a.s.}
   \end{eqnarray}

\end{itemize}

In order to prove our results, we first prove that there exists a subsequence, which for simplicity still be denoted by $m$, such that
\begin{eqnarray}\label{Eq Q4 step 1 03}
     \lim_{m\to\infty}
     \mathbb{E}^{\Q}\Big(\Big|\mathcal{M}^{\eps}_{t}(\frac{1}{\psi_m})-\mathcal{M}^{\eps}_{t}(\varphi_{\eps})\Big|\Big)
     =
     0,
   \end{eqnarray}
and, for any ${O}\in\mathcal{B}({\rZ})$ satisfying $\nu({O})<\infty$,
\begin{eqnarray}\label{Eq Q4 step 1 04}
     \lim_{m\to\infty}
     \mathbb{E}^{\Q}
     \Big(\sup_{t\in[0,T]}\Big| N^{{\eps}^{-1}\psi_m}(t,{O})-N^{{\eps}^{-1}\varphi_{\eps}}(t,{O})\Big|\Big)
     =
     0.
   \end{eqnarray}


To make the proof in the following more structured, we have divided our argument into four parts. Recall that the set $\bar{\mathbb{A}}_{b,n}$ \addaok{was defined in \eqref{eqn-barA_bn}.}

\noindent
\textbf{Part 1:} For any $\psi\in \bar{\mathbb{A}}_{b,n}$, we have
\begin{eqnarray}\label{eq step claim 1 01}
\int_0^T|\psi(s,z)|\,ds\leq nT,\ \ \ \mbox{ for }\ (z,\omega)\in \rZ\times\bar{\Omega},
\end{eqnarray}
and
\begin{eqnarray}\label{eq step claim 1 02}
&&\hspace{-1truecm}\Big|\int_{(0,T]\times {\rZ}\times[0,{\eps}^{-1}\psi(s,z)]}\log(\frac{1}{\psi(s,z)})\Nbar(ds,dz,dr)\nonumber
\\&&                     +
\int_{(0,T]\times {\rZ}\times[0,{\eps}^{-1}\psi(s,z)]}\Big(-\frac{1}{\psi(s,z)}+1\Big)\nu(dz)\,dsdr\Big|\nonumber\\
&=&
\Big|\int_{(0,T]\times K_n\times[0,{\eps}^{-1}\psi(s,z)]}\log(\frac{1}{\psi(s,z)})\Nbar(ds,dz,dr)\nonumber
 \\&&                    +
\int_{(0,T]\times K_n\times[0,{\eps}^{-1}\psi(s,z)]}\Big(-\frac{1}{\psi(s,z)}+1\Big)\nu(dz)\,dsdr\Big|\nonumber\\
&\leq&
\int_{(0,T]\times K_n\times[0,{\eps}^{-1}n]}\log n\Nbar(ds,dz,dr)
                     +
\int_{(0,T]\times K_n\times[0,{\eps}^{-1}n]}\Big(n+1\Big)\nu(dz)\,dsdr.
\end{eqnarray}
 It is easy to see that, {since } $\nu(K_n)<\infty$,
{\small\begin{equation}\label{eq step claim 1 03}
\mathbb{E}^{\Q}\Big(\exp\Big(\int_{(0,T]\times K_n\times[0,{\eps}^{-1}n]}\log n\Nbar(ds,dz,dr)
                     +
\int_{(0,T]\times K_n\times[0,{\eps}^{-1}n]}\Big(n+1\Big)\nu(dz)\,dsdr\Big)\Big)<\infty.
\end{equation}}

By inequality (\ref{eq step claim 1 01}), we have
\begin{eqnarray}\label{eq step claim 1 01 star}
\int_0^T|\psi_m(s,z)-\varphi_\epsilon(s,z)|\,ds\leq 2nT,\ \ \ on\ (z,\omega)\in \rZ\times\bar{\Omega}.
\end{eqnarray}

\noindent
\textbf{Part 2.} The following inequality holds.
\begin{eqnarray}\label{eq step claim 2 01}
     &&\hspace{-1.5truecm}\lefteqn{ \mathbb{E}^{\Q}
     \Big(\sup_{t\in[0,T]}\Big| N^{{\eps}^{-1}\psi_m}(t,{O})-N^{{\eps}^{-1}\varphi_{\eps}}(t,{O})\Big|\Big)}\nonumber\\
     &=&
     \mathbb{E}^{\Q}
     \Big(\sup_{t\in[0,T]}\Big| \int_0^t\int_{O}\int_0^{{\eps}^{-1}n}\mathds{1}_{[0,{\eps}^{-1}\psi_m(s,z)]}(r)
                   -
                   \mathds{1}_{[0,{\eps}^{-1}\varphi_{\eps}(s,z)]}(r)\Nbar(ds,dz,dr)
               \Big|\Big)\nonumber\\
     &\leq&
     \mathbb{E}^{\Q}
        \Big(\int_0^T\int_{O}\int_0^{{\eps}^{-1}n}\Big| \mathds{1}_{[0,{\eps}^{-1}\psi_m(s,z)]}(r)
                   -
                   \mathds{1}_{[0,{\eps}^{-1}\varphi_{\eps}(s,z)]}(r)\Big|\Nbar(ds,dz,dr)
      \Big)\nonumber\\
      &=&
       \mathbb{E}^{\Q}
        \Big(\int_0^T\int_{O\cap K_n}\int_0^{{\eps}^{-1}n}\Big| \mathds{1}_{[0,{\eps}^{-1}\psi_m(s,z)]}(r)
                   -
                   \mathds{1}_{[0,{\eps}^{-1}\varphi_{\eps}(s,z)]}(r)\Big|\,dr\nu(dz)\,ds
      \Big)\nonumber\\
      &\leq&
      \mathbb{E}^{\Q}
        \Big(\int_0^T\int_{O\cap K_n}{\eps}^{-1}\Big| \psi_m(s,z)
                   -
                   \varphi_{\eps}(s,z)\Big|\nu(dz)\,ds
      \Big).
   \end{eqnarray}

\textbf{Part 3.} The following equality holds.
\begin{eqnarray}\label{eq step claim 3 01}
&&\hspace{-1.5truecm}\lefteqn{ \Big|\int_{(0,T]\times {\rZ}\times[0,{\eps}^{-1}\psi_m(s,z)]}\Big(-\frac{1}{\psi_m(s,z)}+1\Big)\nu(dz)\,dsdr}\nonumber\\
    &&\ \ \ \ \ \ \ \ \ \ \ \ \ \ \ -
    \int_{(0,T]\times {\rZ}\times[0,{\eps}^{-1}\varphi_{\eps}(s,z)]}\Big(-\frac{1}{\varphi_{\eps}(s,z)}+1\Big)\nu(dz)\,dsdr
 \Big|\nonumber\\
&=&
 \Big|\int_{(0,T]\times K_n\times[0,{\eps}^{-1}\psi_m(s,z)]}\Big(-\frac{1}{\psi_m(s,z)}+1\Big)\nu(dz)\,dsdr\nonumber\\
    &&\ \ \ \ \ \ \ \ \ \ \ \ \ \ \ -
    \int_{(0,T]\times K_n\times[0,{\eps}^{-1}\varphi_{\eps}(s,z)]}\Big(-\frac{1}{\varphi_{\eps}(s,z)}+1\Big)\nu(dz)\,dsdr
 \Big|\nonumber\\
&=&
  \Big|\int_{(0,T]\times K_n\times[0,{\eps}^{-1} n]}\mathds{1}_{[0,{\eps}^{-1}\psi_m(s,z)]}(r)\Big(-\frac{1}{\psi_m(s,z)}+1\Big)
 \nonumber \\ &&\hspace{2truecm}\lefteqn{   -
    \mathds{1}_{[0,{\eps}^{-1}\varphi_{\eps}(s,z)]}(r)\Big(-\frac{1}{\varphi_{\eps}(s,z)}+1\Big)\nu(dz)\,dsdr
 \Big|}\nonumber\\
&\leq&
\int_{(0,T]\times K_n\times[0,{\eps}^{-1} n]}
   \Big|\mathds{1}_{[0,{\eps}^{-1}\psi_m(s,z)]}(r)
    -
    \mathds{1}_{[0,{\eps}^{-1}\varphi_{\eps}(s,z)]}(r) \Big|\nu(dz)\,dsdr
\nonumber\\
 &&+
 \int_{(0,T]\times K_n\times[0,{\eps}^{-1} n]}
   \Big|\mathds{1}_{[0,{\eps}^{-1}\psi_m(s,z)]}(r)
    -
    \mathds{1}_{[0,{\eps}^{-1}\varphi_{\eps}(s,z)]}(r) \Big|\frac{1}{\psi_m(s,z)}\nu(dz)\,dsdr
\nonumber\\
 &&+
\int_{(0,T]\times K_n\times[0,{\eps}^{-1} n]}\mathds{1}_{[0,{\eps}^{-1}\varphi_{\eps}(s,z)]}(r)
      \Big|\frac{1}{\psi_m(s,z)}-\frac{1}{\varphi_{\eps}(s,z)}\Big|\nu(dz)\,dsdr
 \nonumber\\
&\leq&
  {\eps}^{-1}(1+n+n^3)\int_{(0,T]\times K_n}|\psi_m(s,z)-\varphi_{\eps}(s,z)| \nu(dz)\,ds.
\end{eqnarray}

\textbf{Part 4:} Using arguments similar to those used in \textbf{Parts} 2 and 3, we have
\begin{eqnarray}\label{eq step claim 4 01}
&&\hspace{-1.5truecm}\lefteqn{ \mathbb{E}^{\Q}
  \Big(\sup_{t\in[0,T]}\Big|
  \int_{(0,t]\times {\rZ}\times[0,{\eps}^{-1}\psi_m(s,z)]}\log(\frac{1}{\psi_m(s,z)})\Nbar(ds,dz,dr)}\nonumber\\
  &&\ \ \ \ \ \ \ \ \ \ \ \ \ \ \ \ \ \ \ \ \ \ -
  \int_{(0,t]\times {\rZ}\times[0,{\eps}^{-1}\varphi_{\eps}(s,z)]}\log(\frac{1}{\varphi_{\eps}(s,z)})\Nbar(ds,dz,dr)
  \Big|\Big)\nonumber\\
&=&
  \mathbb{E}^{\Q}
  \Big(\sup_{t\in[0,T]}\Big|
  \int_{(0,t]\times {\rZ}\times[0,{\eps}^{-1}n]}\mathds{1}_{[0,{\eps}^{-1}\psi_m(s,z)]}(r)\log(\frac{1}{\psi_m(s,z)})\nonumber\\
  &&\ \ \ \ \ \ \ \ \ \ \ \ \ \ \ \ \ \ \ \ \ -
  \mathds{1}_{[0,{\eps}^{-1}\varphi_{\eps}(s,z)]}(r)\log(\frac{1}{\varphi_{\eps}(s,z)})\Nbar(ds,dz,dr)
  \Big|\Big)\nonumber\\
&\leq&
  \mathbb{E}^{\Q}\Big(
  \int_{(0,T]\times K_n\times[0,{\eps}^{-1}n]}\Big|
  \mathds{1}_{[0,{\eps}^{-1}\psi_m(s,z)]}(r)\log(\frac{1}{\psi_m(s,z)})
 \nonumber \\ &&\hspace{2.9truecm}\lefteqn{  -
  \mathds{1}_{[0,{\eps}^{-1}\varphi_{\eps}(s,z)]}(r)\log(\frac{1}{\varphi_{\eps}(s,z)})\Big|\Nbar(ds,dz,dr)
\Big)}\nonumber\\
&=&
 \mathbb{E}^{\Q}\Big(
  \int_{(0,T]\times K_n\times[0,{\eps}^{-1}n]}\Big|
  \mathds{1}_{[0,{\eps}^{-1}\psi_m(s,z)]}(r)\log(\frac{1}{\psi_m(s,z)})
   \nonumber \\ &&\hspace{2.9truecm}\lefteqn{ -
  \mathds{1}_{[0,{\eps}^{-1}\varphi_{\eps}(s,z)]}(r)\log(\frac{1}{\varphi_{\eps}(s,z)})\Big|\nu(dz)\,dsdr
\Big)}\nonumber\\
&\leq&
   {\eps}^{-1}C_n\mathbb{E}^{\Q}\Big(
    \int_{(0,T]\times K_n}|\psi_m(s,z)-\varphi_{\eps}(s,z)| \nu(dz)\,ds
   \Big).
\end{eqnarray}

Keeping in mind $\nu(K_n)<\infty$, we consider together
assertion {\rm\textbf{(R2)}} from Lemma \ref{Lem Q4 01},
the estimates in (\ref{eq step claim 1 01 star}),
and \textbf{Parts} 2, 3 and 4.
Doing so, and applying the Lebesgue dominated convergence theorem (DCT),
we get equality (\ref{Eq Q4 step 1 04}),
as well as the following equality
\begin{eqnarray}\label{eq step conclusion 01}
&&\lim_{m\to\infty}\mathbb{E}^{\Q}
  \Big(\sup_{t\in[0,T]}\Big|
  \int_{(0,t]\times {\rZ}\times[0,{\eps}^{-1}\psi_m(s,z)]}\log(\frac{1}{\psi_m(s,z)})\Nbar(ds,dz,dr)\nonumber\\
  &&\ \ \ \ \ \ \ \ \ \ \ \ \ \ \ \ \ \ \ \ \ \ -
  \int_{(0,t]\times {\rZ}\times[0,{\eps}^{-1}\varphi_{\eps}(s,z)]}\log(\frac{1}{\varphi_{\eps}(s,z)})\Nbar(ds,dz,dr)
  \Big|\Big)=0,
\end{eqnarray}
and, $\Q\text{-a.s.}$,
\begin{eqnarray}\label{eq step conclusion 02}
&&\lim_{m\to\infty}
 \Big|\int_{(0,T]\times {\rZ}\times[0,{\eps}^{-1}\psi_m(s,z)]}\Big(-\frac{1}{\psi_m(s,z)}+1\Big)\nu(dz)\,dsdr\nonumber\\
    &&\ \ \ \ \ \ \ \ \ \ \ \ \ \ \ -
    \int_{(0,T]\times {\rZ}\times[0,{\eps}^{-1}\varphi_{\eps}(s,z)]}\Big(-\frac{1}{\varphi_{\eps}(s,z)}+1\Big)\nu(dz)\,dsdr
 \Big|=0.
\end{eqnarray}

{Let us observe that in view of } (\ref{eq step conclusion 01}),
there exists a subsequence, which for simplicity is still denoted by $m$, such that $\Q\text{-a.s.}$,
\begin{eqnarray}\label{eq step conclusion 03}
&&\lim_{m\to\infty}
  \Big(\sup_{t\in[0,T]}\Big|
  \int_{(0,t]\times {\rZ}\times[0,{\eps}^{-1}\psi_m(s,z)]}\log(\frac{1}{\psi_m(s,z)})\Nbar(ds,dz,dr)\nonumber\\
  &&\ \ \ \ \ \ \ \ \ \ \ \ \ \ \ \ \ \ \ \ \ \ -
  \int_{(0,t]\times {\rZ}\times[0,{\eps}^{-1}\varphi_{\eps}(s,z)]}\log(\frac{1}{\varphi_{\eps}(s,z)})\Nbar(ds,dz,dr)
  \Big|\Big)=0.
\end{eqnarray}
Combining (\ref{eq step conclusion 02}), (\ref{eq step conclusion 03}), (\ref{eq step claim 1 02}), (\ref{eq step claim 1 03}), (\ref{eqn-M^eps}), and the definition
of $\mathcal{M}_T^{\eps}(\cdot)$, and by again employing the Lebesgue DCT, we infer that (\ref{Eq Q4 step 1 03}) holds.

\vskip 0.2cm
Having proved (\ref{Eq Q4 step 1 03}) and (\ref{Eq Q4 step 1 04}), we are now in position to prove (S1) and (S4).
\vskip 0.1cm

Since (S1) can be immediately obtained by (\ref{Eq Q4 step 1 03}) and (\ref{Eq Q4 step 1 02}), we now prove (S4).
\vskip 0.1cm

By (\ref{Eq Q4 step 1 04}), there exists a subsequence, for simplicity, still denoted by $m$, such that
$$
 \lim_{m\to\infty}N^{{\eps}^{-1}\psi_m}(t,{O})=N^{{\eps}^{-1}\varphi_{\eps}}(t,{O}),\ \ \ \Q\text{-a.s.}
$$
Combining this result with (\ref{Eq Q4 step 1 03}) and \textbf{Part} 1, using the Lebesgue DCT again, we have
\begin{eqnarray*}
&&\hspace{-1truecm}\lefteqn{\Big|
\mathbb{E}^{\Q}\Big(\exp\Big(\sqrt{-1}\xi N^{{\eps}^{-1}\psi_m}(T,{O})\Big)\mathcal{M}^{\eps}_T(\frac{1}{\psi_m}) \Big)
-
\mathbb{E}^{\Q}\Big(\exp\Big(\sqrt{-1}\xi N^{{\eps}^{-1}\varphi_{\eps}}(T,{O})\Big)\mathcal{M}^{\eps}_T(\varphi_{\eps}) \Big)
\Big|}\nonumber\\
&\leq&
\mathbb{E}^{\Q}\Big(\Big|\mathcal{M}^{\eps}_T(\frac{1}{\psi_m})-\mathcal{M}^{\eps}_T(\varphi_{\eps})\Big|\Big)\nonumber\\
&&+
\mathbb{E}^{\Q}\Big(
    \Big|\exp\Big(\sqrt{-1}\xi N^{{\eps}^{-1}\varphi_{\eps}}(T,{O})\Big)-\exp\Big(\sqrt{-1}\xi N^{{\eps}^{-1}\psi_m}(T,{O})\Big)\Big|
    \mathcal{M}^{\eps}_T(\varphi_{\eps})
    \Big)\nonumber\\
&\leq&
  \mathbb{E}^{\Q}\Big(\Big|\mathcal{M}^{\eps}_T(\frac{1}{\psi_m})-\mathcal{M}^{\eps}_T(\varphi_{\eps})\Big|\Big)\nonumber\\
&&+
\mathbb{E}^{\Q}\Big(
    \Big|\exp\Big(\sqrt{-1}\xi N^{{\eps}^{-1}\varphi_{\eps}}(T,{O})\Big)-\exp\Big(\sqrt{-1}\xi N^{{\eps}^{-1}\psi_m}(T,{O})\Big)\Big|\nonumber\\
    &&\ \ \ \ \ \ \ \ \cdot\exp\Big(\int_{(0,T]\times K_n\times[0,{\eps}^{-1}n]}\log n\Nbar(ds,dz,dr)
                   +
    \int_{(0,T]\times K_n\times[0,{\eps}^{-1}n]}\Big(n+1\Big)\nu(dz)\,dsdr\Big)
    \Big).\nonumber
\end{eqnarray*}
Since the RHS above $\to 0$ as $m\to\infty$, we infer, by recalling (\ref{Eq Q4 step 1 01}), that
   \begin{eqnarray*}
     \mathbb{E}^{\Q}\Big(\exp\Big(\sqrt{-1}\xi N^{{\eps}^{-1}\varphi_{\eps}}(T,{O})\Big)\mathcal{M}^{\eps}_T(\varphi_{\eps}) \Big)
     =
     \exp\Big(T\nu({O}){\eps}^{-1}(e^{\sqrt{-1}\xi}-1)\Big),
   \end{eqnarray*}
for any ${O}\in\mathcal{B}({\rZ})$ such that $\nu({O})<\infty$, which implies assertion (S4) in Theorem \ref{thm-Girsanov}.
The proof of \textbf{Step 2} is now complete. \end{proof}


\begin{proof}[\bf{Step 3. Proof of assertion (S3) in Theorem \ref{thm-Girsanov}}]

{We observe that}, by (\ref{eqn-M^eps}) and arguments similar to those for
(\ref{eq step claim 1 02}), {$\Q\text{-a.s.}$}
\begin{eqnarray*}
&&\hspace{-2truecm}\lefteqn{\exp\Bigl( \int_{(0,T]\times K_n\times[0,{\eps}^{-1}n]}-\log n\Nbar(ds,dz,dr)
}\\
                     &+& \int_{(0,T]\times K_n\times[0,{\eps}^{-1}n]}[-(n+1)]\nu(dz)\,dsdr\Bigr)\leq
\mathcal{M}^{\eps}_T(\psi_{\eps}),
\end{eqnarray*}
and
\begin{eqnarray*}
&&\hspace{-2truecm}\lefteqn{\exp\Bigl( \int_{(0,T]\times K_n\times[0,{\eps}^{-1}n]}\log n\Nbar(ds,dz,dr)}
\\                     &+&
\int_{(0,T]\times K_n\times[0,{\eps}^{-1}n]}\Big(n+1\Big)\nu(dz)\,dsdr \Bigr)
\geq
\mathcal{M}^{\eps}_T(\psi_{\eps}).
\end{eqnarray*}
Using the facts that $\nu(K_n)<\infty$ and that $\int_{(0,t]\times K_n\times[0,{\eps}^{-1}n]}\Nbar(ds,dz,dr)$ only has finite jumps on $[0,T]$ $\Q$-a.s., we can conclude that {the probability measures} $\Q$ and $\mathbb{P}_T^{\eps}$ are {equivalent}, i.e., assertion (S3) in Theorem \ref{thm-Girsanov}.
This completes the proof of \textbf{Step 3}. \end{proof}

\vskip 0.2cm

The proof of Theorem \ref{thm-Girsanov} is thus complete. \end{proof}

\section{Verification of the \textbf{Claim-LDP-2}}\label{sec-LDP-verification}


 {The main result of this section is Proposition \ref{prop3 02}, in which we prove \textbf{Claim-LDP-2}. To this end, we first
prove that the process $X^{\eps}:=\mathcal{G}^{\eps}({\eps} N^{{\eps}^{-1}\varphi_{\eps}})$ is the unique solution of the controlled SPDE (\ref{eq3 prop2 00}), which is given in Lemma \ref{lem LDP stochastic 00}. Then to prove \textbf{Claim-LDP-2}, we only need to \dela{consider $X^{\eps}$, and we }prove some \emph{a priori} estimates and establish the tightness of laws of the processes $X^{\eps}$, $\eps>0$,
which we do in Lemmata \ref{lem LDP stochastic 01}--\ref{lem LDP stochastic 05}.
The key to proving Lemma \ref{lem LDP stochastic 00} is a Girsanov-type theorem for Poisson random measures, which is formulated within Theorem \ref{thm-Girsanov}.

This section is divided into two subsections. In the second one, we formulate and prove Proposition \ref{prop3 02} from which \textbf{Claim-LDP-2} follows. In the first subsection, we prove Lemma \ref{lem LDP stochastic 00} and find necessary estimates.

\subsection{A representation result and \emph{a priori}  estimates}\label{subsec-4.2}

Fix $u_0\in \rV$ and $f\in L^2([0,T];\rH)$.
Assume that the control $\phieps$ belongs to the set $\mathcal{U}$
(see equality (\ref{eqn-U^N})).
Let us consider the following controlled SPDE:
\begin{eqnarray}\label{eq3 prop2 00}
dX^{\eps}(t)+{\rA}X^{\eps}(t)dt + \rB(X^{\eps}(t))\,dt
  &=&
  f(t)\,dt\\
  \nonumber
  &&+
  {\eps}\int_{{\rZ}}G(X^{\eps}(t-),z)\Big(N^{{\eps}^{-1}\phieps}(dz,dt)-{\eps}^{-1}\nu(dz)\,dt\Big),\\
  \label{eq3 prop2 00-second}
  &=&f(t)\,dt + \int_{{\rZ}}G(X^{\eps}(t),z)(\phieps(t,z)-1)\nu(dz)\,dt\\
  \nonumber
  && +
  {\eps}\int_{{\rZ}}G(X^{\eps}(t-),z)\widetilde{N}^{{\eps}^{-1}\phieps}(dz,dt),\\
X^{\eps}(0)&=&u_0.
\nonumber
\end{eqnarray}

{Note that below, in (\ref{eq3 sto-lemm 00}) for example, we use the second version of the above equation, i.e.,
\eqref{eq3 prop2 00-second}.
We also observe that it is easy to see that the integral ${\eps}\int_0^t\int_{{\rZ}}G(X^{\eps}(s-),z)\Big(N^{{\eps}^{-1}\phieps}(dz,ds)-{\eps}^{-1}\nu(dz)\,ds\Big)$ exists. }
\vskip 0.2cm

Recall the definition of $\mathcal{G}^{\eps}$ in the proof of Theorem \ref{thm-LDP} (see around (\ref{eqn-u^eps})). By Corollary \ref{cor-solution to eq LDP u epsilon 00} we infer that the process
\begin{equation}\label{eq LDP 0000}
u^{\eps}
=\mathcal{G}^{\eps}({\eps} N^{{\eps}^{-1}})
\end{equation}
 is the unique solution of Problem (\ref{eq LDP u epsilon 00}) on the probability space on $(\bar{\Omega},\G,\Gb,\Q)$.

 We prove the following fundamental result.

\begin{lemma}\label{lem LDP stochastic 00}
Assuming $\eps>0$, for every process $\phieps\in\bar{\mathbb{A}}_b$ defined on $(\bar{\Omega},\G,\Gb,\Q)$, the process $X^{\eps}$ defined by
\begin{equation}\label{eqn-X^eps}
X^{\eps}=\mathcal{G}^{\eps}({\eps} N^{{\eps}^{-1}\phieps})
\end{equation}
 is the unique solution of Equation (\ref{eq3  prop2 00}).
\end{lemma}

\begin{proof}[Proof of Lemma \ref{lem LDP stochastic 00}]
Fix $\eps>0$ and a process $\phieps\in\bar{\mathbb{A}}_b$ defined on $(\bar{\Omega},\G,\Gb,\Q)$. Define a process $X^{\eps}$ by formula
\eqref{eqn-X^eps}. Then
by assertion (S4) in Theorem \ref{thm-Girsanov} and the definition of $\mathcal{G}^{\eps}$, we infer that the process
$X^{\eps}$ is the unique solution of (\ref{eq3 prop2 00})
on $(\bar{\Omega},\G,\Gb,\mathbb{P}^{\eps}_T)$, that is
\begin{itemize}
 \item [(C1)] $X^{\eps}$ is $\Gb$-{progressively measurable} process,

 \item [(C2)] trajectories of $X^{\eps}$ belong to $\UT$ $\mathbb{P}^{\eps}_T$-a.s.,

 \item [(C3)] the following equality holds, in $\rH$, for all $t\in[0,T]$, $\mathbb{P}^{\eps}_T$-a.s.:
  \begin{equation}\label{eqn-C3}
\begin{split}   X^{\eps}(t)
   &=
   u_0-\int_0^t{\rA}X^{\eps}(s)\,ds - \int_0^t\rB(X^{\eps}(s))\,ds\\
    &\ \ \ +
    \int_0^tf(s)\,ds + {\eps}\int_0^t\int_{{\rZ}}G(X^{\eps}(s-),z)(N^{{\eps}^{-1}\phieps}(dz,ds)-{\eps}^{-1}\nu(dz)\,ds).
    \end{split}
\end{equation}
\end{itemize}

Next, we prove that the process $X^{\eps}$ is the unique solution of (\ref{eq3 prop2 00})
on $(\bar{\Omega},\G,\Gb,\Q)$;
that is,
\begin{itemize}
 \item [(C1-0)] $X^{\eps}$ is $\Gb$-{progressively measurable} process,

 \item [(C2-0)] {trajectories} of $X^{\eps}$ belong to $\UT$ $\mathbb{Q}$-a.s.,

 \item [(C3-0)] the following equality holds, in $\rH$, for all $t\in[0,T]$, $\mathbb{Q}$-a.s.:
  \begin{equation}
  \label{eqn-C3-b}
\begin{split}
   X^{\eps}(t)
   &=
   u_0-\int_0^t{\rA}X^{\eps}(s)\,ds - \int_0^t\rB(X^{\eps}(s))\,ds\\
    &\ \ \ +
    \int_0^tf(s)\,ds + {\eps}\int_0^t\int_{{\rZ}}G(X^{\eps}(s-),z)(N^{{\eps}^{-1}\phieps}(dz,ds)-{\eps}^{-1}\nu(dz)\,ds).
\end{split}
\end{equation}
\end{itemize}
Note that  despite the two measures $\Q$ and $\mathbb{P}^{\eps}_T$
being equivalent, equality \eqref{eqn-C3-b} does not follow from \eqref{eqn-C3} without additional
justification. We provide this justification below. The proof is divided into two steps.\\
{\textbf{Step 1}:} We prove that the process $X^{\eps}$ satisfies (C1-0)-(C3-0).

Let us observe that obviously,  condition (C1) implies condition (C1-0). In view of assertion (S3) in Theorem \ref{thm-Girsanov}, i.e., that the measures $\Q$ and $\mathbb{P}^{\eps}_T$ are equivalent, condition (C2) implies condition (C2-0).

We are now in position to prove that condition (C3-0) holds as well. For this, let us fix a natural number $n$ (in view of the definition \eqref{eqn-A_b} of the set $\bar{\mathbb{A}}_b$).
Observe that equality (\ref{eq3 prop2 00}) can be rewritten as
\begin{eqnarray*}
dX^{\eps}(t)+{\rA}X^{\eps}(t)\,dt&&\!\!\!\!\!\!\!\!\!\!\!\!+ \rB(X^{\eps}(t))\,dt\\
  &=&
  f(t)\,dt
 +
  {\eps}\int_{K_n}G(X^{\eps}(t-),z)\Big(N^{{\eps}^{-1}\phieps}(dz,dt)-{\eps}^{-1}\nu(dz)\,dt\Big)\nonumber\\
  &&
   +
  {\eps}\int_{K_n^c}G(X^{\eps}(t-),z)\Big(N^{{\eps}^{-1}\phieps}(dz,dt)-{\eps}^{-1}\nu(dz)\,dt\Big)\nonumber\\
  &=&
   f(t)\,dt
  +
  {\eps}\int_{K_n}G(X^{\eps}(t-),z)\Big(N^{{\eps}^{-1}\phieps}(dz,dt)-{\eps}^{-1}\nu(dz)\,dt\Big)\nonumber\\
  &&
  +
  {\eps}\int_{K_n^c}G(X^{\eps}(t-),z)\Big(N^{{\eps}^{-1}}(dz,dt)-{\eps}^{-1}\nu(dz)\,dt\Big)\nonumber\\
X_0^{\eps}&=&u_0.
\end{eqnarray*}
The second equality follows from Proposition \ref{prop-PRM-b} because  $\phieps(s,z,\omega)=1$ if $(s,z,\omega)\in[0,T]\times K_n^c \times \bar{\Omega}$.

Thus, we have the following:
\begin{itemize}
 \item [(V1)] Since $\nu(K_n)<\infty$, by assertion (S4) in Theorem \ref{thm-Girsanov}, we infer that there exists $\Omega_1\subset\bar{\Omega}$ with $\mathbb{P}^{\eps}_T(\Omega_1)=1$ such that for any $\omega\in\Omega_1$, the process  $\{N^{{\eps}^{-1}\phieps}((0,t]\times K_n),\ t\in[0,T]\}$ has only finite
 jumps. Hence for any $\omega\in\Omega_1$, the integrals
 \[\int_0^t\int_{K_n}G(X^{\eps}(s-),z)N^{{\eps}^{-1}\phieps}(dz,ds) \mbox{ and }
 \int_0^t\int_{K_n}G(X^{\eps}(s-),z)\nu(dz)\,ds\]
 are well-defined as the Lebesgue-Stieltjes integrals.

 \item [(V2)] Similar to the proof of condition (V1), for any $m>n$, the integrals
 \[\int_0^t\int_{K^c_n\cap K_m}G(X^{\eps}(s-),z)N^{{\eps}^{-1}}(dz,ds) \mbox{ and }
\int_0^t\int_{K^c_n\cap K_m}G(X^{\eps}(s-),z)\nu(dz)\,ds\]
are well-defined as the  Lebesgue-Stieltjes integrals, $\mathbb{P}^{\eps}_T$-a.s.

 \item [(V3)] By \cite[Section 3 in Chapter II, pages 59-63]{Ikeda-Watanabe} and by the definition of the integral
 $$
 \int_0^t\int_{K_n^c}G(X^{\eps}(s-),z)\Big(N^{{\eps}^{-1}}(dz,ds)-{\eps}^{-1}\nu(dz)\,ds\Big)
 $$
 {on} the probability space $(\bar{\Omega},\G,\Gb,\mathbb{P}^{\eps}_T)$, there exist $\Omega_2\subset\bar{\Omega}$ with $\mathbb{P}^{\eps}_T(\Omega_2)=1$ and a subsequence $\{m_k\}$ such that for any $\omega\in\Omega_2$
 \begin{eqnarray*}
   &&\hspace{-1truecm}\lefteqn{\lim_{m_k\to\infty}\int_0^t\int_{K_n^c\cap K_{m_k}}G(X^{\eps}(s-),z)\Big(N^{{\eps}^{-1}}(dz,ds)-{\eps}^{-1}\nu(dz)\,ds\Big)}\\
   &=&
   \int_0^t\int_{K_n^c}G(X^{\eps}(s-),z)\Big(N^{{\eps}^{-1}}(dz,ds)-{\eps}^{-1}\nu(dz)\,ds\Big).
 \end{eqnarray*}
\end{itemize}

Using an argument as in the proof of assertions (V1)-(V3), and the equality
\begin{eqnarray*}
N^{{\eps}^{-1}\phieps}((0,t]\times K_n)
&=&
\int_0^t\int_{K_n}\int_0^\infty\mathds{1}_{[0,{\eps}^{-1}\phieps(z,s)]}(r)\Nbar(dz,ds,dr)\\
&=&
\int_0^t\int_{K_n}\int_0^{{\eps}^{-1} n}\mathds{1}_{[0,{\eps}^{-1}\phieps(z,s)]}(r)\Nbar(dz,ds,dr),
\end{eqnarray*}
we infer three facts:
\begin{itemize}
 \item [(V1-0)] There exists $\Omega_3\subset\bar{\Omega}$ with $\mathbb{Q}(\Omega_3)=1$ such that $\{N^{{\eps}^{-1}\phieps}((0,t]\times K_n),\ t\in[0,T]\}$ has only finite
 jumps  for any $\omega\in\Omega_3$. Hence for any $\omega\in\Omega_3$, $\int_0^t\int_{K_n}G(X^{\eps}(s-),z)N^{{\eps}^{-1}\phieps}(dz,ds)$ and
 $\int_0^t\int_{K_n}G(X^{\eps}(s-),z)\nu(dz)\,ds$ are well-defined as the Lebesgue-Stieltjes integrals.

 \item [(V2-0)] For any $m>n$, $\int_0^t\int_{K^c_n\cap K_m}G(X^{\eps}(s-),z)N^{{\eps}^{-1}}(dz,ds)$ and
 $\int_0^t\int_{K^c_n\cap K_m}G(X^{\eps}(s-),z)\nu(dz)\,ds$ are
well-defined $\mathbb{Q}$-a.s. as the Lebesgue-Stieltjes integrals.

 \item [(V3-0)] By the definition of
 $$
 \int_0^t\int_{K_n^c}G(X^{\eps}(s-),z)\Big(N^{{\eps}^{-1}}(dz,ds)-{\eps}^{-1}\nu(dz)\,ds\Big)
 $$
 {on} $(\bar{\Omega},\G,\Gb,\mathbb{Q})$, there exist $\Omega_4\subset\bar{\Omega}$ with $\mathbb{Q}(\Omega_4)=1$
 and a subsequence of the subsequence of $\{m_k\}$ from (V3))
(for simplicity still denoted by $m_k$), such that for any $\omega\in\Omega_4$,
 \begin{eqnarray*}
   &&\hspace{-2truecm}\lefteqn{\lim_{m_k\to\infty}\int_0^t\int_{K_n^c\cap K_{m_k}}G(X^{\eps}(s-),z)\Big(N^{{\eps}^{-1}}(dz,ds)-{\eps}^{-1}\nu(dz)\,ds\Big)}\\
   &=&
   \int_0^t\int_{K_n^c}G(X^{\eps}(s-),z)\Big(N^{{\eps}^{-1}}(dz,ds)-{\eps}^{-1}\nu(dz)\,ds\Big).
 \end{eqnarray*}
\end{itemize}

By (C3) and (V1)-(V3), there exists $\Omega_5\subset \bar{\Omega}$ such that $\mathbb{P}^{\eps}_T(\Omega_5)=1$ and, for any $\omega\in\Omega_5$,
\begin{eqnarray*}
&&\hspace{-2truecm}\lefteqn{-\lim_{m_k\to\infty}\int_0^t\int_{K_n^c\cap K_{m_k}}G(X^{\eps}(s-),z)\Big(N^{{\eps}^{-1}}(dz,ds)-{\eps}^{-1}\nu(dz)\,ds\Big)}\\
   &=&
X^{\eps}(t)-u_0+\int_0^t{\rA}X^{\eps}(s)\,ds + \int_0^t\rB(X^{\eps}(s))\,ds
  -
  \int_0^tf(s)\,ds\\
  &&-
  {\eps}\int_0^t\int_{K_n}G(X^{\eps}(s-),z)\Big(N^{{\eps}^{-1}\phieps}(dz,ds)-{\eps}^{-1}\nu(dz)\,ds\Big),\text{ in }\rH.
\end{eqnarray*}

\addaok{Since, by assertion} (S3) in Theorem \ref{thm-Girsanov}, the measures $\Q$ and $\mathbb{P}^{\eps}_T$ are {equivalent}, we deduce that $\mathbb{Q}(\cap_{i=1}^5\Omega_i)=1$. Moreover, since by (V1) and (V1-0) the right side of the above equality is path-wise well-defined for $\omega\in\Omega_1\cap\Omega_3$ , we infer for any $\omega\in\cap_{i=1}^5\Omega_i$
\begin{eqnarray*}
&&\hspace{-2truecm}\lefteqn{-\lim_{m_k\to\infty}\int_0^t\int_{K_n^c\cap K_{m_k}}G(X^{\eps}(s-),z)\Big(N^{{\eps}^{-1}}(dz,ds)-{\eps}^{-1}\nu(dz)\,ds\Big)}\\
   &=&
X^{\eps}(t)-u_0+\int_0^t{\rA}X^{\eps}(s)\,ds + \int_0^t\rB(X^{\eps}(s))\,ds
  -
  \int_0^tf(s)\,ds\\
  &&-
  {\eps}\int_0^t\int_{K_n}G(X^{\eps}(s-),z)\Big(N^{{\eps}^{-1}\phieps}(dz,ds)-{\eps}^{-1}\nu(dz)\,ds\Big),\text{ in }\rH.
\end{eqnarray*}

\vskip 0.2cm
Combining this last equality with (V3-0), the proof of claim (C3-0) is complete. The proof of \textbf{Step 1} is now over.\\
{\textbf{Step 2}:} The solution of (\ref{eq3 prop2 00}) on $(\bar{\Omega},\G,\Gb,\Q)$ is unique.

Assume that a process $Y^\eps$ is another solution of (\ref{eq3 prop2 00}) on $(\bar{\Omega},\G,\Gb,\Q)$, that is, (C1-0)-(C3-0) are satisfied with
$X^\eps$ replaced by $Y^\eps$.
By arguments similar to those for {\textbf{Step 1}}, the process
$Y^{\eps}$ is a solution of (\ref{eq3 prop2 00}) on $(\bar{\Omega},\G,\Gb,\mathbb{P}^{\eps}_T)$, and the uniqueness of solution to (\ref{eq3 prop2 00}) on $(\bar{\Omega},\G,\Gb,\mathbb{P}^{\eps}_T)$ implies that $Y^\eps=X^\eps$ $\mathbb{P}^{\eps}_T$-a.s.. Since the measures $\Q$ and $\mathbb{P}^{\eps}_T$ are equivalent, $Y^\eps=X^\eps$ $\Q$-a.s..

\vskip 0.2cm

Thus the proof of Lemma \ref{lem LDP stochastic 00} is complete.
\end{proof}

\vskip 0.2cm

 Now we give some \emph{a priori} estimates to be used later. For simplicity, in the following $\mathbb{E}^{\mathbb{Q}}$ is denoted by $\mathbb{E}$. Let us recall that the norm $\|\cdot\|^2_{W^{\alpha,2}([0,T],\rV^\prime)}$ was introduced in equality \eqref{eqn-W^alpha2-norm}.

\begin{lemma}\label{lem LDP stochastic 01}
For every $N \in \mathbb{N}$, there exist constants $C_N>0$ and $\eps_N \in (0,1]$, and for every
 $\alpha\in(0,\frac{1}{2})$ there exists a constant $C_{\alpha,N}>0$ such that for every process $\phieps \in \mathcal{U}^N$ and every $\eps \in (0,\eps_N]$,
 the process $X^\eps$ defined by \eqref{eqn-X^eps} satisfies
\begin{eqnarray}\label{eq3 X ep pro 01}
\dela{\sup_{{\eps}\in(0,{\eps}_{N})}}
&&                   \mathbb{E}\Big(\sup_{t\in[0,T]}\|X^{\eps}(t)\|^2_{\rH}+
                   \int_0^T\|X^{\eps}(t)\|^2_{\rV}\,dt\Big) \leq C_N,
\\\label{eq3 X ep pro 02}
\dela{\sup_{{\eps}\in(0,{\eps}_{N})}}
 &&                   \mathbb{E}\Big(\|X^{\eps}\|_{W^{\alpha,2}([0,T],\rV^\prime)}^2\Big)
\leq C_{\alpha,N}.
\end{eqnarray}

\end{lemma}

\begin{proof}
Fix $\eps>0$, $N\in\mathbb{N}$, and process $\phieps \in \mathcal{U}^N$. Let the process $X^\eps$ be defined by formula \eqref{eqn-X^eps}.
By Lemma \ref{lem LDP stochastic 00}, this process
 is the unique solution of (\ref{eq3 prop2 00}) on the probability space $(\bar{\Omega},\G,\Gb,\Q)$.

 Therefore, we can apply {the It\^o} formula to deduce that
\begin{eqnarray}\label{eq3 sto-lemm 00}
&&\hspace{-1truecm}|X^{\eps}(t)|^2_{\rH} + 2\int_0^t\|X^{\eps}(s)\|^2_{\rV}\,ds\nonumber\\
&=& |u_0|^2_{\rH}+ 2\int_0^t\langle f(s),X^{\eps}(s)\rangle_{\rV^\prime,\rV}\,ds
  +
  2\int_0^t\int_{{\rZ}}\langle G(X^{\eps}(s),z),X^{\eps}(s)\rangle_{\rH}(\phieps(s,z)-1)\nu(dz)\,ds\nonumber\\
  &&+
  2{\eps}\int_0^t\int_{{\rZ}}\langle G(X^{\eps}(s-),z), X^{\eps}(s-)\rangle_{\rH}\widetilde{N}^{{\eps}^{-1}\phieps}(dz,ds)\nonumber\\
  &&+
 {\eps}^2\int_0^t\int_{{\rZ}}|G(X^{\eps}(s-),z)|^2_{\rH}N^{{\eps}^{-1}\phieps}(dz,ds)\nonumber\\
&\leq&
  |u_0|^2_{\rH} + \int_0^t\|X^{\eps}(s)\|^2_{\rV}\,ds + \int_0^t\|f(s)\|^2_{\rV^\prime}\,ds\nonumber\\
  &&+
  2\int_0^t(1+2|X^{\eps}(s)|^2_{\rH})\int_{{\rZ}}L_3(z)|\phieps(s,z)-1|\nu(dz)\,ds\nonumber\\
  &&+
  2{\eps}\int_0^t\int_{{\rZ}}\langle G(X^{\eps}(s-),z), X^{\eps}(s-)\rangle_{\rH}\widetilde{N}^{{\eps}^{-1}\phieps}(dz,ds)\nonumber\\
  &&+
 {\eps}^2\int_0^t\int_{{\rZ}}|G(X^{\eps}(s-),z)|^2_{\rH}N^{{\eps}^{-1}\phieps}(dz,ds).
\end{eqnarray}
Set
$$
J_1(t):=2{\eps}\int_0^t\int_{{\rZ}}\langle G(X^{\eps}(s-),z), X^{\eps}(s-)\rangle_{\rH}\widetilde{N}^{{\eps}^{-1}\phieps}(dz,ds)
$$
and
$$
J_2(t):={\eps}^2\int_0^t\int_{{\rZ}}|G(X^{\eps}(s-),z)|^2_{\rH}N^{{\eps}^{-1}\phieps}(dz,ds).
$$
Applying Gronwall's lemma and (\ref{eq3 star}), we get
\begin{eqnarray}\label{eq3 prop2 lem 00}
&&\hspace{-1truecm}\lefteqn{\sup_{t\in[0,T]}|X^{\eps}(t)|^2_{\rH} + \int_0^T\|X^{\eps}(t)\|^2_{\rV}\,dt}\nonumber\\
&\leq&
 C_N\Big(
   |u_0|^2_{\rH} + \int_0^T\|f(s)\|^2_{\rV^\prime}\,ds + 1 + \sup_{t\in[0,T]}|J_1(t)| + J_2(T)
   \Big).
\end{eqnarray}
For $\sup_{t\in[0,T]}|J_1(t)|$, the Burkholder-Davis-Gundy inequality implies that
\begin{eqnarray}\label{eq3 sto lem J1}
&& \hspace{-1truecm}\lefteqn{ \mathbb{E}\Big(\sup_{t\in[0,T]}|J_1(t)|\Big)}\nonumber\\
&\leq&
   C{\eps} \mathbb{E}\Big(
              \int_0^T\int_{{\rZ}}|X^{\eps}(s-)|^2_{\rH}|G(X^{\eps}(s-),z)|^2_{\rH}N^{{\eps}^{-1}\phieps}(dz,ds)
             \Big)^{\frac{1}{2}}\nonumber\\
&\leq&
  {\eps}^{\frac{1}{2}} \mathbb{E}\Big(\sup_{t\in[0,T]}|X^{\eps}(t)|^2_{\rH}\Big)
  +
  C{\eps}^{\frac{1}{2}}\mathbb{E}\Big(\int_0^T(1+|X^{\eps}(s)|^2_{\rH})\int_{{\rZ}}L^2_3(z)\phieps(s,z)\nu(dz)\,ds\Big)\nonumber\\
&\leq&
  C_N{\eps}^{\frac{1}{2}} \mathbb{E}\Big(\sup_{t\in[0,T]}|X^{\eps}(t)|^2_{\rH}\Big)
  +
  C_N{\eps}^{\frac{1}{2}}.
\end{eqnarray}
To deduce the last inequality above, we use the fact
(see (3.3) in \cite[Lemma 3.4]{Budhiraja-Chen-Dupuis}),
that \addjzok{for any fixed $\Im\in \mathcal{H}\cap L^2(\nu)$,
\begin{eqnarray}\label{eq3 star3}
C_{\Im,N}:=\sup_{k\in S^N}\int_0^T\int_{{\rZ}}\Im^2(z)(k(s,z)+1)\nu(dz)\,ds<\infty.
\end{eqnarray}}

As in (\ref{eq3 sto lem J1}), we get
\begin{eqnarray}\label{eq3 sto lem J2}
\mathbb{E}\Big(|J_2(T)|\Big)
\leq
C_N{\eps} + C_N{\eps} \mathbb{E}\Big(\sup_{t\in[0,T]}|X^{\eps}(t)|^2_{\rH}\Big).
\end{eqnarray}
Substituting (\ref{eq3 sto lem J1}) and (\ref{eq3 sto lem J2}) into (\ref{eq3 prop2 lem 00}), and then choosing ${\eps}_N>0$ small enough, we get
(\ref{eq3 X ep pro 01}).

\vskip 0.3cm
Using the same arguments that prove (4.67) in \cite{Zhai-Zhang}, we infer (\ref{eq3 X ep pro 02}).

This completes the proof of Lemma \ref{lem LDP stochastic 01}.
\end{proof}

\vskip 0.3cm

Let us define a stopping time $\tau_{{\eps},M}$ by\footnote{In fact, this stopping time depends on $X^\eps$ so it depends on both $\eps$ and $\phieps$. Hence, it should denoted $\tau_{{X^\eps},M}$ or $\tau_{{\phieps,\eps},M}$. Since these two are cumbersome, we decided not to use them. In the same vein, $X^\eps$ should be denoted by $X^{\eps,\phieps}$, but we decided to use the simpler notation.}

\begin{equation}\label{eqn-tau_eps,M}
\tau_{{\eps},M}:=\inf\{t\geq 0:\ \sup_{s\in[0,t]}|X^{\eps}(s)|^2_{\rH}+\int_0^t\|X^{\eps}(s)\|^2_{\rV}\,ds \deln{\geq}{>} M\},\;\; M>0.
\end{equation}

\deln{We use the convention here that $\inf \emptyset =\infty$.}

Before we continue with our estimates, let us state  the following simple but useful corollary from the previous result and the Chebyshev inequality.

\begin{corollary}\label{cor-Cheb}
In the framework above, we have
\begin{equation}\label{eqn-Cheb}\begin{split}
\mathbb{Q} \bigl( \tau_{{\eps},M} < T\bigr) &\leq \frac{C_N}{M}, \;\; M>0,\\
\mathbb{Q} \bigl( \|X^{\eps}\|_{W^{\alpha,2}([0,T],\rV^\prime)}^2 \geq R^2\bigr)
&\leq \frac{C_{\alpha,N}}{R^2}, \;\; R,M>0.
\end{split}
\end{equation}

\end{corollary}

We have the following estimate.

\begin{lemma}\label{lem LDP stochastic 02}
For all $N \in \mathbb{N}$ and $M>0$, there exist constants $C_{N,M}>0$ and $\eps_{N,M} \in (0,1]$
 such that for every process $\phieps \in \mathcal{U}^N$ and every $\eps \in (0,\eps_{N,M}]$,
 the process $X^\eps$ defined by \eqref{eqn-X^eps} satisfies

\dela{\begin{eqnarray}\label{eq3 X ep pro 03}
C_{N,M}:=\sup_{{\eps}\in(0,{\eps}_{N,M})}\Big(
                   \mathbb{E}\Big(\sup_{t\in[0,T\wedge\tau_{{\eps},M})}\|X^{\eps}(t)\|^2_{\rV}\Big)
                   +
                   \mathbb{E}\Big(\int_0^{T\wedge\tau_{{\eps},M}}\|X^{\eps}(t)\|^2_{\mathcal{D}({\rA})}\,dt\Big)
                 \Big)
<\infty.
\end{eqnarray}
}
\begin{eqnarray}\label{eq3 X ep pro 03}
\sup_{{\eps}\in(0,{\eps}_{N,M})}
                   \mathbb{E}\Big(\sup_{t\in[0,T]}\|X^{\eps}(t\wedge\tau_{{\eps},M})\|^2_{\rV}
                   +
                   \int_0^{T\wedge\tau_{{\eps},M}}\|X^{\eps}(s)\|^2_{\mathcal{D}({\rA})}\,ds
                 \Big)\leq C_{N,M}.
\end{eqnarray}
\end{lemma}

Before we give the proof of Lemma \ref{lem LDP stochastic 02}, the following result is immediate from Lemma \ref{lem LDP stochastic 02} and the Chebyshev inequality.
\begin{corollary}\label{cor-Cheb-2}
In the framework above, for all $M,R>0$, the following inequality holds:
\begin{equation}\label{eqn-Cheb-2}
\mathbb{Q} \bigl( \tau_{{\eps},M} \geq T, \sup_{t\in[0,T]}\|X^{\eps}(t)\|^2_{\rV}
                   +
                   \int_0^{T}\|X^{\eps}(s)\|^2_{\mathcal{D}({\rA})}\,ds \geq R^2\bigr) \leq \frac{C_{N,M}}{R^2}.
\end{equation}
\end{corollary}
\begin{proof} Notice that $\tau_{{\eps},M} \geq T$ iff $T \wedge \tau_{{\eps},M} = T$. Thus, if
$\tau_{{\eps},M} \geq T$ then $t \wedge \tau_{{\eps},M} = t$ for all $t \in [0,T]$.
\end{proof}

\begin{proof}[Proof of Lemma \ref{lem LDP stochastic 02}] We argue as in the proof of Lemma 7.3  in \cite{Brzezniak+Motyl_2013_NSE}.

By {the It\^o} formula and Lemma \ref{lem B baisc prop}, we have, for all $t\in [0,T]$,
\begin{eqnarray}\label{eq3 sto-lemm2 00}
&&\hspace{-1truecm}\|X^{\eps}({t\wedge\tau_{{\eps},M}})\|^2_{\rV} + 2\int^{t\wedge\tau_{{\eps},M}}_0\|X^{\eps}(s)\|^2_{\mathcal{D}({\rA})}\,ds\nonumber\\
&=&
  \|u_0\|^2_{\rV} -2\int_0^{t\wedge\tau_{{\eps},M}}\langle \rB(X^{\eps}(s)),{\rA}X^{\eps}(s)\rangle_{\rH}\,ds
  +
  2\int_0^{t\wedge\tau_{{\eps},M}}\langle f(s),{\rA}X^{\eps}(s)\rangle_{\rH}\,ds \nonumber\\
  &&+
  2\int_0^{t\wedge\tau_{{\eps},M}}\int_{{\rZ}}\langle G(X^{\eps}(s),z),X^{\eps}(s)\rangle_{\rV}(\phieps(s,z)-1)\nu(dz)\,ds\nonumber\\
  &&+
  2{\eps}\int_0^{t\wedge\tau_{{\eps},M}}\int_{{\rZ}}\langle G(X^{\eps}(s-),z), X^{\eps}(s-)\rangle_{\rV}\widetilde{N}^{{\eps}^{-1}\phieps}(dz,ds)\nonumber\\
  &&+
 {\eps}^2\int_0^{t\wedge\tau_{{\eps},M}}\int_{{\rZ}}\|G(X^{\eps}(s-),z)\|^2_{\rV}N^{{\eps}^{-1}\phieps}(dz,ds)\nonumber\\
&\leq&
  \|u_0\|^2_{\rV} + \int_0^{t\wedge\tau_{{\eps},M}}\|X^{\eps}(s)\|^2_{\mathcal{D}({\rA})}\,ds + C\int_0^{t\wedge\tau_{{\eps},M}}\|X^{\eps}(s)\|^4_{\rV}\|X^{\eps}(s)\|^2_{\rH}\,ds
\nonumber \\  &&+
   2\int_0^{t\wedge\tau_{{\eps},M}}|f(s)|^2_{\rH}\,ds\nonumber\\
  &&+
  2\int_0^{t\wedge\tau_{{\eps},M}}(1+2\|X^{\eps}(s)\|^2_{\rV})\int_{{\rZ}}L_2(z)|\phieps(s,z)-1|\nu(dz)\,ds\nonumber\\
  &&+
  2{\eps}\int_0^{t\wedge\tau_{{\eps},M}}\int_{{\rZ}}\langle G(X^{\eps}(s-),z), X^{\eps}(s-)\rangle_{\rV}\widetilde{N}^{{\eps}^{-1}\phieps}(dz,ds)\nonumber\\
  &&+
 {\eps}^2\int_0^{t\wedge\tau_{{\eps},M}}\int_{{\rZ}}\|G(X^{\eps}(s-),z)\|^2_{\rV}N^{{\eps}^{-1}\phieps}(dz,ds).
\end{eqnarray}
We define:
\begin{equation*}\begin{split}
J_1(t)&:=2{\eps}\int_0^t\int_{{\rZ}}\langle G(X^{\eps}(s-),z), X^{\eps}(s-)\rangle_{\rV}\widetilde{N}^{{\eps}^{-1}\phieps}(dz,ds),
\\
J_2(t)&:={\eps}^2\int_0^t\int_{{\rZ}}\|G(X^{\eps}(s-),z)\|^2_{\rV}N^{{\eps}^{-1}\phieps}(dz,ds),
\end{split}\;\;\; t\in [0,T].
\end{equation*}

Hence, by (\ref{eq3 star}),
\begin{equation}
\begin{split}
&\hspace{-0.1truecm}\lefteqn{  \|X^{\eps}({t\wedge\tau_{{\eps},M}})\|^2_{\rV} + \int_0^{t\wedge\tau_{{\eps},M}}\|X^{\eps}(s)\|^2_{\mathcal{D}({\rA})}\,ds}\nonumber\\
&\leq
   \|u_0\|^2_{\rV}
   + 2\int_0^{t\wedge\tau_{{\eps},M}}\!\!\!\!|f(s)|^2_{\rH}\,ds + \sup_{t\in[0,T]}|J_1({t\wedge\tau_{{\eps},M}})| + J_2(T\wedge\tau_{{\eps},M}) + C_N\\
   &\ \ \ +
   \int_0^{t\wedge\tau_{{\eps},M}}\!\!\!\!\|X^{\eps}(s)\|^2_{\rV} \Big(C\|X^{\eps}(s)\|^2_{\rV}|X^{\eps}(s)|^2_{\rH}+4\int_{\rZ}L_2(z)|\phieps(s,z)-1|\nu(dz)\Big)\,ds,\nonumber
   \ \ \ t\in [0,T].
\end{split}
\end{equation}

By (\ref{eq3 star}) again and the definition of $\tau_{{\eps},M}$, Gronwall's lemma implies that
\begin{eqnarray}\label{eq3 prop2 lem2 00}
&&\hspace{-1.0truecm}\lefteqn{ \sup_{t\in [0,T]}\|X^{\eps}({t\wedge\tau_{{\eps},M}})\|^2_{\rV} + \int_0^{t\wedge\tau_{{\eps},M}}\|X^{\eps}(s)\|^2_{\mathcal{D}({\rA})}\,ds} \\
&\leq&
 e^{M^2+C_N}\Big(
        \|u_0\|^2_{\rV}
\!+\! 2\int_0^T|f(s)|^2_{\rH}\,ds \!+\! \sup_{t\in[0,T]}|J_1({t\wedge\tau_{{\eps},M}})| \!+\! J_2(T\wedge\tau_{{\eps},M}) \!+\! C_N
       \Big), \;t\in [0,T].\nonumber
\end{eqnarray}
Similar to (\ref{eq3 sto lem J1}) and (\ref{eq3 sto lem J2}), we can get
\begin{equation}\label{eq3 sto lem2 J1}
 \mathbb{E}\Big(\sup_{t\in[0,T]}|J_1({t\wedge\tau_{{\eps},M}})|\Big)\leq
  C_N{\eps}^{\frac{1}{2}} \mathbb{E}\Big(\sup_{t\in[0,T]}\|X^{\eps}({t\wedge\tau_{{\eps},M}})\|^2_{\rV}\Big)
  +
  C_N{\eps}^{\frac{1}{2}},
\end{equation}
and
\begin{eqnarray}\label{eq3 sto lem2 J2}
\mathbb{E}\Big(|J_2(T\wedge\tau_{{\eps},M})|\Big)
\leq
C_N{\eps} + C_N{\eps} \mathbb{E}\Big(\sup_{t\in[0,T]}\|X^{\eps}({t\wedge\tau_{{\eps},M}})\|^2_{\rV}\Big).
\end{eqnarray}
Substituting (\ref{eq3 sto lem2 J1}) and (\ref{eq3 sto lem2 J2}) into (\ref{eq3 prop2 lem2 00}), and then choosing ${\eps}_{N,M}>0$ small enough, we get
(\ref{eq3 X ep pro 03}).

\end{proof}

\vskip 0.3cm
Our next result is the tightness result.
\begin{lemma}\label{lem LDP stochastic 03}
For every $N \in \mathbb{N}$,
for any fixed subsequence $\{{\eps}_k\}_{k\in\mathbb{N}}$ such that  ${\eps}_k\to0$,
and
for every $\mathcal{U}^N$-valued sequence  $\varphi_{\eps_k}$,
the laws of the sequence
$\{X^{{\eps}_k}\}_{k\in\mathbb{N}}$ are tight on the Hilbert space $L^2([0,T],\rV)$.
\end{lemma}

\begin{proof} Assume that $N \in \mathbb{N}$.
Fix a number $\eta>0$, and choose $M>0$ such that
\[\frac{C_N}{M} < \frac{\eta}{2},\]
where $C_N$ is the constant appearing in Lemma \ref{lem LDP stochastic 01}.
Let ${\eps}_N$ and ${\eps}_{N,M}$ be as in
 Lemmata \ref{lem LDP stochastic 01} and \ref{lem LDP stochastic 02}.

Without loss of generality, we can assume that ${\eps}_k\in(0,{\eps}_N\wedge {\eps}_{N,M})$ for all $k\in \mathbb{N}$.

Choose and fix an auxiliary number $\alpha\in(0,\frac{1}{2})$. Since the imbedding $\mathcal{D}({\rA})\subset \rV$ is compact,  by  \cite[Theorem 2.1]{Flandoli+Gatarek_1995}, the embedding
$$
\Lambda = L^2([0,T],\mathcal{D}({\rA}))\cap W^{\alpha,2}([0,T],\rV^\prime) \embed L^2([0,T],\rV)
$$
is also compact.
Define
$$
\|g\|^2_\Lambda=\int_0^T\|g(t)\|^2_{\mathcal{D}({\rA})}\,dt + \|g\|^2_{W^{\alpha,2}([0,T],\rV^\prime)},\ \ \ \ g\in\Lambda.
$$
Choose $R>0$ such that
\[
\frac{2C_{N,M}+2C_{\alpha,N}}{R^2}<\frac{\eta}{2},
\]
where the constants $C_{\alpha,N}$ and $C_{N,M}$
are those that appear in Lemmata \ref{lem LDP stochastic 01} and \ref{lem LDP stochastic 02}.

Since the set
$$
K_R=\{g\in \Lambda,\ \|g\|_\Lambda\leq R\}
$$
is relatively compact in $L^2([0,T],\rV)$,
 it is sufficient to show that
\[
\Q(X^{{\eps}_k}\not\in K_R) < \eta, \mbox{ for all } k.
\]
To do so, by Lemma \ref{lem LDP stochastic 01} and Corollaries \ref{cor-Cheb} and \ref{cor-Cheb-2} \dela{and Lemma \ref{lem LDP stochastic 02}}, we infer that
\begin{eqnarray*}
\Q(X^{{\eps}_k}\not\in K_R)
&\leq&
\Q(\tau_{{\eps}_k,M}< T) + \Q\Big((\tau_{{\eps}_k,M} \geq T)\cap (X^{{\eps}_k}\not\in K_R)\Big)\nonumber\\
&\leq&
\frac{C_N}{M} + \frac{2C_{N,M}+2C_{\alpha,N}}{R^2}< \eta.
\end{eqnarray*}

\dela{Hence for any $\eta>0$, there exists $M_0,R_0>0$ such that
\begin{eqnarray}\label{eq3 lem 3 00}
\Q(X^{{\eps}_k}\not\in K_{R_0})\leq\eta,\ \ \ \forall {\eps}_k\in(0,{\eps}_{N,M_0}).
\end{eqnarray}

\vskip 0.3cm
Because for any fixed ${\eps}_k$,
$$
1=\Q(X^{{\eps}_k}\in \Lambda)=\lim_{R\to\infty}\Q(X^{{\eps}_k}\in K_R),
$$
one can find $R_1$ such that
\begin{eqnarray}\label{eq3 lem 3 01}
\Q(X^{{\eps}_k}\not\in K_{R_1})\leq\eta,\ \ \ \forall {\eps}_k\not\in(0,{\eps}_{N,M_0}).
\end{eqnarray}
We have used the fact that there are only finite elements of $\{{\eps}_k\}_{k\in\mathbb{N}}$ such that ${\eps}_k\not\in (0,{\eps}_{N,M_0})$.

\vskip 0.2cm
(\ref{eq3 lem 3 00}) and (\ref{eq3 lem 3 01}) imply that for any fixed subsequence $\{{\eps}_k\}_{k\in\mathbb{N}}$, ${\eps}_k\in(0,{\eps}_N)$ and ${\eps}_k\to0$,
$\{X^{{\eps}_k}\}_{k\in\mathbb{N}}$ is tight in $L^2([0,T],V)$.
}
\vskip 0.2cm

The proof is complete.

\end{proof}

Using arguments similar to those proving  \cite[Lemma 4.5]{Zhai-Zhang}, we get
\begin{lemma}\label{lem LDP stochastic 04}
There exists $\varrho>1$ such that
for every $N \in \mathbb{N}$,
for any fixed subsequence $\{{\eps}_k\}_{k\in\mathbb{N}}$ such that  ${\eps}_k\to0$,
and
for every $\mathcal{U}^N$-valued sequence  $\varphi_{\eps_k}$,
the laws of the sequence
$\{X^{{\eps}_k}\}_{k\in\mathbb{N}}$ are tight on the
Skorokhod space $D([0,T],\mathcal{D}({\rA}^{-\varrho}))$.
\end{lemma}

\vskip 0.3cm

Next, consider a family $\varphi_{\eps}$, $\eps\in (0,1]$, of $\mathcal{U}^N$-valued processes, for some fixed $N\in\mathbb{N}$.
For each $\eps$, let $Y^{\eps}$ be the unique solution of the following (auxiliary) stochastic Langevin equation:
\begin{equation}\label{eqn-Y^eps}
Y^{\eps}(t)=\int_0^t{\rA}Y^{\eps}(s)\,ds + {\eps} \int_0^t\int_{{\rZ}}G(X^{\eps}(s-),z)\widetilde{N}^{{\eps}^{-1}\phiepsold}(dz,ds).
\end{equation}

In this situation, we have  the following.
\begin{lemma}\label{lem LDP stochastic 05}
In the above framework, if $\eta>0$, then
\begin{equation}\label{eqn-limit}
\lim_{{\eps}\to 0}\Q\Big(\sup_{t\in[0,T]}\|Y^{\eps}(t)\|^2_{\rV}+\int_0^T\|Y^{\eps}(s)\|^2_{\mathcal{D}({\rA})}\,ds\geq\eta\Big)
=
0.
\end{equation}

\end{lemma}

\begin{proof} Fix $\eta>0$.
Suppose that we have proved that for every  $M>0$,
\begin{eqnarray}\label{eq3 lem 5 Y 00}
&&\hspace{-1truecm}\mathbb{E}\Big(\sup_{t\in[0,T\wedge\tau_{{\eps},M}]}\|Y^{\eps}(t)\|^2_{\rV}
+
\int_0^{T\wedge\tau_{{\eps},M}}\|Y^{\eps}(s)\|^2_{\mathcal{D}({\rA})}\,ds\Big)
\leq
 {\eps} C_N C_{N,M},\; {\eps}\in(0,{\eps}_{N,M}),
\end{eqnarray}
where the stopping time $\tau_{{\eps},M}$ is defined in \eqref{eqn-tau_eps,M} and
$C_{N,M}$ and ${\eps}_{N,M}$ are as they appeared in Lemma \ref{lem LDP stochastic 02}.

Then we can conclude the proof of the lemma as follows.
\vskip 0.3cm
First, we set
$$
\Lambda_\eta:=\Big\{\omega\in\Omega:\ \sup_{t\in[0,T]}\|Y^{\eps}(t)\|^2_{\rV}+\int_0^{T}\|Y^{\eps}(s)\|^2_{\mathcal{D}({\rA})}\,ds\geq \eta\Big\}.
$$
Hence, by Lemma \ref{lem LDP stochastic 01} and (\ref{eq3 lem 5 Y 00}), we infer that
for all $M>0, {\eps}\in(0,{\eps}_{N,M})$
\begin{eqnarray*}
 \Q(\Lambda_\eta)
&\leq&
   \Q(\tau_{{\eps},M}< T)
   +
   \Q\Big((\tau_{{\eps},M} \geq T)\cap \Lambda_\eta\Big)\nonumber\\
&\leq&
   \frac{C_N}{M}+\frac{{\eps} C_N C_{N,M}}{\eta} .
\end{eqnarray*}
which, by a standard argument, implies \eqref{eqn-limit}.
\dela{\begin{eqnarray*}
\lim_{{\eps}\to 0}\Q\Big(\sup_{t\in[0,T]}\|Y^{\eps}(t)\|^2_{\rV}+\int_0^T\|Y^{\eps}(s)\|^2_{\mathcal{D}({\rA})}\,ds\geq\eta\Big)
=
0,\ \ \forall \eta>0.
\end{eqnarray*}}

Thus, we only have to show inequality \eqref{eq3 lem 5 Y 00}.
Let us fix $M>0$.
 By {the It\^o} formula,
\begin{eqnarray*}
 &&\hspace{-1truecm}\|Y^{\eps}(t)\|^2_{\rV} + 2\int_0^t\|Y^{\eps}(s)\|^2_{\mathcal{D}({\rA})}\,ds\nonumber\\
&=&
  2{\eps} \int_0^t\int_{{\rZ}}\langle G(X^{\eps}(s-),z),Y^{\eps}(s-)\rangle_{\rV}\widetilde{N}^{{\eps}^{-1}\phiepsold}(dz,ds)
 \nonumber \\ && +
  {\eps}^2\int_0^t\int_{{\rZ}}\|G(X^{\eps}(s-),z)\|^2_{\rV}N^{{\eps}^{-1}\phiepsold}(dz,ds).
\end{eqnarray*}

By Assumption \ref{con LDP} and (\ref{eq3 star3}),
\begin{eqnarray*}
&&\hspace{-1.5truecm}\lefteqn{ {\eps}^2\mathbb{E}\Big(
    \int_0^{T\wedge\tau_{{\eps},M}}\int_{{\rZ}}\|G(X^{\eps}(s-),z)\|^2_{\rV}N^{{\eps}^{-1}\phiepsold}(dz,ds)
        \Big)}\nonumber\\
&\leq&
  {\eps}\mathbb{E}\Big(
     \int_0^{T\wedge\tau_{{\eps},M}}\int_{{\rZ}}\|G(X^{\eps}(s),z)\|^2_{\rV}\phiepsold(s,z)\nu(dz)\,ds
     \Big)\nonumber\\
&\leq&
  2{\eps}\mathbb{E}\Big(
     \int_0^{T\wedge\tau_{{\eps},M}}\int_{{\rZ}}(1+\|X^{\eps}(s)\|^2_{\rV})L_2^2(z)\phiepsold(s,z)\nu(dz)\,ds
     \Big)\nonumber\\
&\leq&
 {\eps} C_N\mathbb{E}\Big(\sup_{t\in[0,T\wedge\tau_{{\eps},M}]}\|X^{\eps}(t)\|^2_{\rV}\Big)
 +
 {\eps} C_N.
\end{eqnarray*}
Applying the Burkholder-Davis-Gundy inequality and (\ref{eq3 star3}) again,

\begin{eqnarray*}
&&\hspace{-1.0truecm}\lefteqn{
 2{\eps}\mathbb{E}\Big(\sup_{t\in[0,T\wedge\tau_{{\eps},M}]}
              \Big|\int_0^t\int_{{\rZ}}\langle G(X^{\eps}(s-),z),Y^{\eps}(s-)\rangle_{\rV}\widetilde{N}^{{\eps}^{-1}\phieps}(dz,ds)\Big|
            \Big)
            }
            \nonumber\\
&\leq&
{\eps} C\mathbb{E}\Big(
              \Big|\int_0^{T\wedge\tau_{{\eps},M}}\int_{{\rZ}} \|G(X^{\eps}(s-),z)\|^2_{\rV}\|Y^{\eps}(s-)\|^2_{\rV}N^{{\eps}^{-1}\phieps}(dz,ds)\Big|^{\frac{1}{2}}
            \Big)\nonumber\\
&\leq&
 \frac{1}{2}\mathbb{E}\Big(\sup_{t\in[0,T\wedge\tau_{{\eps},M}]}\|Y^{\eps}(t)\|^2_{\rV}\Big)
 +
 {\eps} C \mathbb{E}\Big(
     \int_0^{T\wedge\tau_{{\eps},M}}\int_{{\rZ}}\|G(X^{\eps}(s),z)\|^2_{\rV}\phieps(s,z)\nu(dz)\,ds
     \Big)\nonumber\\
&\leq&
   \frac{1}{2}\mathbb{E}\Big(\sup_{t\in[0,T\wedge\tau_{{\eps},M}]}\|Y^{\eps}(t)\|^2_{\rV}\Big)
 +
   {\eps} C_N\mathbb{E}\Big(\sup_{t\in[0,T\wedge\tau_{{\eps},M}]}\|X^{\eps}(t)\|^2_{\rV}\Big)
 +
 {\eps} C_N.
\end{eqnarray*}

Combining the above three estimates and Lemma \ref{lem LDP stochastic 02}, we deduce inequality \eqref{eq3 lem 5 Y 00}.

\vskip 0.2cm
The proof is complete.

\end{proof}

\subsection{The second continuity lemma}
\label{subsec-2nd continuity lemma}

This subsection is devoted to proving \textbf{Claim-LDP-2}.
That is,
we are now in position to give the proof of \textbf{Claim-LDP-2} formulated in the proof of Theorem \ref{thm-LDP}.
To proceed, we prove the following result.

\begin{proposition}[The second continuity lemma]\label{prop3 02}
Let $\varphi_{{\eps}_n}$, $\varphi\in \mathcal{U}^N$ be such that $\varphi_{{\eps}_n}$ converges in law to $\varphi$ as ${{\eps}_n}\to0$.
Then the sequence of processes
$$
\mathcal{G}^{\epsilon_n}({{\eps}_n} N^{{{\eps}_n}^{-1}\varphi_{{\eps}_n}})$$ converges in law on $\UT$ to a process $$\mathcal{G}^0(\varphi).
$$
\end{proposition}

\begin{proof}[Proof of Proposition \ref{prop3 02}]
Fix a natural number $N$,
a sequence ${{\eps}_n}$ such that ${{\eps}_n}\to0$,
and
a $\mathcal{U}^N$-valued sequence $\bigl\{\varphi_{{\eps}_n}\bigr\}_{n\in\mathbb{N}}$,
such that  $\varphi_{{\eps}_n}$ converges in law to $\varphi$ for some $\varphi\in  \mathcal{U}^N$.

\medskip

By Lemma \ref{lem LDP stochastic 00}, the process \[X^{{\eps}_n}:=\mathcal{G}^{\epsilon_n}({{\eps}_n} N^{{{\eps}_n}^{-1}\varphi_{{\eps}_n}})\]
 is the unique solution of problem (\ref{eq3 prop2 00}) with ${\eps}$ and $\varphi$ replaced by ${\eps}_n$ and $\varphi_{{\eps}_n}$ respectively.
 Recall that a process $Y^{\eps}$, for $\eps>0$, was defined in formula \eqref{eqn-Y^eps}.

By Lemmata \ref{lem LDP stochastic 03}, \ref{lem LDP stochastic 04}, and \ref{lem LDP stochastic 05},
\begin{itemize}
\item[(1)] the laws of the processes $\{X^{{\eps}_n}\}_{n\in\mathbb{N}}$ are tight on $L^2([0,T],\rV)\cap D([0,T],\mathcal{D}({\rA}^{-\varrho}))$;

\item[(2)] the sequence $\{Y^{{\eps}_n}\}_{n\in\mathbb{N}}$ converges in probability to $0$ in $\UT$.
\end{itemize}

Set
$$
\Gamma_T=\Big[L^2([0,T],\rV)\cap D([0,T],\mathcal{D}({\rA}^{-\varrho}))\Big]\otimes \UT\otimes S^N.
$$
Let $(X,0,\varphi)$ be any limit point of the tight family $\{(X^{{\eps}_n},Y^{{\eps}_n},\varphi_{{\eps}_n}),\ n\in\mathbb{N}\}$.
By the Skorokhod representation theorem, there exists a stochastic basis $(\Omega^1,\mathbb{F}^1,\mathbb{P}^1)$ and, on this basis,
$\Gamma_T$-valued random variables $(X_1,0,\varphi_1)$, $(X^{n}_1,Y^{n}_1,\varphi_1^n),\ n\in\mathbb{N}$ such that
\begin{itemize}
\item[(a)] $(X_1,0,\varphi_1)$ has the same law as $(X,0,\varphi)$;

\item[(b)] for any $n\in\mathbb{N}$, $(X^{n}_1,Y^{n}_1,\varphi_1^n)$ has the same law as $(X^{{\eps}_n},Y^{{\eps}_n},\varphi_{{\eps}_n})$;

\item[(c)] $\lim_{n\to\infty}(X^{n}_1,Y^{n}_1,\varphi_1^n)=(X_1,0,\varphi_1)$ in $\Gamma_T$, $\mathbb{P}^1$-a.s..
\end{itemize}

Because equations \eqref{eqn-C3} and \eqref{eqn-Y^eps} are satisfied by processes $(X^{{\eps}_n},Y^{{\eps}_n},\varphi_{{\eps}_n})$, we infer that $(X^{n}_1,Y^{n}_1,\varphi_1^n)$ satisfies
\begin{eqnarray}\label{query-01}
 &&\hspace{-0.4truecm}X^n_1(t)-Y_1^n(t)\nonumber\\
&&=
 u_0-\int_0^t{\rA}(X^n_1(s)-Y_1^n(s))\,ds
 \int_0^t\rB(X^n_1(s))\,ds + \int_0^tf(s)\,ds\nonumber\\
 &&\ \ +
 \int_0^t\int_{{\rZ}}G(X^n_1(s),z)(\varphi^n_1(s,z)-1)\nu(dz)\,ds,\ t\in[0,T].
\end{eqnarray}
Hence, by deterministic results and (b), we conclude that
\begin{eqnarray*}
&&\hspace{-1truecm}\mathbb{P}^1\Big(X^n_1-Y_1^n\in C([0,T],\rV)\cap L^2([0,T],\mathcal{D}({\rA}))\Big)\nonumber\\
&=&
 \Q\Big(X^{{\eps}_n}-Y^{{\eps}_n}\in C([0,T],\rV)\cap L^2([0,T],\mathcal{D}({\rA}))\Big)
=1.
\end{eqnarray*}

Since
\begin{eqnarray}\label{eq LDP Y}
\lim_{n\to 0}\Big(\sup_{t\in[0,T]}\|Y^n_1(t)\|^2_{\rV}+\int_0^T\|Y^n_1(s)\|^2_{\mathcal{D}({\rA})}\,ds\Big)=0,\ \ \mathbb{P}^1\text{-a.s.,}
\end{eqnarray}
by applying arguments similar to the proof of Proposition \ref{prop-1st continuity lemma}, we can show that $(X_1,\varphi_1)$ satisfies
\begin{eqnarray*}
X_1(t)&\!=\!&u_0-\int_0^t{\rA}X_1(s)\,ds - \int_0^t\rB(X_1(s))\,ds+ \int_0^tf(s)\,ds
\\
&&+ \int_0^t\int_{{\rZ}}G(X_1(s),z)(\varphi_1(s,z)-1)\nu(dz)\,ds.
\end{eqnarray*}
The maximal regularity property of the solutions to the
deterministic 2D Navier-Stokes equations,
taking into account Lemma \ref{lem3 LDP deter 1'},
imply that $\mathbb{P}^1$-a.s.,
 \[X_1\in C([0,T],\rV)\cap L^2([0,T],\mathcal{D}({\rA})).\]

By (\ref{eq LDP Y}) and (c), using arguments as in the proof of (\ref{eq LDP claim 1}), we get
\begin{eqnarray*}
\lim_{n\to 0}\Big(\sup_{t\in[0,T]}\|X^n_1(t)-X_1(t)\|^2_{\rV}+\int_0^T\|X^n_1(s)-X_1(s)\|^2_{\mathcal{D}({\rA})}\,ds\Big)=0,\ \ \mathbb{P}^1\text{-a.s..}
\end{eqnarray*}

Hence, by (\ref{eq3 define G0}), which is the definition of $\mathcal{G}^0$,
$$
X^{{\eps}_n}\text{ converges in law to }\mathcal{G}^0(\varphi),
$$
which implies the desired result.

\vskip 0.2cm
The proof of Proposition \ref{prop3 02} is complete.
\end{proof}

\begin{appendices}

\section{Poisson random measures}\label{sec-A-PRM}

Recall the following definition, which is taken from \cite[Definition I.8.1]{Ikeda-Watanabe}; see also \cite{BHZ}.
\begin{definition}\label{def-PRM-Y}
A time-homogeneous Poisson random measure on $\rY=\rZ\times [0,\infty)$ (i.e., a Poisson random measure on $\rY_T=[0,T] \times \rZ\times [0,\infty)$)
over probability space $(\bar{\Omega},\G,\Gb,\Q)$  with the intensity measure $\Leb_{[0,T]}\otimes \nu\otimes~\Leb_{[0,\infty)}$, is a measurable function  \[\eta: (\bar{\Omega},\mathcal{G})\to \mathcal{M}(\rY_T)=\MTb \]
satisfying the following conditions
\begin{enumerate}
\item \label{prm-i} for each $U\in \mathcal{B}([0,T]) \otimes \mathcal{B}(\rY) $,
 $\eta(U):=i_U\circ \eta : \bar{\Omega}\to \overline{\mathbb{N}}$ \footnote{$\overline{\mathbb{N}}:=\mathbb{N}\cup\{0\}\cup\{\infty\}$} is a Poisson random variable with parameter\footnote{If $\mathbb{E} \eta(U) = \infty$, then obviously $\eta(U)=\infty$ a.s..} $\mathbb{E}\eta(U)$;
\item \label{prm-ii} $\eta$ is independently scattered, \textit{i.e}., if the sets
$ U_j \in  \mathcal{B}([0,T]) \otimes \mathcal{B}(\rY)$, $j=1,\cdots, n$ are disjoint,  then the random
variables $\eta(U_j)$, $j=1,\cdots,n $ are mutually independent;
\item \label{prm-3} for all $U\in \mathcal{B}(\rY) $ and $I\in \mathcal{B}([0,T])$,
\[\mathbb{E}\big[\eta (I\times U)\big]=(\Leb_{[0,T]}\otimes~\nu \otimes \Leb_{[0,\infty)}) (I\times U)=\Leb_{[0,T]}(I)(\nu\otimes \Leb_{[0,\infty)})(U);\]
\item \label{prm-4} for each $U\in \mathcal{B}(\rY) $, the $\overline{\mathbb{N}}$-valued process
$$ (0,\infty) \times \Omega \ni (t,\omega) \mapsto \eta(\omega)\bigl(U \times (0,t]\bigr)$$
is $\mathbb{G}$-adapted, and its increments are independent of the past, \textit{i.e}., the increment between times $t$ and $s$, $t>s> 0$, are independent of the $\sigma$-field $\mathcal{G}_s$.
\end{enumerate}
\end{definition}
Similarly, we have

\begin{definition}\label{def-PRM-Z}
A time-homogeneous Poisson random measure on $\rZ$ (or Poisson random measure on $\rZ_T=[0,T] \times \rZ$)
over probability space $(\bar{\Omega},\G,\Gb,\Q)$  with the intensity measure $\Leb_{[0,T]} \otimes~\nu$, is a measurable function  \[\eta: (\bar{\Omega},\mathcal{G})\to \mathcal{M}(\rZ_T)=\MT \]
satisfying the following conditions
\begin{enumerate}
\item \label{prm-i} for each $U\in \mathcal{B}([0,T]) \otimes \mathcal{B}(\rZ) $,
 $\eta(U):=i_U\circ \eta : \bar{\Omega}\to \overline{\mathbb{N}} $ is a Poisson random variable with parameter $\mathbb{E}\eta(U)$;
\item \label{prm-ii} $\eta$ is independently scattered, \textit{i.e}., if the sets
$ U_j \in  \mathcal{B}([0,T]) \otimes \mathcal{B}(\rZ)$, $j=1,\cdots, n$ are disjoint,  then the random
variables $\eta(U_j)$, $j=1,\cdots,n $ are mutually independent;
\item \label{prm-3} for all $U\in \mathcal{B}(\rZ) $ and $I\in \mathcal{B}([0,T])$,
\[\mathbb{E}\big[\eta (I\times U)\big]=(\Leb_{[0,T]}\otimes~\nu) (I\times U)=\Leb_{[0,T]}(I)\nu(U);\]
\item \label{prm-4} for each $U\in \mathcal{B}(\rZ) $, the $\overline{\mathbb{N}}$-valued process
$$ (0,\infty) \times \Omega \ni (t,\omega) \mapsto \eta(\omega)\bigl(U \times (0,t]\bigr)$$
is $\mathbb{G}$-adapted, and its increments are independent of the past, \textit{i.e}., the increment between times $t$ and $s$, $t>s> 0$, are independent of  the $\sigma$-field $\mathcal{G}_s$.
\end{enumerate}
\end{definition}

\section{Proof of Lemma \ref{lem3 LDP deter 1}}\label{sec-B}

\renewcommand{\theequation}{B.\arabic{equation}}

This section is devoted to the proof of Lemma \ref{lem3 LDP deter 1} which, for the convenience of the reader, we state again.

\begin{lemma}\label{lem3 LDP deter 1'}
Assume that $N\in\mathbb{N}$. Then, for all $u_0\in \rV$, $f\in L^2([0,T],\rH)$, and $g\in S^N$, there exists a unique  solution $u^g \in C([0,T],\rV)\cap L^2([0,T],\mathcal{D}({\rA}))$ of Problem (\ref{eq3 LDP deter 00}). Moreover, for any $\rho>0$ and $R>0$,
there exists a positive constant $C_N=C_{N,\rho,R}$ such that for every $g \in S^N$ and all $u_0\in \rV$ and $f\in L^2([0,T],\rH)$, such that
$\Vert u_0\Vert_\rV \leq \rho $ and $\vert f \vert_{L^2([0,T];\rH)} \leq R$,
the following estimate is satisfied
\begin{eqnarray}\label{eq3 LDP deter pro}
             \sup_{t\in[0,T]}\|u^g(t)\|^2_{\rV}+\int_0^T\|u^g(t)\|^2_{\mathcal{D}({\rA})}\,dt\leq C_N.
\end{eqnarray}
\end{lemma}

\begin{proof}[Proof of Lemma \ref{lem3 LDP deter 1}]
Fix $N\in\mathbb{N}$, $u_0\in \rV$, $f\in L^2([0,T],\rH)$, and $g\in S^N$.
Define an auxiliary function
$$
F(t,y):=\int_{{\rZ}}G(y,z)(g(t,z)-1)\nu(dz), \;\; t \in [0,T],\; y\in \rV.
$$
 {By \textbf{Assumption \ref{con LDP}}, \addjzok{for every $t\in[0,T]$, $\hbar>0$,
and $y,y_1,\,y_2\in\rV$ with $\|y_1\|_\rV\vee\|y_2\|_\rV\leq \hbar$,}
\begin{eqnarray*}
\|F(t,y_1)-F(t,y_2)\|_{\rV} \leq \int_{\rZ}L_\hbar(z)|g(t,z)-1|\nu(dz)\|y_1-y_2\|_{\rV},
\end{eqnarray*}}
\deljz{for every $y\in\rV$,}
\begin{eqnarray*}
\|F(t,y)\|_{\rV} \leq \int_{\rZ}L_2(z)|g(t,z)-1|\nu(dz)(1+\|y\|_{\rV}),
\end{eqnarray*}
and \deljz{for every $y\in\rH$,}
\begin{eqnarray*}
\|F(t,y)\|_{\rH} \leq \int_{\rZ}L_3(z)|g(t,z)-1|\nu(dz)(1+\|y\|_{\rH}).
\end{eqnarray*}

By Lemma 3.4 in \cite{Budhiraja-Chen-Dupuis},
\begin{eqnarray}\label{eq3 star}
\dela{C_{i,N}}
C_{1,N}:=\max_{i=2,3}\sup_{g\in S^N}\int_0^T\int_{{\rZ}}L_i(z)|g(t,z)-1|\nu(dz)\,dt<\infty,
\end{eqnarray}
 {\addjzok{and, for every $\hbar>0$,
\begin{eqnarray}\label{eq3 star local}
C_{\hbar,N}:=\sup_{g\in S^N}\int_0^T\int_{{\rZ}}L_\hbar(z)|g(t,z)-1|\nu(dz)\,dt<\infty.
\end{eqnarray}}}

Combining the above five inequalities and using a argument similar to those for Theorems \ref{thm-initial data in V} {and \ref{thm-initial data in V local}}, we can deduce
 that there exists a unique solution $u^g \in C([0,T],\rV)\cap L^2([0,T],\mathcal{D}({\rA}))$ of Equation (\ref{eq3 LDP deter 00}).

Now we are ready to prove (\ref{eq3 LDP deter pro}). We begin with \emph{a priori} estimates in the space $\rH$.
\addjzok{Note that we only use the assumption ``Linear growth in $\rV$" and ``Linear growth in $\rH$" in \textbf{Assumption \ref{con LDP}} to get (\ref{eq3 LDP deter pro}).} By \textbf{Assumption \ref{con LDP}} and the Lions-Magenes lemma,
we have
\begin{eqnarray*}
&&\hspace{-1truecm}\lefteqn{|u^g(t)|^2_{\rH} + 2\int_0^t\|u^g(s)\|^2_{\rV}\,ds}\nonumber\\
&=&\!\!
  |u_0|^2_{\rH} + 2\int_0^t\langle f(s),u^g(s)\rangle_{\rH}\,ds
 +
  2\int_0^t\int_{{\rZ}}\langle G(u^g(s),z),u^g(s)\rangle_{\rH}(g(s,z)-1)\nu(dz)\,ds\nonumber\\
&\leq&\!\!
  |u_0|^2_{\rH}\!+\!\int_0^t\!\|u^g(s)\|^2_{\rV}\,ds\!+\!\int_0^t\!\|f(s)\|^2_{\rV^\prime}\,ds
\!+\!
  2\int_0^t\!(1\!+\!2|u^g(s)|^2_{\rH})\int_{{\rZ}}\!L_3(z)|g(s,z)\!-\!1|\nu(dz)\,ds.
\end{eqnarray*}
Hence,
\begin{eqnarray}\label{eq3 lem1 00}
&&\hspace{-1truecm}\lefteqn{|u^g(t)|^2_{\rH} + \int_0^t\|u^g(s)\|^2_{\rV}\,ds}\nonumber\\
&\leq&
 |u_0|^2_{\rH} + \int_0^T\|f(s)\|^2_{\rV^\prime}\,ds + 2\int_0^T\int_{{\rZ}}L_3(z)|g(s,z)-1|\nu(dz)\,ds\nonumber\\
 &&+
 4\,\int_0^t|u^g(s)|^2_{\rH}\int_{{\rZ}}L_3(z)|g(s,z)-1|\nu(dz)\,ds.
\end{eqnarray}

By applying Gronwall's lemma, we get
\begin{eqnarray*}
&&\hspace{-1truecm}\lefteqn{\sup_{t\in[0,T]}|u^g(t)|^2_{\rH} + \int_0^T\|u^g(t)\|^2_{\rV}\,dt}
\nonumber\\
&\leq&
  \Big(
   |u_0|^2_{\rH} + \int_0^T\|f(s)\|^2_{\rV^\prime}\,ds + 2\int_0^T\int_{{\rZ}}L_3(z)|g(s,z)-1|\nu(dz)\,ds
  \Big)\\
  &&\hspace{4truecm}\lefteqn{\cdot
   e^{4\int_0^T\int_{{\rZ}}L_3(z)|g(s,z)-1|\nu(dz)\,ds}.}
\end{eqnarray*}

Applying (\ref{eq3 star}), we get
\begin{eqnarray}\label{eq3 star2}
\sup_{g\in S^N}\Big(
               \sup_{t\in[0,T]}|u^g(t)|^2_{\rH} + \int_0^T\|u^g(t)\|^2_{\rV}\,dt
              \Big) \leq K_{N,H}    <\infty.
\end{eqnarray}

Now, consider that, by Assumption \ref{con LDP} and \cite[Lemma III.1.2]{Temam_2001},
in the space $\rV$, we have
\begin{eqnarray}\label{eq3 deter V 00}
&&\hspace{-1truecm}\lefteqn{\|u^g(t)\|^2_{\rV} + 2\int_0^t\|u^g(s)\|^2_{\mathcal{D}({\rA})}\,ds}\nonumber\\
&=&
 \|u_0\|^2_{\rV} + 2\int_0^t\langle \rB(u^g(s)),{\rA}u^g(s)\rangle_{\rH}\,ds
 +
 2\int_0^t\langle f(s),{\rA}u^g(s)\rangle_{\rH}\,ds\nonumber\\
 &&+
 2\int_0^t\int_{{\rZ}}\langle G(u^g(s),z),u^g(s)\rangle_{\rV}(g(s,z)-1)\nu(dz)ds\nonumber\\
&\leq&
 \|u_0\|^2_{\rV} + \int_0^t\|u^g(s)\|^2_{\mathcal{D}({\rA})}\,ds + C\int_0^t\|u^g(s)\|^4_{\rV}|u^g(s)|^2_{\rH}\,ds
 +
 2\int_0^t|f(s)|^2_{\rH}\,ds \nonumber\\
 &&+
 2\int_0^t(1+2\|u^g(s)\|^2_{\rV})\int_{{\rZ}}L_2(z)|g(s,z)-1|\nu(dz)\,ds.
\end{eqnarray}
Hence, we find that
\begin{eqnarray*}
&&\hspace{-1truecm}\lefteqn{\|u^g(t)\|^2_{\rV} + \int_0^t\|u^g(s)\|^2_{\mathcal{D}({\rA})}\,ds}\nonumber\\
&\leq&
  \|u_0\|^2_{\rV} + 2\int_0^T|f(s)|^2_{\rH}\,ds
  +
  2\int_0^T\int_{\rZ}L_2(z)|g(s,z)-1|\nu(dz)\,ds\nonumber\\
  &&+
  \int_0^t\|u^g(s)\|^2_{\rV}\Big(
                      C\|u^g(s)\|^2_{\rV}|u^g(s)|^2_{\rH}+4\int_{\rZ}L_2(z)|g(s,z)-1|\nu(dz)
                     \Big)
          ds.
\end{eqnarray*}
Therefore, by applying Gronwall's lemma, we deduce that
\begin{eqnarray*}
&&\hspace{-1.5truecm}\lefteqn{\sup_{t\in[0,T]}\|u^g(t)\|^2_{\rV} + \int_0^T\|u^g(s)\|^2_{\mathcal{D}({\rA})}\,ds}\nonumber\\
&\leq&
   \Big(
    \|u_0\|^2_{\rV} + 2\int_0^T|f(s)|^2_{\rH}\,ds + 2\int_0^T\int_{\rZ}L_2(z)|g(s,z)-1|\nu(dz)\,ds
   \Big)\\
   &&\cdot
   e^{C\int_0^T\|u^g(s)\|^2_{\rV}|u^g(s)|^2_{\rH}\,ds + 4\int_0^T\int_{\rZ}L_2(z)|g(s,z)-1|\nu(dz)\,ds}.
\end{eqnarray*}
Thus, in view of (\ref{eq3 star}) and (\ref{eq3 star2}),
we know that
\begin{eqnarray*}
\sup_{g\in S^N}\Big(\sup_{t\in[0,T]}\|u^g(t)\|^2_{\rV} + \int_0^T\|u^g(s)\|^2_{\mathcal{D}({\rA})}\,ds\Big)\leq K_{N,\rV}
<
\infty.
\end{eqnarray*}
This completes the proof of Lemma \ref{lem3 LDP deter 1}
\end{proof}

\end{appendices}

\end{document}